\documentclass[a4paper,10pt,reqno]{amsart}
\usepackage{amsmath,amsfonts,amscd,amssymb,graphicx,mathrsfs,eufrak,dsfont}
\usepackage[dvips,all,arc,curve,color,frame]{xy}
\usepackage[usenames]{color}
\usepackage{setspace}

\usepackage[colorlinks]{hyperref}
\usepackage{tikz,mathrsfs}
\usetikzlibrary{arrows,decorations.pathmorphing,decorations.pathreplacing,positioning,shapes.geometric,shapes.misc,decorations.markings,decorations.fractals,calc,patterns}

\tikzset{>=stealth',
    cvertex/.style={circle,draw=black,inner sep=1pt,outer sep=3pt},
    vertex/.style={circle,fill=black,inner sep=1pt,outer sep=3pt},
    star/.style={circle,fill=yellow,inner sep=0.75pt,outer sep=0.75pt},
    tvertex/.style={inner sep=1pt,font=\criptsize},
    gap/.style={inner sep=0.5pt,fill=white}}

\newcommand{\marginparstretch}{0.6}
\let\oldmarginpar\marginpar
\renewcommand\marginpar[1]{\-\oldmarginpar[\framebox{\setstretch{\marginparstretch}\begin{minipage}{\marginparwidth}{\raggedleft\tiny #1}\end{minipage}}]{\framebox{\setstretch{\marginparstretch}\begin{minipage}{\marginparwidth}{\raggedright\tiny #1}\end{minipage}}}}

\newcommand{\arrowrl}[3][20]
{
\hspace{-5pt}
\begin{tikzpicture}
\node (A) at (0,0) {};
\node (B) at (1,0) {};
\draw[->] ($(A)+(0,0.2)$) -- node [above] {$\scriptstyle f^*$} ($(B)+(0,0.2)$);
\draw [->] ($(B)+(0,0.2)$) -- node [below] {$\scriptstyle f_*$} ($(A)+(0,0.2)$);
\end{tikzpicture}
\hspace{-5pt}
}
\newcommand{\adj}[2][20]{\arrowrl}

\newtheorem{thm}{Theorem}[section]
\newtheorem{prop}[thm]{Proposition}
\newtheorem{lemma}[thm]{Lemma}
\newtheorem{cor}[thm]{Corollary}

\theoremstyle{definition} 
\newtheorem{defin}[thm]{Definition}
 
\newtheorem{example}[thm]{Example}

\newtheorem{remark}[thm]{Remark}
\newtheorem{conj}[thm]{Conjecture}

\numberwithin{equation}{section}

\newcommand{\F}{\mathcal{F}}
\newcommand{\G}{\mathcal{G}}

\newcommand{\C}[1]{\mathbb{C}^{#1}}
\newcommand{\m}{\mathfrak{m}}

\newcommand{\p}{\mathfrak{p}}

\renewcommand{\c}[1]{\mathcal{#1}}
\renewcommand{\u}[1]{\underline{#1}}

\renewcommand{\t}[1]{\textnormal{#1}}

\newcommand{\looptop}[2]{\xy \SelectTips{cm}{10}
\POS(0,0) \endxy}

\def\RHom{\mathop{\rm {\bf R}Hom}\nolimits}
\def\Cl{\mathop{\rm Cl}\nolimits}

\def\op{\mathop{\rm op}\nolimits}
\def\GL{\mathop{\rm GL}\nolimits}
\def\CM{\mathop{\rm CM}\nolimits}
\def\uCM{\mathop{\underline{\rm CM}}\nolimits}

\def\depth{\mathop{\rm depth}\nolimits}
\def\fl{\mathop{\sf fl}\nolimits}

\def\mod{\mathop{\rm mod}\nolimits}

\def\coh{\mathop{\rm coh}\nolimits}
\def\Mod{\mathop{\rm Mod}\nolimits}

\def\refl{\mathop{\rm ref}\nolimits}
\def\proj{\mathop{\rm proj}\nolimits}

\def\pd{\mathop{\rm proj.dim}\nolimits}

\def\Hom{\mathop{\rm Hom}\nolimits}
\def\End{\mathop{\rm End}\nolimits}
\def\Ext{\mathop{\rm Ext}\nolimits}

\def\Tr{\mathop{\rm Tr}\nolimits}
\def\add{\mathop{\rm add}\nolimits}
\def\rk{\mathop{\rm rk}\nolimits}
\def\Cok{\mathop{\rm Cok}\nolimits}
\def\Ker{\mathop{\rm Ker}\nolimits}
\def\ker{\mathop{\rm ker}\nolimits}
\def\Im{\mathop{\rm Im}\nolimits}
\def\length{\mathop{\rm length}\nolimits}

\def\Sing{\mathop{\rm Sing}\nolimits}

\def\Ann{\mathop{\rm Ann}\nolimits}
\def\Spec{\mathop{\rm Spec}\nolimits}
\def\Max{\mathop{\rm Max}\nolimits}

\def\Db{\mathop{\rm{D}^b}\nolimits}

\def\Dsg{\mathop{\rm{D}_{\mathsf{sg}}}\nolimits}

\def\Kb{\mathop{\rm{K}^b}\nolimits}
\def\Perf{\mathop{\rm{perf}}\nolimits}

\def\MMG{\mathop{\rm MMG}\nolimits}
\def\tilt{\mathop{\rm tilt}\nolimits}

\def\Fac{\mathop{\rm Fac}\nolimits}
\newcommand{\Gr}{{\sf EG}}

\newcommand{\cC}{\mathcal{C}}

\newcommand{\cF}{\mathcal{F}}
\newcommand{\cG}{\mathcal{G}}

\newcommand{\cM}{\mathcal{M}}

\newcommand{\cT}{\mathcal{T}}
\newcommand{\cU}{\mathcal{U}}

\newcommand{\cZ}{\mathcal{Z}}

\newcommand{\shift}[1]{\langle{#1}\rangle}

\renewcommand{\k}{k}

\edef\marginnotetextwidth{\the\textwidth}

\begin{document}
\title[Reduction and MMAs]{Reduction of triangulated categories and Maximal Modification Algebras for $cA_n$ singularities}
\author{Osamu Iyama}
\address{Osamu Iyama\\ Graduate School of Mathematics\\ Nagoya University\\ Chikusa-ku, Nagoya, 464-8602, Japan}
\email{iyama@math.nagoya-u.ac.jp}
\author{Michael Wemyss}
\address{Michael Wemyss, School of Mathematics, James Clerk Maxwell Building, The King's Buildings, Mayfield Road, Edinburgh, EH9 3JZ, UK.}
\email{wemyss.m@googlemail.com}
\dedicatory{Dedicated to Yuji Yoshino on the occasion of his 60th birthday.}
\thanks{The first author was partially supported by JSPS Grant-in-Aid for Scientific Research 24340004, 23540045, 20244001 and 22224001, and the second author by EP/K021400/1.}
\begin{abstract}
In this paper we define and study triangulated categories in which the Hom-spaces have Krull dimension at most one over some base ring (hence they have a natural 2-step filtration), and each factor of the filtration satisfies some Calabi--Yau type property.
If $\cC$ is such a category, we say that $\cC$ is Calabi--Yau with
$\dim\cC\le1$. We extend the notion of Calabi--Yau reduction to this setting, and prove general results which are an analogue of known results in cluster theory.

Such categories appear naturally in the setting of Gorenstein singularities in dimension three as the stable categories $\uCM R$ of Cohen--Macaulay modules. We explain the connection between Calabi--Yau reduction of $\uCM R$ and both partial crepant resolutions and $\mathds{Q}$-factorial terminalizations of
$\Spec R$, and we show under quite general assumptions that Calabi--Yau reductions exist.

In the remainder of the paper we focus on complete local $cA_n$ singularities $R$. By using a purely algebraic argument based on Calabi--Yau reduction of $\uCM R$, we give a complete classification of maximal modifying modules in terms of
the symmetric group, generalizing and strengthening results in \cite{BIKR} and \cite{DH}, where we do not need any restriction on the ground field. We also describe the mutation of modifying modules at an arbitrary (not necessarily indecomposable) direct summand. As a corollary when $k=\mathbb{C}$ we obtain
many autoequivalences of the derived category of the $\mathds{Q}$-factorial terminalizations of $\Spec R$.
\end{abstract}
\maketitle
\parindent 20pt
\parskip 0pt

\setcounter{tocdepth}{1}
\tableofcontents

\section{Introduction}

\subsection{Overview}\label{philosophy}
Let $R$ be a commutative Gorenstein ring of dimension $3$.  This paper develops algebraic tools, specifically CY reduction and other commutative algebraic techniques,
that allow us to deduce various results related to the geometry of partial resolutions of $\Spec R$ by arguing directly on the base $R$.   Let us first explain why this should be possible.

Suppose that we have a chain of crepant morphisms
\[
Y_n\to Y_{n-1}\to\hdots\to Y_1\to\Spec R.
\]
Then as the spaces get `larger' the corresponding singular derived categories $\Dsg(Y_i):=\Db(\coh Y_i)/\Perf(Y_i)$ get `smaller', since the singularities are improving under crepant modification.
The `largeness' of these categories is measured by the number of modifying objects that they contain (for the definition, see \S\ref{1.3}),
the best case being when there are no non-zero modifying objects, in which case the space must be a $\mathds{Q}$-factorial terminalization of $\Spec R$ \cite{IW5}. 

Now under favourable conditions, each $Y_i$ is derived equivalent to some ring $\Lambda_i$, and whenever this happens necessarily $\Lambda_i$ has the form $\Lambda_i\cong\End_R(M_i)\in\CM R$ for some $M_i\in\refl R$ \cite{IW5}.
In this case, we obtain full subcategories $\c{Z}_i$ of $\Dsg(\Lambda_i)$ and full dense functors
\[
\begin{tikzpicture}
\node (a) at (0,0) {$\Dsg(\Lambda_n)$};
\node (b) at (2,0) {$\Dsg(\Lambda_{n-1})$};
\node at (2,-0.5) {$\rotatebox[origin=c]{90}{$\subseteq$}$}; 
\node (b1) at (2,-1) {$\c{Z}_{n-1}$};
\node at (4,-0.5) {$\rotatebox[origin=c]{90}{$\subseteq$}$};
\node at (4.75,0) {$\hdots$};
\node (M1) at (5.2,-0.2) {$\,$};
\node (M1a) at (4,-1) {$\c{Z}_{n-2}$};
\node (c) at (7,0) {$\Dsg(\Lambda_{1})$};
\node at (7,-0.5) {$\rotatebox[origin=c]{90}{$\subseteq$}$}; 
\node (c1) at (7,-1) {$\c{Z}_{1}$};
\node (d) at (9,0) {$\Dsg(R)$};
\node at (9,-0.5) {$\rotatebox[origin=c]{90}{$\subseteq$}$}; 
\node (d1) at (9,-1) {$\c{Z}_{0}$};
\draw[<<-] (a) -- (b1);
\draw[<<-] (b) --  (M1a);
\draw[<<-] (M1) -- (c1);
\draw[<<-] (c) --   (d1);
\end{tikzpicture}
\]
(see \S\ref{CYredandMMAs}). Calabi--Yau reduction gives us the language in which to say $\Dsg(\Lambda_i)$ is obtained by quotienting the subcategory $\c{Z}_{i-1}$ of $\Dsg(\Lambda_{i-1})$, i.e.\ a way of obtaining $\Dsg(\Lambda_i)$ from the previous level.
Even better, we do not have to do this step-by-step, as in the above setup there is a functor from a certain subcategory $\mathcal{Z}$ of $\Dsg(R)$
all the way to $\Dsg(\Lambda_n)$, and we can obtain $\Dsg(\Lambda_n)$ by simply quotienting $\mathcal{Z}$ (see \S\ref{CYredandMMAs}).

Thus the idea is that we should be able to detect all the categories $\Dsg(\Lambda_i)$, and thus all the categories $\Dsg(Y_i)$, by tracking this entirely in the category $\Dsg(R)=\uCM R$ associated with the base singularity $R$.
We thus view any CY reduction of the category $\uCM R$ as the `shadow' of a partial crepant resolution of $\Spec R$ and in this way  $\uCM R$ should `see' all the singularities in the minimal models $Y_n$ of $\Spec R$.

\subsection{CY Categories of Dimension at Most One}\label{CY section intro}
We begin in a somewhat more general setting.  We let $\cC$ denote a triangulated category, and we suppose that $\cM\subseteq \cZ$ are full (not necessarily triangulated) subcategories of $\cC$.  

\begin{thm}
With the setup as above, assume further that $\cM$ is functorially finite in $\cZ$, and that $\cZ$ is closed under cones of $\cM$-monomorphisms and cocones of $\cM$-epimorphisms (see \S\ref{Triangle reduction} for more details).
Then $\cZ/[\cM]$ has the structure of a triangulated category.
\end{thm}

We then show that $\cZ/[\cM]$ inherits properties from $\cC$.
We fix a commutative $d$-dimensional equi-codimensional CM ring $R$ with a canonical module $\omega_R$.
For $X\in\mod R$, we denote by $\fl_RX$ the largest sub $R$-module with finite length.

\begin{defin}(=\ref{filtered CY defin})
Let $\cC$ be an $R$-linear triangulated category. We assume $\dim_R\cC\leq 1$, i.e. $\dim_R\Hom_{\cC}(X,Y)\leq 1$ for all $X,Y\in\cC$.
Let $T_\cC(X,Y):=\fl_R\Hom_{\cC}(X,Y)$ for every $X,Y\in\cC$, then there exists a short exact sequence
\[
0\to T_{\cC}(X,Y)\to\Hom_{\cC}(X,Y)\to F_{\cC}(X,Y)\to 0.
\]
We say that an autoequivalence  $S\colon\cC\to\cC$ is a \emph{Serre functor} if for all $X,Y\in\cC$ there are functorial isomorphisms
\begin{align*}
D_0(T_\cC(X,Y))&\cong T_\cC(Y,SX),\\
D_1(F_\cC(X,Y))&\cong F_\cC(Y,SX).
\end{align*}
where $D_i:=\Ext_R^{d-i}(-,\omega_R)$.
If $S:=[n]$ is a Serre functor for an integer $n$, we say that $\cC$ is an \emph{$n$-Calabi--Yau triangulated category of dimension at most one}, and write `$\cC$ is $n$-CY with $\dim_R\cC\leq 1$'.
\end{defin}

Now if $\cC$ is $n$-CY with $\dim_R\cC\leq 1$, we say that $M\in\cC$ is \emph{modifying} if $\Hom_{\cC}(M,M[i])=0$ for all $1\leq i\leq n-2$ and further $T_\cC(M,M[n-1])=0$. 
Given a modifying object $M$, we define
\begin{eqnarray*}
\cZ_M&:=&\{ X\in\cC\mid \Hom_{\cC}(X,M[i])=0\mbox{ for all }1\leq i\leq n-2\mbox{ and }T_{\cC}(X,M[n-1])=0\},\\
\cC_M&:=&\cZ_M/[M].
\end{eqnarray*} 
The following is our main result in this abstract setting. 
\begin{thm}\label{CYreduction theorem intro}(=\ref{CYreduction theorem}, \ref{bijection for MM})
Let $M$ be a modifying object in an $n$-CY triangulated category $\cC$ with $\dim_R\cC\leq1$.  Then \\
\t{(1)} $\cC_M$ is an $n$-CY triangulated category with $\dim_R\cC_M\leq1$.\\
\t{(2)} Assume that $\cC$ is Krull--Schmidt and $M$ is basic. Then there exists a bijection between basic modifying objects with summand $M$ in $\cC$ and basic modifying objects in $\cC_M$.
\end{thm}

We call the category $\cC_M$ the \emph{Calabi--Yau reduction} of $\cC$ with respect to $M$.

\subsection{CY Reduction in $\Dsg(R)$}\label{1.3}  

We then apply and improve the general results in \S\ref{CY section intro} in the setting of our original motivation (\S\ref{philosophy}).  When $R$ is a commutative Gorenstein ring, 
it is well-known that $\Dsg(R)\simeq\uCM R$ \cite{Buch}.  As an application of AR duality on not-necessarily-isolated singularities we obtain the following, which is our main motivation for studying $n$-CY categories $\cC$ with $\dim_R\cC\leq 1$.

\begin{thm}\label{nonisolatedARintro}(=\ref{nonisolatedAR})
Let $R$ be a commutative $d$-dimensional equi-codimensional Gorenstein ring with $\dim\Sing R\leq 1$.
Then $\uCM R$ is a $(d-1)$-CY triangulated category with $\dim_R(\uCM R)\le 1$.
\end{thm}

Now to relate \S\ref{CY section intro} to our previous work \cite{IW4}, the next result says that when $R$ is Gorenstein and $M\in\CM R$, the notion of modifying introduced in \S\ref{CY section intro} is equivalent to the condition $\End_R(M)\in\CM R$, which was the definition of modifying used in \cite{IW4}.
\begin{lemma}\label{modifyingThesame} \cite[4.3]{IW4} 
Let $R$ be a commutative $d$-dimensional equi-codimensional Gorenstein ring with $\dim\Sing R\leq 1$, and let $M\in\CM R$.
Then $\End_R(M)\in\CM R$ if and only if $M$ is a modifying object in $\uCM R$ (i.e. $\Hom_{\uCM R}(M,M[i])=0$ for all $1\leq i\leq d-3$ and $T_{\uCM R}(M,M[d-2])=0$).
\end{lemma}

By \ref{nonisolatedARintro} and \ref{modifyingThesame}, the first part of the next result is an immediate corollary of \ref{CYreduction theorem intro}.

\begin{cor}\label{CY red is uCM of End intro}(=\ref{CY red is uCM of End})
Let $R$ be a commutative $d$-dimensional equi-codimensional Gorenstein normal domain with $\dim\Sing R\leq 1$, and let $M\in\CM R$ be modifying. Then\\
\t{(1)} the CY reduction $(\uCM R)_M$ of $\uCM R$ is $(d-1)$-CY with $\dim_R(\uCM R)_M\leq 1$.\\
\t{(2)} $(\uCM R)_M\simeq \uCM\End_R(R\oplus M)$ as triangulated categories.
\end{cor}

As in \cite{IW4}, we view modifying modules as the building blocks of our theory:

\begin{defin} \cite{IW4}
\t{(1)} We say that a modifying $R$-module $M$ is \emph{maximal modifying} (or simply \emph{MM}) if whenever $M\oplus X$ is modifying for some $X\in\refl R$, then $X\in\add M$.  If $M$ is a MM module, we say that $\End_R(M)$ is a \emph{maximal modification algebra} (=\emph{MMA}).\\
\t{(2)} We say that $M\in\CM R$ is \emph{cluster tilting} (or simply \emph{CT}) if $\add M=\{X\in\CM R\mid \Hom_R(M,X)\in\CM R\}$.
\end{defin}

Recall that a normal scheme $X$ is defined to be {\em $\mathds{Q}$-factorial} if for every Weil divisor $D$, there exists $n\in\mathbb{N}$ for which $nD$ is Cartier.  If $X$ and $Y$ are varieties, then a projective birational morphism $f\colon Y\to X$ is called {\em crepant} if $f^*\omega_X=\omega_Y$. A {\em $\mathds{Q}$-factorial terminalization} of $X$ is a crepant projective birational morphism $f\colon Y\to X$ such that $Y$ has only $\mathds{Q}$-factorial terminal singularities.

Using \ref{CY red is uCM of End intro} together with our previous work relating MMAs to the minimal model program, we obtain the following. Recall we say that a module $M$ is a {\em generator} if $R\in\add M$. 

\begin{thm}(=\ref{tilting complex gives})
Let $R$ be a $3$-dimensional Gorenstein normal domain over $\mathbb{C}$ with rational singularities.
If some $\mathds{Q}$-factorial terminalization $Y$ of $\Spec R$ is derived equivalent to some ring $\Lambda$, then there exists an MM generator $M\in\CM R$ of $R$ such that\\
\t{(1)} the CY reduction $(\uCM R)_M$ of $\uCM R$ is triangle equivalent to $\Dsg(Y)$,\\
\t{(2)} $(\uCM R)_M$ is a $2$-CY triangulated category with $\dim_R(\uCM R)_M=0$ and has no non-zero rigid objects.
\end{thm}

Thus we can detect the $\mathds{Q}$-factorial terminalizations of $\Spec R$ on the level of CY reduction of $\uCM R$. 

\subsection{Mutation of MM modules and Tilting Mutation}\label{mutation intro section}
The results of CY reduction in \S\ref{1.3} allow us to deduce the existence of MM generators in certain concrete examples (see \S\ref{cAn section intro}), and so the question becomes how to deduce that they are \emph{all} the MM generators.  

Suppose that $R$ is a complete local normal Gorenstein domain with $\dim R=3$, and we denote by $\MMG R$ the set of isomorphism classes of basic MM generators of $R$.
By \cite[\S 6.2]{IW4} we have an operation on $\MMG R$ called \emph{mutation} which gives a new MM generator $\mu_i(M)$ for a given basic MM generator $M=R\oplus(\bigoplus_{i\in I}M_i)$ by replacing an indecomposable non-free direct summand $M_i$ of $M$.
We denote by $\Gr(\MMG R)$ the \emph{exchange graph} of MM generators of $R$, i.e. the set of vertices is $\MMG R$, and we draw an edge between $M$ and $\mu_i(M)$ for each $M\in\MMG R$ and $i\in I$.  

In this setting we have the following. 

\begin{thm}(=\ref{Auslander type})\label{Auslander type intro}
If the exchange graph $\Gr(\MMG R)$ has a finite connected component $C$, then $\Gr(\MMG R)=C$.
\end{thm}

Thus by \ref{Auslander type intro}  if we start with a MM generator $M$ and show that only finitely many MM modules are produced after repeatedly mutating at all possible indecomposable non-free direct summands,
then we can conclude that this finite list of MM generators are all.    This fact will be used in \S\ref{cAn section intro}, and is also needed in the geometric setting of Nolla--Sekiya \cite[\S 5.5]{NS}.

\subsection{$cA_n$ Singularities}\label{cAn section intro}
The remainder of this paper consists of an application of the above techniques to the case of complete local $cA_n$ singularities.
Let $k$ be any field, and let $S:=k[[x,y]]$. For $f\in\m$ where $\m:=(x,y)$, let
\[R:=S[[u,v]]/(f-uv)\]
and $f=f_1\hdots f_n$ be a factorization into prime elements of $S$. For any subset $I\subseteq\{ 1,\hdots, n\}$ we denote
\[
f_I:=\prod_{i\in I}f_i\ \mbox{ and }\ T_I:=(u,f_I)
\]
which is an ideal of $R$.  For a collection of subsets $\emptyset\subsetneq I_1\subsetneq I_2\subsetneq\hdots\subsetneq
I_m\subsetneq\{1,2,\hdots,n\}$, we say that $\F=(I_1,\hdots,I_m)$ is a {\em flag in the set $\{ 1,2,\hdots, n\}$}.  We say that the flag $\c{F}$ is {\em maximal} if $n=m+1$.
We can (and do) identify maximal flags with elements of the symmetric group (see \S\ref{Section cAn}). Given a flag $\c{F}=(I_1,\hdots,I_m)$, we define
\[T^\c{F}:=R\oplus\left(\bigoplus_{j=1}^{m} T_{I_j}\right)\]
and so for each $\omega\in\mathfrak{S}_n$ we let 
\[
T^\omega:=R\oplus (u,f_{\omega(1)})\oplus(u,f_{\omega(1)}f_{\omega(2)})\oplus\hdots\oplus (u,f_{\omega(1)}\hdots f_{\omega(n-1)}).
\]
As an application of results above, we have the following.
\begin{thm}\label{splitting algebraic intro} (=\ref{CY red is uCM of End}(2), \ref{key of classification}) 
Let $\F=(I_1,\hdots,I_m)$ be a flag, and let $(\uCM R)_{T^\c{F}}$ be the CY reduction of $\uCM R$ with respect to $T^\c{F}$.
Then we have triangle equivalences
\[
\uCM\End_R(T^{\F})\simeq (\uCM R)_{T^\c{F}}\simeq 
\bigoplus_{i=1}^{m+1}\uCM\left(\frac{k[[x,y,u,v]]}{(f_{I^i\backslash I^{i-1}}-uv)}\right).
\]
\end{thm}

We remark that \ref{splitting algebraic intro} is expected from, and generalizes, some results in \cite[\S5]{IW5} which rely on very precise information regarding the singularities in the $\mathds{Q}$-factorial terminalizations of $\Spec R$.  Here there are no restrictions on the field, and also our method generalizes; since in many other examples the explicit forms of $\mathds{Q}$-factorial terminalizations are not known, being able to argue directly on the base singularity $\Spec R$ is desirable.  

In fact, most of the proof of \ref{splitting algebraic intro} can be reduced to the following calculation, in dimension one.

\begin{thm}(=\ref{3.7})
Let $S=k[[x,y]]$ be a formal power series ring of two variables over an arbitrary field $k$, $f,g\in S$ and $R:=S/(fg)$ be a one-dimensional hypersurface.\\
\t{(1)} $\uCM R$ is a $2$-CY triangulated category with $\dim_R(\uCM R)\le 1$.\\
\t{(2)} $S/(f)$ is a modifying object in $\uCM R$, and the CY reduction $(\uCM R)_{S/(f)}$ of $\uCM R$ is triangle equivalent to $\uCM(S/(f))\times\uCM(S/(g))$.
\end{thm}

Then, combining \ref{splitting algebraic intro} with the mutation in \S\ref{mutation intro section}, we are able to give a complete classification of the modifying, MM and CT $R$-modules.  This generalizes and strengthens results from \cite{BIKR} and \cite{DH}, since our singularities are not necessarily isolated, and there is no restriction on the ground field.

\begin{thm}\label{commalgMMA_intro}(=\ref{classification})
Suppose that $f_1,\hdots,f_n\in\m:=(x,y)\subseteq k[[x,y]]$ are irreducible power series.  Let $R=k[[x,y,u,v]]/(f_1\hdots f_n-uv)$, then\\
\t{(1)} The basic modifying generators of $R$ are precisely $T^{\F}$, where $\F$ is a flag in $\{1,2,\hdots,n\}$.\\
\t{(2)} The basic MM generators of $R$ are precisely $T^\omega$, where $\omega\in\mathfrak{S}_n$.\\
\t{(3)} $R$ has a CT module if and only if $f_i\notin\m^2$ for all $1\le i\le n$. In this case, the basic CT $R$-modules are precisely $T^\omega$, where $\omega\in\mathfrak{S}_n$.
\end{thm}
This gives many examples of MMAs, and we give (in \ref{quiverOK}) the explicit quivers of these MMAs.  
We then specialize the field to $k=\mathbb{C}$ in order to apply our results to geometry.  As a corollary to \ref{commalgMMA_intro} we obtain the following remarkable result.
\begin{cor}(=\ref{MMnotgen}, \ref{finite number cor}) When $k=\mathbb{C}$,\\
\t{(1)} The MM modules are precisely $(I\otimes_R T^\omega)^{**}$ for some $\omega\in\mathfrak{S}_n$ and some $I\in\Cl(R)$.\\
\t{(2)} There are only finitely many algebras (up to Morita equivalence) in the derived equivalence class containing the MMAs.\\
\t{(3)} There are only finitely many algebras (up to Morita equivalence) in the derived equivalence class containing the $\mathds{Q}$-factorial terminalizations of $\Spec R$.
\end{cor}

Keeping $k=\mathbb{C}$ and $R$ as above, we then move from studying the $\mathds{Q}$-factorial terminalizations of $\Spec R$ to the arbitrary partial crepant resolutions of $\Spec R$, which in general have canonical singularities. 
We produce many examples of derived equivalences and autoequivalences on these singular spaces.  
The partial crepant resolutions of $\Spec R$ have a certain number of curves above the origin, and all singularities on these curves have the form $uv=f_{I}$ (see \cite[5.6]{IW5}).  
We describe the partial resolutions combinatorially in terms of flags $\F$, and denote the corresponding spaces by $X^\F$ (see \S\ref{Geo cor} for more details).
\begin{thm}\label{geometric_intro}(=\ref{permuteDbresult})
Let $\F$ and $\G$ be flags in $\{1,2,\hdots,n\}$. Then $X^\F$ and $X^\G$ are derived equivalent if they have the same number of curves above the origin of $\Spec R$, and the singularities of $X^\F$ can be permuted to the singularities of $X^\G$.  
\end{thm}
In fact, \ref{geometric_intro} comes very easily from a simple calculation which determines the mutations of a given modifying module: 
\begin{thm}(=\ref{mainmutation})
Fix a flag $\c{F}=(I_1,\hdots,I_m)$, and associate to $\c{F}$ the module $T^\F$ and the combinatorial picture $\c{P}(\F)$ (see \S\ref{mutation}). Choose $\emptyset\neq J\subseteq \{1,\hdots,m\}$,  then $\mu_J(T^\F)$ is the module corresponding to the $J$-reflection of $\c{P}(\F)$. 
\end{thm}
In particular, since in the proof of \ref{commalgMMA_intro} we prove that the exchange graph of MM modules is connected,
\ref{geometric_intro} gives the following alternative proof of \cite{Chen} in the case of complete local $cA_n$ singularities, which does not involve the argument passing to dimension four:
\begin{cor}\label{AllQfactDb_intro}
Let $R=\C{}[[x,y,u,v]]/(f_1\hdots f_n-uv)$, then all $\mathds{Q}$-factorial terminalizations of $\Spec R$ are derived equivalent.
\end{cor}
We remark that although all the results above are given in the complete local setting, this is mainly for our own convenience, since it simplifies calculations.  Most of our results also hold in the polynomial setting, but the proofs are much more technical.

\medskip
\noindent
{\bf Conventions.}  Throughout $R$ will always denote a commutative noetherian ring, and in Section \ref{Section cAn} $R$ will always denote $k[[x,y,u,v]]/(f-uv)$.  
All modules will be left modules, so for a ring $A$  we denote $\mod A$ to be the category of finitely generated left $A$-modules, and $\Mod A$ will denote the category of all left $A$-modules.  
Throughout when composing maps $fg$ will mean $f$ then $g$, similarly for quivers $ab$ will mean $a$ then $b$.  
Note that with this convention $\Hom_R(M,X)$ is a $\End_R(M)$-module and $\Hom_R(X,M)$ is a $\End_R(M)^{\rm op}$-module.  
For $M\in\mod A$ we denote $\add M$ to be the full subcategory consisting of summands of finite direct sums of copies of $M$, and we denote $\proj A:=\add A$ to be the category of finitely generated projective $A$-modules. 

\medskip
\noindent
{\bf Acknowledgements.} The authors would like to Hailong Dao for many interesting discussions regarding this work, especially with regards to the class group calculation in \S\ref{class section}.

\section{Triangulated and CY Reduction}\label{reduction section}

\subsection{Triangulated Reduction}\label{Triangle reduction}  In this section we let $\cC$ denote a triangulated category, and we suppose that $\cM\subseteq \cZ$ are full (not necessarily triangulated) subcategories of $\cC$.

Recall that we say a morphism $f:A\to B$ in $\cC$ is an \emph{$\cM$-monomorphism} if $(f\cdot):\Hom_{\cC}(B,M)\rightarrow \Hom_{\cC}(A,M)$ is surjective for all $M\in\cM$.
We say that $f$ is an \emph{$\cM$-epimorphism} if $(\cdot f):\Hom_{\cC}(M,A)\rightarrow \Hom_{\cC}(M,B)$ is surjective for all $M\in\cM$.
Similarly, we say that $f$ is a \emph{left $\cM$-approximation} of $B$ if $A\in\cM$ with $f$ an $\cM$-monomorphism, whereas we say that $f$ is a \emph{right $\cM$-approximation} of $A$ if $B\in\cM$ with $f$ an $\cM$-epimorphism.

Throughout this subsection we assume that $\cM\subseteq\cZ$ satisfies
\begin{enumerate}
\item[(1)] Every $Z\in\cZ$ admits a left $\cM$-approximation $Z\to M_Z$ and a right $\cM$-approximation $N_Z\to Z$ (i.e.\ $\cM$ is functorially finite in $\cZ$).
\item[(2a)]  Whenever $Z_1\stackrel{f}{\to} Z_2$ in $\cZ$ is an $\cM$-monomorphism, if we complete $f$ to a triangle $Z_1\stackrel{f}{\to} Z_2\stackrel{g}{\to} C\stackrel{h}{\to} Z_1[1]$ then $C\in\cZ$ and $g$ is an $\cM$-epimorphism. 
\item[(2b)] Whenever $Z_2\stackrel{g}{\to} Z_3$ in $\cZ$ is an $\cM$-epimorphism, if we complete $g$ to a triangle $B\stackrel{f}{\to} Z_2\stackrel{g}{\to} Z_3\stackrel{h}{\to} B[1]$ then $B\in\cZ$ and $f$ is an $\cM$-monomorphism.
\end{enumerate}

Recall that we denote $\cZ/[\cM]$ to be the additive category with the same objects as $\cZ$, but the morphism sets are defined to be $\Hom_{\cZ/[\cM]}(X,Y):=\Hom_{\cZ}(X,Y)/\cM(X,Y)$ where $\cM(X,Y)$ are the subspace of morphisms
that factor through an object in $\cM$. The following result generalizes \cite[4.2]{Iyama-Yoshino} where a very restrictive condition $\Hom_{\cC}(\cM,\cM[1])=0$ was assumed. Also we refer to \cite{LZ} for a related result.

\begin{thm}\label{reduction is triangulated}
With the assumptions (1), (2a) and (2b) as above, $\cU:=\cZ/[\cM]$ has the structure of a triangulated category.
\end{thm}
\begin{proof}
We first define an autoequivalence $\shift{1}$ on $\cU$.  For $Z\in\cU$ fix a left $\cM$-approximation $Z\xrightarrow{\alpha_Z}M_Z$ in $\cZ$, then define $Z\shift{1}$ to be the cone of $\alpha_Z$ in $\cC$, so we have a triangle
\[
Z\xrightarrow{\alpha_Z}M_Z \xrightarrow{\beta_Z}Z\shift{1}\xrightarrow{\gamma_Z}Z[1]
\]
Note by assumption (2a) that $Z\shift{1}\in\cZ$.  Now for $\overline{f}\in\Hom_{\cU}(Z_1,Z_2)$ consider 
\begin{equation}
\begin{array}{c}
{\SelectTips{cm}{10}
\xy0;/r.37pc/:
(0,0)*+{Z_1}="2",
(15,0)*+{M_{Z_1}}="3",
(30,0)*+{Z_1\shift{1}}="4",
(45,0)*+{Z_1[1]}="5",
(0,-9)*+{Z_2}="b2",
(15,-9)*+{M_{Z_2}}="b3",
(30,-9)*+{Z_2\shift{1}}="b4",
(45,-9)*+{Z_2[1]}="b5",
\ar"2";"3"^{\alpha_{Z_1}}
\ar"3";"4"^{\beta_{Z_1}}
\ar"4";"5"^{\gamma_{Z_1}}
\ar"b2";"b3"^{\alpha_{Z_2}}
\ar"b3";"b4"^{\beta_{Z_2}}
\ar"b4";"b5"^{\gamma_{Z_2}}
\ar"2";"b2"_f
\ar@{.>}"3";"b3"^g
\ar@{.>}"4";"b4"^{h:=f\shift{1}}
\ar"5";"b5"^{f[1]}
\endxy}
\end{array}
\label{ashift1}
\end{equation}
where $g$ exists because $\alpha_{Z_1}$ is an $\cM$-monomorphism, and $h$ exists by TR3.   We define $\overline{f}\shift{1}:=\overline{h}$.  It is standard to check that $\shift{1}$ is a well-defined functor $\cU\to\cU$ (e.g. \cite[2.6]{Iyama-Yoshino}).  
For the quasi-inverse functor, for every $Z\in\cU$ fix a right $\cM$-approximation $N_Z\xrightarrow{\varepsilon_Z} Z$ in $\cZ$, then define $Z\shift{-1}$ via the triangle
\[
Z\shift{-1}\xrightarrow{\delta_Z} N_Z \xrightarrow{\varepsilon_Z} Z \xrightarrow{\zeta_Z} Z\shift{-1}[1].
\]
In a similar way $\shift{-1}$ gives a well-defined functor $\cU\to\cU$.  Since by assumption (2a) $\beta_Z$ is a left $\cM$-approximation, 
and by assumption (2b) $\delta_Z$ is a right $\cM$-approximation, it is easy to check that $\shift{1}$ and $\shift{-1}$ are quasi-inverse to each other.

We now define triangles.  For $Z_1\xrightarrow{a}Z_2$ an $\cM$-monomorphism, we complete $a$ to a triangle $Z_1\xrightarrow{a}Z_2 \xrightarrow{b}Z_3\xrightarrow{c}Z_1[1]$ and so obtain a commutative diagram 
\begin{equation}\label{comm triangles}
\begin{array}{c}
{\SelectTips{cm}{10}
\xy0;/r.37pc/:
(0,0)*+{Z_1}="2",
(15,0)*+{Z_2}="3",
(30,0)*+{Z_3}="4",
(45,0)*+{Z_1[1]}="5",
(0,-9)*+{Z_1}="b2",
(15,-9)*+{M_{Z_1}}="b3",
(30,-9)*+{Z_1\shift{1}}="b4",
(45,-9)*+{Z_1[1]}="b5",
\ar"2";"3"^{a}
\ar"3";"4"^{b}
\ar"4";"5"^{c}
\ar"b2";"b3"^{\alpha_{Z_1}}
\ar"b3";"b4"^{\beta_{Z_1}}
\ar"b4";"b5"^{\gamma_{Z_1}}
\ar@{=}"2";"b2"
\ar@{.>}"3";"b3"^{\psi}
\ar@{.>}"4";"b4"^d
\ar@{=}"5";"b5"
\endxy}
\end{array}
\end{equation}
where $\psi$ exists since $a$ is an $\cM$-monomorphism, and $d$ exists by TR3.  We define triangles in $\cU$ to be all those isomorphic to the sequences
\[
Z_1\xrightarrow{\overline{a}}Z_2\xrightarrow{\overline{b}} Z_3\xrightarrow{\overline{d}} Z_1\shift{1}
\]
obtained in this way.  We now check the axioms of a triangulated category.\\
TR1(a) Let $Z\in\cU$, then $Z\xrightarrow{{\rm id}}Z$ is an $\cM$-monomorphism in $\cZ$, so 
\[
{\SelectTips{cm}{10}
\xy0;/r.37pc/:
(0,0)*+{Z}="2",
(15,0)*+{Z}="3",
(30,0)*+{0}="4",
(45,0)*+{Z[1]}="5",
(0,-9)*+{Z}="b2",
(15,-9)*+{M_{Z}}="b3",
(30,-9)*+{Z\shift{1}}="b4",
(45,-9)*+{Z[1]}="b5",
\ar"2";"3"^{{\rm id}}
\ar"3";"4"
\ar"4";"5"
\ar"b2";"b3"^{\alpha_{Z}}
\ar"b3";"b4"^{\beta_{Z}}
\ar"b4";"b5"^{\gamma_{Z}}
\ar@{=}"2";"b2"
\ar@{.>}"3";"b3"^{\psi}
\ar@{.>}"4";"b4"^d
\ar@{=}"5";"b5"
\endxy}
\]
shows that $Z\xrightarrow{\overline{{\rm id}}}Z\to 0\to Z\shift{1}$ is a triangle in $\cU$.\\
TR1(b) Every sequence isomorphic to a triangle is by definition a triangle.\\
TR1(c) Suppose that $Z_1\xrightarrow{\overline{f}}Z_2$ is a morphism in $\cU$.  Then $Z_1\xrightarrow{(\alpha_{Z_1}\,\, f)}M_{Z_1}\oplus Z_2$ is an $\cM$-monomorphism in $\cZ$, so completing to a triangle gives 
\begin{eqnarray}
\begin{array}{c}
{\SelectTips{cm}{10}
\xy0;/r.37pc/:
(0,0)*+{Z_1}="2",
(15,0)*+{M_{Z_1}\oplus Z_2}="3",
(30,0)*+{Z_3}="4",
(45,0)*+{Z_1[1]}="5",
(0,-9)*+{Z_1}="b2",
(15,-9)*+{M_{Z_1}}="b3",
(30,-9)*+{Z_1\shift{1}}="b4",
(45,-9)*+{Z_1[1]}="b5",
\ar"2";"3"^(0.4){(\alpha_{Z_1}\,\, f)}
\ar"3";"4"^(0.6){g_1\choose g_2}
\ar"4";"5"
\ar"b2";"b3"^{\alpha_{Z_1}}
\ar"b3";"b4"^{\beta_{Z_1}}
\ar"b4";"b5"^{\gamma_{Z_1}}
\ar@{=}"2";"b2"
\ar@{.>}"3";"b3"^{\psi}
\ar@{.>}"4";"b4"^d
\ar@{=}"5";"b5"
\endxy}
\end{array}
\label{from_ses_ok}
\end{eqnarray}
which shows that $Z_1\xrightarrow{\overline{f}} Z_2\xrightarrow{\overline{g_2}}Z_3\xrightarrow{\overline{h}}Z_1\shift{1}$, being isomorphic in $\cU$ to 
\[
Z_1\xrightarrow{\overline{(\alpha_{Z_1}\,\, f)}} M_{Z_1}\oplus Z_2\xrightarrow{\overline{g_1\choose g_2}}Z_3\xrightarrow{\overline{h}}Z_1\shift{1},
\]
is a triangle in $\cU$.\\
TR2. (Rotation) Suppose that $Z_1\xrightarrow{\overline{a}}Z_2\xrightarrow{\overline{b}} Z_3\xrightarrow{\overline{d}} Z_1\shift{1}$ is a triangle in $\cU$.  
By (\ref{from_ses_ok}) we can assume that $a$ is an $\cM$-monomorphism, and the triangle arises from the commutative diagram (\ref{comm triangles}).

Now by rotating (\ref{comm triangles}) we have a commutative diagram of triangles
\begin{equation}
\begin{array}{c}
{\SelectTips{cm}{10}
\xy0;/r.37pc/:
(0,0)*+{Z_2}="2",
(15,0)*+{Z_3}="3",
(30,0)*+{Z_1[1]}="4",
(45,0)*+{Z_2[1]}="5",
(0,-9)*+{M_{Z_1}}="b2",
(15,-9)*+{Z_1\shift{1}}="b3",
(30,-9)*+{Z_1[1]}="b4",
(45,-9)*+{M_{Z_1}[1]}="b5",
\ar"2";"3"^{b}
\ar"3";"4"^{c}
\ar"4";"5"^{-a[1]}
\ar"b2";"b3"^{\beta_{Z_1}}
\ar"b3";"b4"^{\gamma_{Z_1}}
\ar"b4";"b5"^{-\alpha_{Z_1}[1]}
\ar"2";"b2"^{\psi}
\ar"3";"b3"^d
\ar@{=}"4";"b4"
\ar"5";"b5"^{\psi[1]}
\endxy}\label{dgamma=c}
\end{array}
\end{equation}
from which is follows that $c\cdot(-\alpha_{Z_1}[1])=0$.  Hence applying the octahedral axiom 
\begin{equation}
\begin{array}{c}
\begin{tikzpicture}[xscale=1.2,yscale=1.2]
\node (A) at (0,0) {$Z_3$};
\node (B) at (1.5,0) {$Z_1[1]$};
\node (C) at (4,0) {$Z_2[1]$};
\node (A1) at (6,0) {$Z_3[1]$};
\node (A1a) at (5,-2) {$Z_3[1]$};
\node (B1) at (3.75,-3.5) {$Z_1[2]$};
\node (C1) at (2.25,-4.5) {$Z_2[2]$};
\node (D) at (intersection of B--B1 and A--A1a) {$M_{Z_1}[1]$};
\node (E) at (intersection of C--C1 and A--A1a) {$(M_{Z_1}\oplus Z_3)[1]$};
\node (F) at (intersection of C--C1 and B--B1) {$Z_1\shift{1}[1]$};
\draw[->] (A) -- node[above] {$\scriptstyle c$} (B);
\draw[->] (B) -- node[above] {$\scriptstyle -a[1]$} (C);
\draw[->] (C) -- node[above] {$\scriptstyle -b[1]$} (A1);
\draw[->] (B) -- node[right,pos=0.2] {$\scriptstyle -\alpha_{Z_1}[1]$} (D);
\draw[->] (D) -- node[above right,pos=0.1] {$\scriptstyle (1\,\, 0)$} (E);
\draw[->] (E) -- node[above,pos=0.95] {$\scriptstyle {0\choose1}$} (A1a);
\draw[->] (A) -- node[below] {$\scriptstyle 0$} (D);
\draw[->] (D) -- node[left] {$\scriptstyle -\beta_{Z_1}[1]$} (F);
\draw[->] (F) -- node[right, pos=0.2] {$\scriptstyle -\gamma_{Z_1}[1]$} (B1);
\draw[double distance=1.5pt] (A1) -- (A1a);
\draw[->] (A1a) -- node[right]  {$\scriptstyle c[1]$} (B1);
\draw[->] (B1) -- node[below,pos=0.2] {$\scriptstyle $} (C1);
\draw[densely dotted, ->] (C) -- (E);
\draw[densely dotted, ->] (E) -- (F);
\draw[densely dotted, ->] (F) -- (C1);
\end{tikzpicture}
\end{array}\label{oct1}
\end{equation}
and rotating we obtain a triangle
\[
Z_2\xrightarrow{e} M_{Z_1}\oplus Z_3\xrightarrow{f} Z_1\shift{1}\xrightarrow{g=-\gamma_{Z_1}\cdot a[1]} Z_2[1]
\]
where $f$ is an $\cM$-epimorphism since $\beta_{Z_1}$ is, and the diagram (\ref{oct1}) commutes.  By assumption (2b) $e$ is an $\cM$-monomorphism, so there exists a commutative diagram of triangles
\begin{equation}
\begin{array}{c}
{\SelectTips{cm}{10}
\xy0;/r.37pc/:
(0,0)*+{Z_2}="2",
(15,0)*+{M_{Z_1}\oplus Z_3}="3",
(30,0)*+{Z_1\shift{1}}="4",
(45,0)*+{Z_2[1]}="5",
(0,-9)*+{Z_2}="b2",
(15,-9)*+{M_{Z_2}}="b3",
(30,-9)*+{Z_2\shift{1}}="b4",
(45,-9)*+{Z_2[1]}="b5",
\ar"2";"3"^(0.4){e}
\ar"3";"4"^(0.6){f}
\ar"4";"5"^{-\gamma_{Z_1}\cdot a[1]}
\ar"b2";"b3"^{\alpha_{Z_2}}
\ar"b3";"b4"^{\beta_{Z_2}}
\ar"b4";"b5"^{\gamma_{Z_2}}
\ar@{=}"2";"b2"
\ar@{.>}"3";"b3"^{\phi}
\ar@{.>}"4";"b4"^h
\ar@{=}"5";"b5"
\endxy}
\end{array}\label{hgamma=g}
\end{equation}
and so by definition $Z_2\xrightarrow{\overline{e}}M_{Z_1}\oplus Z_3\xrightarrow{\overline{f}}Z_1\shift{1}\xrightarrow{\overline{h}} Z_2\shift{1}$ is a triangle in $\cU$.  We now claim that the diagram
\begin{eqnarray}
\begin{array}{c}
{\SelectTips{cm}{10}
\xy0;/r.37pc/:
(0,0)*+{Z_2}="2",
(15,0)*+{M_{Z_1}\oplus Z_3}="3",
(30,0)*+{Z_1\shift{1}}="4",
(45,0)*+{Z_2\shift{1}}="5",
(0,-9)*+{Z_2}="b2",
(15,-9)*+{Z_3}="b3",
(30,-9)*+{Z_1\shift{1}}="b4",
(45,-9)*+{Z_2\shift{1}}="b5",
\ar"2";"3"^(0.4){\overline{e}}
\ar"3";"4"^(0.6){\overline{f}}
\ar"4";"5"^{\overline{h}}
\ar"b2";"b3"^{\overline{b}}
\ar"b3";"b4"^{\overline{d}}
\ar"b4";"b5"^{-\overline{a}\shift{1}}
\ar@{=}"2";"b2"
\ar"3";"b3"_{\cong}^{\overline{{0\choose 1}}}
\ar@{=}"4";"b4"
\ar@{=}"5";"b5"
\endxy}
\end{array}
\label{want_comm}
\end{eqnarray}
commutes in $\cU$, as this proves that the rotation $Z_2\xrightarrow{\overline{b}}Z_3\xrightarrow{\overline{d}}Z_1\shift{1}\xrightarrow{-\overline{a}\shift{1}}Z_2\shift{1}$ is a triangle in $\cU$.  The left square in (\ref{want_comm}) commutes immediately from the commutativity of the top right square in (\ref{oct1}).  For the middle square in (\ref{want_comm}), write $f={f_1\choose f_2}$, so $\overline{f}=\overline{f_2}$.  Then from (\ref{oct1}) we see that $f_2\cdot \gamma_{Z_1}[1]=c[1]$, hence $(d-f_2)\cdot\gamma_{Z_1}\stackrel{\mbox{\scriptsize(\ref{dgamma=c})}}{=}c-c=0$.  This implies that $d-f_2$ factors through $\beta_{Z_1}$, thus $\overline{d}=\overline{f_2}$ and so the middle square in (\ref{want_comm}) commutes.  For the right hand square in (\ref{want_comm}), note that 
\[
(h+a\shift{1})\cdot\gamma_{Z_2}\stackrel{\mbox{\scriptsize(\ref{hgamma=g})}}{=}-\gamma_{Z_1}\cdot a[1]+a\shift{1}\cdot\gamma_{Z_2}\stackrel{\mbox{\scriptsize(\ref{ashift1})}}{=}0
\]
This implies that $h+a\shift{1}$ factors through $\beta_{Z_2}$ and so $\overline{h}=-\overline{a}\shift{1}$ as required.\\
The proofs of TR3 and TR4 are identical to those in \cite[4.2]{Iyama-Yoshino}.
\end{proof}

\subsection{CY Categories and CY Reduction}
In this subsection we let $R$ denote a $d$-dimensional equi-codimensional (i.e.\ $\dim R_\m=\dim R$ for all $\m\in\Max R$) CM ring with a canonical module $\omega_R$.
We assume that all our categories $\cC$ are \emph{$R$-linear}, in the sense that each Hom-set in $\cC$ is a finitely generated $R$-module such that the composition map is $R$-bilinear.

Let $\CM_iR:=\{ X\in\mod R\mid \depth_{R_\m}X_\m=\dim_{R_\m}X_\m=i \mbox{ for all }\m\in\Max R \}$ be the category of CM $R$-modules of dimension $i$.
Then the functor
\[D_i:=\Ext^{d-i}_R(-,\omega_R)\colon\mod R\to\mod R\]
gives a duality $D_i:\CM_iR\to\CM_iR$.
In the rest let $\cT:=\CM_0R$ and $\cF:=\CM_1R$. Thus $\CM_0R$ is the category of finite length $R$-modules. Clearly we have $\Hom_R(\cT,\cF)=0$, and also we have dualities $D_0\colon\cT\to\cT$ and $D_1\colon\cF\to\cF$. 
Any $X\in\mod R$ has a unique maximal finite length submodule, which we denote by $\fl_RX$.

Recall from the introduction the following.
\begin{defin}\label{filtered CY defin}
Let $\cC$ be an $R$-linear triangulated category. We assume $\dim_R\cC\le1$, where
\[\dim_R\cC:=\sup\{\dim_R\Hom_{\cC}(X,Y)\mid X,Y\in\cC\}.\]
For every $X,Y\in\cC$, by setting $T_\cC(X,Y):=\fl_R\Hom_{\cC}(X,Y)$, there exists a short exact sequence
\[
0\to T_{\cC}(X,Y)\to\Hom_{\cC}(X,Y)\to F_{\cC}(X,Y)\to 0
\]
with $T_\cC(X,Y)\in\cT$ and $F_\cC(X,Y)\in\cF$.  We say that an autoequivalence $S\colon\cC\to\cC$ is a \emph{Serre functor} if for all $X,Y\in\cC$ there are functorial isomorphisms
\begin{align*}
D_0(T_\cC(X,Y))&\cong T_\cC(Y,SX),\\
D_1(F_\cC(X,Y))&\cong F_\cC(Y,SX).
\end{align*}
If $S=[n]$ is a Serre functor for an integer $n$, we say that $\cC$ is an \emph{$n$-Calabi--Yau} triangulated category of dimension at most one.
\end{defin}

\begin{remark}
We remark that the usual definition of $n$-CY is to simply take $R=k$ where $k$ is an algebraically closed field, so $\cT=\mod R$, $\cF=\emptyset$ and $D_0=\Hom_k(-,k)$. 
\end{remark}

For our main examples of $n$-CY triangulated categories $\cC$ with $\dim_R\cC\leq 1$, we refer the reader to \S\ref{Section3}.

\begin{defin}\label{modifying def}
Fix $n\geq 2$ and suppose that $\cC$ is an $n$-CY triangulated category with $\dim_R\cC\leq1$.  We say that $M\in\cC$ is \emph{modifying} if\\
\t{(1)} $\Hom_{\cC}(M,M[i])=0$ for all $1\leq i\leq n-2$.\\
\t{(2)} $T_\cC(M,M[n-1])=0$.\\
Given a modifying object $M$, we define
\[
\cZ_M:=\{ X\in\cC\mid \Hom_{\cC}(X,M[i])=0\mbox{ for all }1\leq i\leq n-2 \mbox{ and } T_{\cC}(X,M[n-1])=0\}.
\]
Since $\cC$ is $n$-CY, we have
\[\cZ_M=\{ X\in\cC\mid \Hom_{\cC}(M,X[i])=0\mbox{ for all }1\leq i\leq n-2 \mbox{ and } T_{\cC}(M,X[n-1])=0\}.\]
We call the factor category $\cC_M:=\cZ_M/[M]$ the \emph{reduction} of $\cC$.
\end{defin}

\begin{remark}\label{fg Hom sets ok}
Since our category $\cC$ is $R$-linear, by assumption all Hom-sets are finitely generated $R$-modules. 
In particular, for any $M\in\cC$ this implies that $\Hom_{\cC}(X,M)\in\mod\End_{\cC}(M)^{\op}$ and $\Hom_{\cC}(M,X)\in\mod\End_{\cC}(M)$ for all $X\in\cC$.  Below, this allows us to construct both left and right $(\add M)$-approximations.
\end{remark}

Now we wish to show that given a modifying object $M$ in an $n$-CY triangulated category $\cC$ with $\dim_R\cC\leq 1$,
then the reduction $\cC_M$ has a structure of an $n$-CY triangulated category with $\dim_R\cC_M\leq 1$.  First we need the following technical lemma.

\begin{lemma}\label{autoequiv lemma}
Suppose that $M$ is a modifying object in an $n$-CY triangulated category $\cC$ with $\dim_R\cC\leq1$, with $n\geq 2$.  
Then for any $X\in\cZ_M$ (respectively, $Y\in\cZ_M$), there exists a triangle
\[
X\xrightarrow{f}M_0\xrightarrow{g}Y\to X[1]
\]
with $Y\in\cZ_M$ (respectively, $X\in\cZ_M$), where $f$ is a left $(\add M)$-approximation and $g$ is a right $(\add M)$-approximation.
\end{lemma}
\begin{proof}
Let $f:X\to M_0$ be a left $(\add M)$-approximation and complete $f$ to a triangle
\begin{equation}\label{approximation triangle}
X\xrightarrow{f}M_0\xrightarrow{g}Y\to X[1]
\end{equation}
in $\cC$.  We will show that $Y\in\cZ_M$ in two stages.\\
\emph{Claim 1:} $\Hom_{\cC}(Y,M[i])=0$ for all $1\leq i\leq n-2$.  When $n=2$ there is nothing to prove, so we suppose $n>2$.  
Simply applying applying $\Hom_{\cC}(-,M)$ to (\ref{approximation triangle}) and using the fact that $X,M_0\in\cZ(M)$, together with the surjectivity of $(f\cdot)\colon\Hom_{\cC}(M_0,M)\to\Hom_{\cC}(X,M)$, verifies the claim.\\
\emph{Claim 2:} We next claim that
\begin{equation}
\Hom_{\cC}(Y[1-n],M)\xrightarrow{g[1-n]\cdot}\Hom_{\cC}(M_0[1-n],M)\label{inj}
\end{equation}
is injective.  To verify this, if $n=2$ then applying $\Hom_{\cC}(-,M)$ to (\ref{approximation triangle}) gives an exact sequence
\[
\Hom_{\cC}(M_0,M)\xrightarrow{f\cdot}\Hom_{\cC}(X,M)\to
\Hom_{\cC}(Y[-1],M)\xrightarrow{g[-1]\cdot}\Hom_{\cC}(M_0[-1],M)
\]
from which the surjectivity of $(f\cdot)$ gives the injectivity of $(g[-1]\cdot)$. If $n>2$ then the exact sequence
\[
\Hom_{\cC}(X[2-n],M)=0\to\Hom_{\cC}(Y[1-n],M)\xrightarrow{g[1-n]\cdot}\Hom_{\cC}(M_0[1-n],M)
\]
verifies the claim.

Hence Claim 2 shows that $T_\cC(Y[1-n],M)$ embeds inside $\Hom_{\cC}(M_0[1-n],M)=F_\cC(M_0[1-n],M)$.  Since $\Hom_R(\cT,\cF)=0$ we deduce that $T_\cC(Y[1-n],M)=0$.  This, together with Claim 1, shows that $Y\in\cZ_M$.

It remains to show that $g$ is a right $(\add M)$-approximation.  If $n>2$ then $\Hom_{\cC}(M,X[1])=0$ and so $(\cdot g)\colon\Hom_{\cC}(M,M_0)\to\Hom_{\cC}(M,Y)$ is surjective, as required. 
Hence we can assume that $n=2$. Now we have the following commutative diagram with exact rows
\[
{\SelectTips{cm}{10}
\xy0;/r.37pc/:
(20,0)*+{0}="1",
(30,0)*+{T_\cC(M_0,M)}="2",
(50,0)*+{\Hom_{\cC}(M_0,M)}="3",
(70,0)*+{F_\cC(M_0,M)}="4",
(83,0)*+{0}="5",
(20,-7)*+{0}="b1",
(30,-7)*+{T_\cC(X,M)}="b2",
(50,-7)*+{\Hom_{\cC}(X,M)}="b3",
(70,-7)*+{F_\cC(X,M)}="b4",
(83,-7)*+{0}="b5",
\ar"1";"2"
\ar"2";"3"
\ar"3";"4"
\ar"4";"5"
\ar"b1";"b2"
\ar"b2";"b3"
\ar"b3";"b4"
\ar"b4";"b5"
\ar@{->>}"3";"b3"^{f\cdot}
\endxy}
\]
so since $\Hom_R(\cT,\cF)=0$ we obtain an induced surjection
\begin{equation}
F_\cC(M_0,M)\xrightarrow{f\cdot} F_\cC(X,M)\to 0.\label{surjection2}
\end{equation}
Now applying $\Hom_{\cC}(M,-)$ to \eqref{approximation triangle} and
$D_1$ to \eqref{surjection2} and comparing them, we have a commutative diagram
\[
{\SelectTips{cm}{10}
\xy0;/r.37pc/:
(10,0)*+{\Hom_{\cC}(M,M_0)}="1",
(30,0)*+{\Hom_{\cC}(M,Y)}="2",
(50,0)*+{\Hom_{\cC}(M,X[1])}="3",
(73,0)*+{\Hom_{\cC}(M,M_0[1])}="4",
(36,-7)*+{0}="b1",
(50,-7)*+{F_\cC(M,X[1])}="b2",
(73,-7)*+{F_\cC(M,M_0[1])}="b3"
\ar"1";"2"^{\cdot g}
\ar"2";"3"
\ar"3";"4"^{\cdot f[1]}
\ar"b1";"b2"
\ar"b2";"b3"^{\cdot f[1]}
\ar"3";"b2"
\ar"4";"b3"
\endxy}
\]
of exact sequences.  Since $M$ is modifying and $X\in\cZ_M$, the two vertical maps are isomorphisms.
Thus the injectivity of $(\cdot f[1])$ implies the surjectivity of $(\cdot g)$.
\end{proof}
The following is the main result of this section.

\begin{thm}\label{CYreduction theorem}
Let $M$ be a modifying object in an $n$-CY triangulated category $\cC$ with $\dim_R\cC\leq1$.  Then $\cC_M$ is an $n$-CY triangulated category with $\dim_R\cC_M\leq1$.
\end{thm}

The fact that $\cC_M$ is triangulated follows by combining \ref{autoequiv lemma} and \ref{reduction is triangulated}.
It is also clear that $\dim_R\cC_M\leq 1$ holds since $\Hom_{\cC_M}(X,Y)$ is a factor module of $\Hom_{\cC}(X,Y)$ for all $X,Y\in\cC$.
To prove the dualities, we need the following observations.

\begin{prop}\label{connection between U and C}
For any $X,Y\in\cC_M$, we have functorial isomorphisms\\
\t{(1)} $\Hom_{\cC_M}(X,Y\shift{i})\cong\Hom_{\cC}(X,Y[i])$ for all $i$ with $1\le i\le n-2$,\\
\t{(2)} $T_{\cC_M}(X,Y\shift{d-1})\cong T_{\cC}(X,Y[n-1])$.
\end{prop}

\begin{proof}
\emph{Step 1:} We claim for all $X,Y\in\cC_M$ that $T_{\cC_M}(Y,X\shift{1})\cong T_{\cC}(Y,X[1])$ if $n=2$, and $\Hom_{{\cC_M}}(Y,X\shift{1})\cong\Hom_{\cC}(Y,X[1])$ if $n>2$.
Considering $X$ and $Y$ as objects in $\cC$, applying $\Hom_\cC(Y,-)$ to the triangle
\begin{equation}\label{X approx triang}
X\xrightarrow{\alpha_X}M_X\xrightarrow{\beta_X}X\shift{1}\xrightarrow{\gamma_X}X[1]
\end{equation}
gives an exact sequence
\[
\Hom_{\cC}(Y,M_X)\xrightarrow{\cdot\beta_X}\Hom_{\cC}(Y,X\shift{1})\to\Hom_{\cC}(Y,X[1])\to\Hom_\cC(Y,M_X[1])
\]
Since $\beta_X$ is a right $(\add M)$-approximation, $\Cok(\cdot\beta_X)=\Hom_{\cC_M}(Y,X\shift{1})$. Thus we obtain an exact sequence
\begin{equation}
0\to\Hom_{\cC_M}(Y,X\shift{1})\to\Hom_\cC(Y,X[1])\to\Hom_\cC(Y,M_X[1])\label{star later}
\end{equation}
If $n=2$ then $T_{\cC}(Y,M_X[1])=0$, which forces $T_{\cC_M}(Y,X\shift{1})\cong T_\cC(Y,X[1])$.  
If $n>2$ then $\Hom_{\cC}(Y,M_X[1])=0$, hence $\Hom_{\cC_M}(Y,X\shift{1})\cong \Hom_\cC(Y,X[1])$.\\
\noindent
\emph{Step 2:} We claim that $\Hom_{\cC}(X,Y\shift{1}[i])\cong \Hom_{\cC}(X,Y[i+1])$ for all $X,Y\in{\cC_M}$ and all $1\leq i\leq n-3$.
If $n\leq 3$ this is vacuously true, so we assume that $n>3$. In this case, the claim follows by applying $\Hom_{\cC}(X,-)$ to the triangle
\begin{equation}\label{Y approx triang}
Y\xrightarrow{\alpha_Y} M_Y\xrightarrow{\beta_Y} Y\shift{1}\xrightarrow{\gamma_Y}Y[1].
\end{equation}
\noindent
\emph{Step 3:} We claim that if $n>2$, then $T_{\cC}(X,Y\shift{1}[n-2])\cong T_{\cC}(X,Y[n-1])$ for all $X,Y\in{\cC_M}$.
Applying $\Hom_{\cC}(X,-)$ to the triangle \eqref{Y approx triang} we obtain an exact sequence
\[
0\to\Hom_{\cC}(X,Y\shift{1}[n-2])\to\Hom_{\cC}(X,Y[n-1])\to\Hom_{\cC}(X,M_Y[n-1]).
\]
Since $T_{\cC}(X,M_Y[n-1])=0$, the claim follows.\\
\noindent
\emph{Step 4:} Now we show the assertions.
For any $i$ with $1\le i\le n-2$, we have
\begin{eqnarray*}
\Hom_{{\cC_M}}(X,Y\shift{i})&\stackrel{\scriptsize\mbox{Step 1}}{\cong}&\Hom_{\cC}(X,Y\shift{i-1}[1])\stackrel{\scriptsize\mbox{Step 2}}{\cong}\Hom_{\cC}(X,Y\shift{i-2}[2])\stackrel{\scriptsize\mbox{Step 2}}{\cong}\hdots\\
&\stackrel{\scriptsize\mbox{Step 2}}{\cong}&\Hom_{\cC}(X,Y\shift{1}[i-1])\stackrel{\scriptsize\mbox{Step 2}}{\cong}\Hom_{\cC}(X,Y[i]).
\end{eqnarray*}
Thus (1) holds. On the other hand, for $n=2$, Step 1 shows that (2) holds. For $n>2$,
\begin{eqnarray*}
T_{{\cC_M}}(X,Y\shift{n-1})&\stackrel{\scriptsize\mbox{Step 1}}{\cong}&T_{\cC}(X,Y\shift{n-2}[1])\stackrel{\scriptsize\mbox{Step 2}}{\cong}T_{\cC}(X,Y\shift{n-3}[2])\stackrel{\scriptsize\mbox{Step 2}}{\cong}\hdots\\
&\stackrel{\scriptsize\mbox{Step 2}}{\cong}&T_{\cC}(X,Y\shift{1}[n-2])\stackrel{\scriptsize\mbox{Step 3}}{\cong} T_{\cC}(X,Y[n-1])
\end{eqnarray*}
shows that (2) holds.
\end{proof}

Now we are ready to prove \ref{CYreduction theorem}.
\begin{proof}
\emph{Step 1:} First we establish the $D_0$ duality for ${\cC_M}$.  For any $X,Y\in{\cC_M}$, we have functorial isomorphisms
\[
T_{{\cC_M}}(Y,X\shift{1})\stackrel{\scriptsize\ref{connection between U and C}}{\cong} T_{\cC}(Y,X[1])\stackrel{\scriptsize\cC:\,\mbox{$n$-CY}}{\cong} 
D_0(T_{\cC}(X,Y[n-1]))\stackrel{\scriptsize\ref{connection between U and C}}{\cong}D_0(T_{\cC}(X,Y\shift{n-1})).
\]
Consequently we have the $D_0$ duality for ${\cC_M}$.\\
\noindent
\emph{Step 2:} We claim that we have an exact sequence
\begin{equation}\label{step 2 sequence}
0\to F_{{\cC_M}}(Y,X\shift{n-1})\to F_{\cC}(Y,X[n-1])\xrightarrow{\cdot\alpha_X[n-1]} F_{\cC}(Y,M_X[n-1]).
\end{equation}
If $n=2$, then this is true by (\ref{star later}).  If $n>2$ then applying $\Hom_{\cC}(Y,-)$ to (\ref{X approx triang}) we obtain an exact sequence
\[
0\to \Hom_{\cC}(Y,X\shift{1}[n-2])\to \Hom_{\cC}(Y,X[n-1])\to \Hom_{\cC}(Y,M_X[n-1]).
\]
Since $\Hom_{\cC}(Y,X\shift{1}[n-2])\cong \Hom_{{\cC_M}}(Y,X\shift{n-1})$ by \ref{connection between U and C}(1) and the right term equals $F_{\cC}(Y,M_X[n-1])$ by $Y\in\cZ_M$, we have an exact sequence
\[
0\to \Hom_{{\cC_M}}(Y,X\shift{n-1})\to \Hom_{\cC}(Y,X[n-1])\to F_{\cC}(Y,M_X[n-1]).
\]
Since $T_{{\cC_M}}(Y,X\shift{n-1})\cong T_{\cC}(Y,X[n-1])$ by \ref{connection between U and C}(2), the claim follows.\\
\noindent
\emph{Step 3:} Now we establish $D_1$ duality for ${\cC_M}$.
Applying $\Hom_\cC(-,Y)$ to (\ref{X approx triang}) and using the fact that $\alpha_X$ is a left $(\add M)$-approximation gives an exact sequence
\begin{equation*}
\Hom_\cC(M_X,Y)\to\Hom_\cC(X,Y)\to\Hom_{\cC_M}(X,Y)\to 0.
\end{equation*}
Applying $D_1$ and using the functorial isomorphism $D_1(X)\cong D_1(X/\fl_RX)$ for $X\in\mod R$ with $\dim_RX\le 1$, we have the upper sequence in the commutative diagram
\[
{\SelectTips{cm}{10}
\xy0;/r.37pc/:
(15,0)*+{0}="1",
(30,0)*+{D_1(F_{{\cC_M}}(X,Y))}="2",
(50,0)*+{D_1(F_{\cC}(X,Y))}="3",
(75,0)*+{D_1(F_{\cC}(M_X,Y))}="4",
(15,-7)*+{0}="b1",
(30,-7)*+{F_{{\cC_M}}(Y,X\shift{n-1})}="b2",
(50,-7)*+{F_{\cC}(Y,X[n-1])}="b3",
(75,-7)*+{F_{\cC}(Y,M_X[n-1])}="b4",
\ar"1";"2"
\ar"2";"3"
\ar"3";"4"
\ar"b1";"b2"
\ar"b2";"b3"
\ar"b3";"b4"^{\cdot\alpha_X[n-1]}
\ar"3";"b3"^{\simeq}
\ar"4";"b4"^{\simeq}
\endxy}
\]
of exact sequences, where the lower sequence is (\ref{step 2 sequence}).
Thus we have the desired isomorphism.
\end{proof}

\begin{thm}\label{bijection for MM}
Let $\cC$ be a Krull--Schmidt $n$-CY triangulated category with $\dim_R\cC\le1$, and let $M$ be a basic modifying object.  Then there exists a bijection between basic modifying (respectively MM, CT) objects with summand $M$ in $\cC$ and basic modifying (respectively MM, CT) objects in $\cC_M$
\end{thm}
\begin{proof}
Any basic maximal modifying object $N\in\cC$ with summand $M$ belongs to $\cZ_M$.
For any $X\in\cZ_M$, it follows from \ref{connection between U and C} that $X$ is modifying as an object in $\cC$ if and only if it is modifying as an object in ${\cC_M}$.
Thus we have the assertion.
\end{proof}

We say that a modifying object $M\in\cC$ is \emph{cluster tilting} (or simply \emph{CT}) if $\cZ_M=\add M$.
The following observation will be used in \S\ref{Section cAn}.

\begin{cor}\label{CT and CY reduction}
Let $M\in\cC$ be a modifying object and $\cC_M$ be the CY reduction of $\cC$ with respect to $M$.\\
\t{(1)} $M$ is MM if and only if $\cC_M$ has no non-zero modifying objects.\\
\t{(2)} $M$ is CT if and only if $\cC_M=0$.
\end{cor}

\subsection{$D_0$ Duality Implies $D_1$ Duality}
In this section, we keep the notation as in the previous section, but now we suppose that $R$ is a \emph{complete local} CM ring, with canonical $\omega_R$.
We denote $\cC_0:=\{ X\in\cC\mid \End_{\cC}(X)\in\cT\}$. By \ref{fg Hom sets ok} the following observation is clear.

\begin{lemma}\label{C_0 lemma}
Let $X\in\cC$. Then $X\in\cC_0$ if and only if $\Hom_{\cC}(X,Y)\in\cT$ holds for all $Y\in\cC$ if and only if $\Hom_{\cC}(Y,X)\in\cT$ for all $Y\in\cC$.
\end{lemma}
The aim of this section is to show that when $R$ is complete local, any $D_0$ duality on $\cC_0$ determines the $D_0$ and $D_1$ dualities on $\cC$.

\begin{thm}\label{D0 implies D1}
Assume $\dim_R\cC\le1$. Let $S$ be an autoequivalence of $\cC$ such that for all $X,Y\in\cC_0$ there exists a functorial isomorphism
\[
\phi_{X,Y}\colon\Hom_{\cC}(X,Y)\cong D_0\Hom_{\cC}(Y,S X).
\]
\t{(1)} For all $X,Y\in\cC$ there exists a functorial isomorphism
\[
\phi_{X,Y}\colon T_\cC(X,Y)\cong D_0\left( T_\cC(Y,S X)\right).
\]
\t{(2)} For all $X,Y\in\cC$ there exists a functorial isomorphism
\[
\psi_{X,Y}\colon F_\cC(X,Y)\cong D_1\left( F_\cC(Y[1],S X)\right).
\]
\end{thm}

\begin{proof}
For every $X,Y\in\cC$ let $I_{X,Y}$ be the annihilator of the $R$-module $\Hom_{\cC}(X,Y)\oplus \Hom_{\cC}(X,Y[1])\oplus\Hom_{\cC}(Y,SX)\oplus \Hom_{\cC}(Y[1],SX)$.  Since $R/I_{X,Y}$ is a local noetherian ring of dimension at most one, we can fix an element $t\in\m$ (depending on $X$ and $Y$) such that $R/(I_{X,Y}+(t))$ is artinian.

For each $\ell\ge0$, consider a triangle
\begin{equation}\label{t ell triangle}
Y\xrightarrow{t^\ell}Y\xrightarrow{\alpha_\ell}Y_\ell\xrightarrow{\beta_\ell}Y[1].
\end{equation}
We first claim that $Y_\ell\in\cC_0$ for all $\ell\geq 0$.
Applying $\Hom_{\cC}(X,-)$, we have an exact sequence
\begin{equation}\label{tell long}
(X,Y)\xrightarrow{t^\ell}(X,Y)\xrightarrow{\cdot\alpha_\ell}(X,Y_\ell)\xrightarrow{\cdot\beta_\ell}(X,Y[1])\xrightarrow{t^\ell}(X,Y[1]).\end{equation}
This gives rise to a short exact sequence
\[
0\to R/(t^\ell)\otimes_R(X,Y)\xrightarrow{\cdot\alpha_\ell}(X,Y_\ell)\xrightarrow{\cdot\beta_\ell}\{ f\in(X,Y[1]) \mid t^\ell f=0\}\to 0.
\]
The right and left hand terms are modules over the artinian ring $R/(I_{X,Y}+(t^\ell))$ and hence are finite length $R$-modules.  It follows that the middle term has finite length, i.e.\ $\Hom_\cC(X,Y_\ell)\in\cT$.
This holds for all $X\in\cC$, so by \ref{C_0 lemma} $Y_\ell\in\cC_0$, as claimed.

Now if $\ell$ is sufficiently large, then the kernel of the map
$t^\ell\colon(X,Y[1])\to(X,Y[1])$ is $T_{\cC}(X,Y[1])$,
and so (\ref{tell long}) gives an exact sequence
\begin{equation}\label{first sequence}
0\to R/(t^\ell)\otimes_R(X,Y)\xrightarrow{\cdot\alpha_\ell}(X,Y_\ell)\xrightarrow{\cdot\beta_\ell}T_{\cC}(X,Y[1])\to 0.
\end{equation}
On the other hand, applying $\Hom_{\cC}(-,S X)$ to \eqref{t ell triangle}, we have an exact sequence
\[
(Y[1],S X)\xrightarrow{t^\ell}(Y[1],S X)\xrightarrow{\beta_\ell\cdot}(Y_\ell,S X)\xrightarrow{\alpha_\ell\cdot}(Y,S X)\xrightarrow{t^\ell}(Y,S X).
\]
Again for sufficiently large $\ell$, we have an exact sequence
\begin{equation}\label{second sequence}
0\to R/(t^\ell)\otimes_R(Y[1],S X)\xrightarrow{\beta_\ell\cdot}(Y_\ell,S X)\xrightarrow{\alpha_\ell\cdot}T_{\cC}(Y,S X)\to0.
\end{equation}
(a) We now show that there exists a functorial isomorphism
\[
\Hom_{\cC}(X,Y)\cong D_0\Hom_{\cC}(Y,S X)
\]
for all $X,Y\in\cC$ if either $X$ or $Y$ belongs to $\cC_0$.

First we assume $X\in\cC_0$.  Since $X$ and $Y_\ell$ belong to $\cC_0$, we have exact sequences
\[
{\SelectTips{cm}{10}
\xy0;/r.37pc/:
(20,0)*+{0}="1",
(30,0)*+{(X,Y)}="2",
(50,0)*+{(X,Y_\ell)}="3",
(70,0)*+{(X,Y[1])}="4",
(83,0)*+{0}="5",
(20,-10)*+{0}="b1",
(30,-10)*+{D_0(Y,SX)}="b2",
(50,-10)*+{D_0(Y_\ell,SX)}="b3",
(70,-10)*+{D_0(Y[1],SX)}="b4",
(83,-10)*+{0}="b5",
\ar"1";"2"
\ar"2";"3"^{\cdot\alpha_\ell}
\ar"3";"4"^{\cdot\beta_\ell}
\ar"4";"5"
\ar"b1";"b2"
\ar"b2";"b3"^{D_0(\alpha_\ell\cdot)}
\ar"b3";"b4"^{D_0(\beta_\ell\cdot)}
\ar"b4";"b5"
\ar"3";"b3"_{\cong}^{\phi_{X,Y_\ell}}
\endxy}
\]
for sufficiently large $\ell$ by \eqref{first sequence} and \eqref{second sequence}.

We now show that the composition
\begin{equation}\label{composition zero}
(X,Y)\xrightarrow{\cdot\alpha_\ell}(X,Y_\ell)\xrightarrow{\phi_{X,Y_\ell}}D_0(Y_\ell,S X)\xrightarrow{D_0(\beta_\ell\cdot)}D_0(Y[1],S X)
\end{equation}
is zero.  For any $f\in(X,Y)$ and $g\in(Y[1],S X)$, consider the following commutative diagram:
\[
{\SelectTips{cm}{10}
\xy0;/r.37pc/:
(50,0)*+{(X,X)}="3",
(68,0)*+{D_0(X,SX)}="4",
(50,-10)*+{(X,Y_\ell)}="b3",
(68,-10)*+{D_0(Y_\ell,X)}="b4",
\ar"3";"4"^{\phi_{X,X}}
\ar"b3";"b4"^{\phi_{X,Y_\ell}}
\ar"3";"b3"^{\cdot f\alpha_\ell}
\ar"4";"b4"^{D_0(f\alpha_\ell\cdot)}
\endxy}
\]

Considering $1_X\in(X,X)$, we have that $\phi_{X,Y_\ell}(f\alpha_\ell)$ is equal to the image of $\phi_{X,X}(1_X)$ under the map $D_0(f\alpha_\ell\cdot)$.  But the composition (\ref{composition zero}) is the image of $\phi_{X,Y_\ell}(f\alpha_\ell)$ under the map $D_0(\beta_\ell\cdot)$, and hence the composition (\ref{composition zero}) is equal to the image of $\phi_{X,X}(1_X)$ under the map $D_0(f\alpha_\ell\beta_\ell\cdot)$. Since $\alpha_\ell\beta_\ell=0$, this is zero and so the assertion follows.

In particular $\phi_{X,Y_\ell}\colon(X,Y_\ell)\to D_0(Y_\ell,S X)$ induces an injective map
\[\phi_{X,Y}:=\phi_{X,Y_\ell}|_{(X,Y)}\colon(X,Y)\to D_0(Y,S X)\]
and a surjective map $(X,Y[1])\to D_0(Y[1],S X)$.
Thus $\length_R(X,Y)\le\length_R(Y,S X)$ and $\length_R(X,Y[1])\ge\length_R(Y[1],S X)$ hold.
Replacing $Y[1]$ in the second inequality by $Y$, we have $\length_R(X,Y)=\length_R(Y,S X)$. Thus $\phi_{X,Y}$ has to be an isomorphism.

It is routine to check $\phi$ is independent of $\ell$ and $t$, and functorial. The case $Y\in\cC_0$ follows immediately from the case $X\in\cC_0$.\\
(1) Let $X,Y\in\cC$. Since $Y_\ell$ belong to $\cC_0$, we have exact sequences
\[
{\SelectTips{cm}{10}
\xy0;/r.37pc/:
(15,0)*+{0}="1",
(30,0)*+{R/(t^\ell)\otimes_R(X,Y)}="2",
(50,0)*+{(X,Y_\ell)}="3",
(75,0)*+{T_{\cC}(X,Y[1])}="4",
(92,0)*+{0}="5",
(15,-10)*+{0}="b1",
(30,-10)*+{D_0(T_{\cC}(Y,SX))}="b2",
(50,-10)*+{D_0(Y_\ell,SX)}="b3",
(75,-10)*+{D_0(R/(t^\ell)\otimes_R(Y[1],S X))}="b4",
(92,-10)*+{0}="b5",
\ar"1";"2"
\ar"2";"3"^(0.6){\cdot\alpha_\ell}
\ar"3";"4"^{\cdot\beta_\ell}
\ar"4";"5"
\ar"b1";"b2"
\ar"b2";"b3"^(0.5){D_0(\alpha_\ell\cdot)}
\ar"b3";"b4"^(0.35){D_0(\beta_\ell\cdot)}
\ar"b4";"b5"
\ar"3";"b3"^{\phi_{X,Y_\ell}}
\endxy}
\]
for sufficiently large $\ell$ by \eqref{first sequence} and \eqref{second sequence}.
Since $\bigcap_{\ell\ge0}t^\ell(X,Y)=0$, we have $T_{\cC}(X,Y)\cap t^\ell(X,Y)=0$ for sufficiently large $\ell$.
Thus the natural map $T_{\cC}(X,Y)\to R/(t^\ell)\otimes_R(X,Y)$ is injective for sufficiently large $\ell$, and we have exact sequences
\[
{\SelectTips{cm}{10}
\xy0;/r.37pc/:
(15,0)*+{0}="1",
(30,0)*+{T_{\cC}(X,Y)}="2",
(50,0)*+{(X,Y_\ell)}="3",
(15,-10)*+{0}="b1",
(30,-10)*+{D_0(T_{\cC}(Y,SX))}="b2",
(50,-10)*+{D_0(Y_\ell,SX)}="b3",
(75,-10)*+{D_0(R/(t^\ell)\otimes_R(Y[1],S X))}="b4",
(92,-10)*+{0}="b5",
\ar"1";"2"
\ar"2";"3"^(0.6){\cdot\alpha_\ell}
\ar"b1";"b2"
\ar"b2";"b3"^(0.5){D_0(\alpha_\ell\cdot)}
\ar"b3";"b4"^(0.35){D_0(\beta_\ell\cdot)}
\ar"b4";"b5"
\ar"3";"b3"^{\phi_{X,Y_\ell}}
\endxy}
\]

Now we show that the following composition is zero (caution: we can not use the argument in (a) since we do not have $\phi_{X,X}$):
\[
T_{\cC}(X,Y)\xrightarrow{\cdot\alpha_\ell}(X,Y_\ell)\xrightarrow{\phi_{X,Y_\ell}}D_0(Y_\ell,S X)\xrightarrow{D_0(\beta_\ell\cdot)}D_0(Y[1],S X)
\]
For $m\ge0$, consider a triangle
\[
X\xrightarrow{t^m}X\xrightarrow{\alpha'_m}X_m\xrightarrow{\beta'_m}X[1]
\]
and a commutative diagram
\[
\begin{array}{c}
{\SelectTips{cm}{10}
\xy0;/r.37pc/:
(5,0)*+{T_{\cC}(X,Y)}="2",
(20,0)*+{(X,Y_\ell)}="3",
(40,0)*+{D_0(Y_\ell,S X)}="4",
(65,0)*+{D_0(Y[1],S X)}="5",
(5,-9)*+{(X_m,Y)}="b2",
(20,-9)*+{(X_m,Y_\ell)}="b3",
(40,-9)*+{D_0(Y_\ell,S X_m)}="b4",
(65,-9)*+{D_0(Y[1],S X_m)}="b5",
\ar"2";"3"^{\cdot\alpha_l}
\ar"3";"4"^(0.45){\phi_{X,Y_\ell}}
\ar"4";"5"^{D_0(\beta_\ell\cdot)}
\ar"b2";"b3"_{\cdot\alpha_l}
\ar"b3";"b4"_(0.45){\phi_{X_m,Y_\ell}}
\ar"b4";"b5"_{D_0(\beta_\ell\cdot)}
\ar"b2";"2"^{\alpha_m^\prime\cdot}
\ar"b3";"3"^{\alpha_m^\prime\cdot}
\ar"b4";"4"^{D_0(\cdot S\alpha_m^\prime)}
\ar"b5";"5"^{D_0(\cdot S\alpha_m^\prime)}
\endxy}
\end{array}
\]
The lower composition is zero by (\ref{composition zero}), and for sufficiently large $m$ the left vertical map is surjective.  Hence the upper composition is also zero, and so the assertion follows.

In particular $\phi_{X,Y_\ell}\colon(X,Y_\ell)\to D_0(Y_\ell,S X)$ induces an injective map $T_{\cC}(X,Y)\to D_0(T_{\cC}(Y,S X))$, so we have an induced injective map
\[
\phi_{X,Y}:=\phi_{X,Y_\ell}|_{T_{\cC}(X,Y)}\colon T_{\cC}(X,Y)\to D_0(T_{\cC}(Y,S X)).
\]
Thus we have $\length_RT_{\cC}(X,Y)\le\length_RT_{\cC}(Y,S X)$, and by replacing $X$ and $Y$ by $Y$ and $S X$ respectively, we have $\length_R T_{\cC}(X,Y)=\length_R T_{\cC}(Y,S X)$. Thus $\phi_{X,Y}$ has to be an isomorphism.

It is routine to check $\phi$ is independent of $\ell$ and $t$, and functorial.\\
(2) We have a commutative diagram of triangles:
\begin{equation}\label{t ell commutative}
\begin{array}{c}
{\SelectTips{cm}{10}
\xy0;/r.37pc/:
(0,0)*+{Y}="2",
(15,0)*+{Y}="3",
(30,0)*+{Y_{\ell-1}}="4",
(45,0)*+{Y[1]}="5",
(0,-9)*+{Y}="b2",
(15,-9)*+{Y}="b3",
(30,-9)*+{Y_\ell}="b4",
(45,-9)*+{Y[1]}="b5",
\ar"2";"3"^{t^{\ell-1}}
\ar"3";"4"^{\alpha_{\ell-1}}
\ar"4";"5"^{\beta_{\ell-1}}
\ar"b2";"b3"^{t^\ell}
\ar"b3";"b4"^{\alpha_\ell}
\ar"b4";"b5"^{\beta_\ell}
\ar"b2";"2"^{t}
\ar@{=}"3";"b3"
\ar"b4";"4"^{\gamma_\ell}
\ar"b5";"5"^t
\endxy}
\end{array}
\end{equation}
By \eqref{first sequence}, for sufficiently large $\ell$ we have a commutative diagram of exact sequences
\[
{\SelectTips{cm}{10}
\xy0;/r.37pc/:
(15,0)*+{0}="1",
(30,0)*+{R/(t^{\ell-1})\otimes_R(X,Y)}="2",
(50,0)*+{(X,Y_{\ell-1})}="3",
(70,0)*+{T_{\cC}(X,Y[1])}="4",
(83,0)*+{0}="5",
(15,-10)*+{0}="b1",
(30,-10)*+{R/(t^\ell)\otimes_R(X,Y)}="b2",
(50,-10)*+{(X,Y_\ell)}="b3",
(70,-10)*+{T_{\cC}(X,Y[1])}="b4",
(83,-10)*+{0}="b5",
\ar"1";"2"
\ar"2";"3"^(0.6){\cdot\alpha_{\ell-1}}
\ar"3";"4"^{\cdot\beta_{\ell-1}}
\ar"4";"5"
\ar"b1";"b2"
\ar"b2";"b3"^(0.6){\cdot\alpha_\ell}
\ar"b3";"b4"^{\cdot\beta_\ell}
\ar"b4";"b5"
\ar"b2";"2"^{1}
\ar"b3";"3"^{\cdot\gamma_\ell}
\ar"b4";"4"^{t}
\endxy}
\]
Now the inverse limit of the right column $t\colon T_{\cC}(X,Y[1])\to T_{\cC}(X,Y[1])$ is zero. 
Since taking inverse limits is left exact, we have an isomorphism
\begin{equation}\label{first limit}
(X,Y)=\varprojlim_\ell R/(t^\ell)\otimes_R(X,Y)\stackrel{\cdot\alpha_\ell}{\cong}\varprojlim_\ell(X,Y_\ell).
\end{equation}
by taking inverse limits of each column.

On the other hand, by \eqref{t ell commutative} and \eqref{second sequence}, we have a commutative diagram of exact sequences:
\[
{\SelectTips{cm}{10}
\xy0;/r.37pc/:
(15,0)*+{0}="1",
(30,0)*+{R/(t^{\ell-1})\otimes_R(Y[1],SX)}="2",
(53,0)*+{(Y_{\ell-1},SX)}="3",
(70,0)*+{T_{\cC}(Y,SX)}="4",
(83,0)*+{0}="5",
(15,-10)*+{0}="b1",
(30,-10)*+{R/(t^\ell)\otimes_R(Y[1],SX)}="b2",
(53,-10)*+{(Y_\ell,SX)}="b3",
(70,-10)*+{T_{\cC}(Y,SX)}="b4",
(83,-10)*+{0}="b5",
\ar"1";"2"
\ar"2";"3"^(0.62){\beta_{\ell-1}\cdot}
\ar"3";"4"^{\alpha_{\ell-1}\cdot}
\ar"4";"5"
\ar"b1";"b2"
\ar"b2";"b3"^(0.6){\beta_\ell\cdot}
\ar"b3";"b4"^{\alpha_\ell\cdot}
\ar"b4";"b5"
\ar"2";"b2"^{t}
\ar"3";"b3"^{\gamma_\ell\cdot}
\ar@{=}"4";"b4"
\endxy}
\]
in which every Hom-set has finite length.  Hence applying $D_0$, we have a commutative diagram of exact sequences:
\[
{\SelectTips{cm}{10}
\xy0;/r.37pc/:
(20,0)*+{0}="1",
(30,0)*+{T_{\cC}(X,Y)}="2",
(47,0)*+{(X,Y_{\ell-1})}="3",
(70,0)*+{D_0(R/(t^{\ell-1})\otimes_R(Y[1],SX))}="4",
(88,0)*+{0}="5",
(20,-10)*+{0}="b1",
(30,-10)*+{T_{\cC}(X,Y)}="b2",
(47,-10)*+{(X,Y_\ell)}="b3",
(70,-10)*+{D_0(R/(t^\ell)\otimes_R(Y[1],SX))}="b4",
(88,-10)*+{0}="b5",
\ar"1";"2"
\ar"2";"3"^{\cdot\alpha_{\ell-1}}
\ar"3";"4"
\ar"4";"5"
\ar"b1";"b2"
\ar"b2";"b3"^{\cdot\alpha_\ell}
\ar"b3";"b4"
\ar"b4";"b5"
\ar@{=}"b2";"2"
\ar"b3";"3"^{\cdot\gamma_\ell}
\ar"b4";"4"^{t}
\endxy}
\]
Since the Mittag--Leffler condition is satisfied, taking the inverse limits of each column, we obtain an exact sequence
\[
0\to T_{\cC}(X,Y)\to\varprojlim_\ell(X,Y_{\ell})\to\varprojlim_\ell D_0(R/(t^{\ell})\otimes_R(Y[1],S X))\to 0.
\]
Comparing with \eqref{first limit}, we have an isomorphism
\[
F_{\cC}(X,Y)=\frac{(X,Y)}{T_{\cC}(X,Y)}\cong\varprojlim_\ell D_0(R/(t^{\ell})\otimes_R(Y[1],S X)).
\]
Now by \ref{prop for D_1 duality} below, we have an isomorphism
\[
\psi_{X,Y}\colon F_{\cC}(X,Y)\cong D_1(Y[1],S X)=D_1(F_{\cC}(Y[1],S X)).
\]
It is routine to check that $\psi$ is functorial.
This finishes the proof.
\end{proof}

\begin{prop}\label{prop for D_1 duality}
For any $M\in\mod R$ such that $R/(\Ann M+(t))$ is artinian, we have
\[
\varprojlim(\cdots\xrightarrow{t}D_0(M/t^3M)\xrightarrow{t}D_0(M/t^2M)\xrightarrow{t}D_0(M/tM))\cong D_1(M).
\]
\end{prop}
\begin{proof}
Since for each $\ell\geq 0$ the kernel of the map $t^\ell\colon M\twoheadrightarrow t^\ell M$ has finite length, we have an isomorphism 
\begin{equation}\label{t ell iso}
t^\ell\colon D_1(t^\ell M)\cong D_1(M).
\end{equation}
for all $\ell\geq 0$. 
Now consider the following commutative diagram of exact sequences:
\[
{\SelectTips{cm}{10}
\xy0;/r.3pc/:
(17,0)*+{0}="1",
(30,0)*+{t^{\ell-1}M}="2",
(50,0)*+{M}="3",
(70,0)*+{M/t^{\ell-1}M}="4",
(82,0)*+{0}="5",
(17,-10)*+{0}="b1",
(30,-10)*+{t^{\ell}M}="b2",
(50,-10)*+{M}="b3",
(70,-10)*+{M/t^{\ell}M}="b4",
(82,-10)*+{0}="b5",
\ar"1";"2"
\ar"2";"3"
\ar"3";"4"
\ar"4";"5"
\ar"b1";"b2"
\ar"b2";"b3"
\ar"b3";"b4"
\ar"b4";"b5"
\ar"2";"b2"^{t}
\ar"3";"b3"^{t}
\ar"4";"b4"^t
\endxy}
\]
Applying $\Hom_R(-,\omega_R)$ and using \eqref{t ell iso}, we have a commutative diagram of exact sequences
\[
{\SelectTips{cm}{10}
\xy0;/r.3pc/:
(30,0)*+{D_1(M)}="2",
(50,0)*+{D_1(M)}="3",
(70,0)*+{D_0(M/t^{\ell-1}M)}="4",
(90,0)*+{D_0(M)}="5",
(30,-10)*+{D_1(M)}="b2",
(50,-10)*+{D_1(M)}="b3",
(70,-10)*+{D_0(M/t^{\ell}M)}="b4",
(90,-10)*+{D_0(M)}="b5",
\ar"2";"3"^{t^{\ell-1}}
\ar"3";"4"
\ar"4";"5"
\ar"b2";"b3"^{t^{\ell}}
\ar"b3";"b4"
\ar"b4";"b5"
\ar"b2";"2"^{t}
\ar@{=}"b3";"3"
\ar"b4";"4"^t
\ar"b5";"5"^t
\endxy}
\]
Using the isomorphism $D_0(M)\cong D_0(\fl_RM)$, we obtain a commutative diagram of exact sequences:
\[
{\SelectTips{cm}{10}
\xy0;/r.3pc/:
(8,0)*+{0}="1",
(25,0)*+{R/(t^{\ell-1})\otimes_RD_1(M)}="2",
(50,0)*+{D_0(M/t^{\ell-1}M)}="3",
(70,0)*+{D_0(\fl_RM)}="4",
(8,-10)*+{0}="b1",
(25,-10)*+{R/(t^{\ell})\otimes_RD_1(M)}="b2",
(50,-10)*+{D_0(M/t^{\ell-1}M)}="b3",
(70,-10)*+{D_0(\fl_RM)}="b4",
\ar"1";"2"
\ar"2";"3"
\ar"3";"4"
\ar"b1";"b2"
\ar"b2";"b3"
\ar"b3";"b4"
\ar"b2";"2"^{1}
\ar"b3";"3"^{t}
\ar"b4";"4"^{t}
\endxy}
\]
Since the inverse limit of the right column is zero,
we have an isomorphism
\[
D_1(M)=\varprojlim_\ell R/(t^\ell)\otimes_RD_1(M)\cong \varprojlim_{\ell}D_0(M/t^\ell M)
\]
by taking the inverse limit of each column.
\end{proof}

\section{Application to Geometry: CY Reduction in $\Dsg(R)$}\label{Section3}

The aim of this section is to apply results in previous sections to CY triangulated categories appearing in geometry.
In \S\ref{CYredandMMAs} we relate CY reduction to our previous work on maximal modification algebras \cite{IW4}, then in \S\ref{one dime CY red} we give natural examples of CY reduction in the setting of one-dimensional hypersurfaces.  
We outline some of the consequences in \S\ref{conj section}.

\subsection{CY Reduction and MMAs}\label{CYredandMMAs}
Let $R$ be a commutative equi-codimensional Gorenstein ring with $\dim R=d$.
The functor $D_i:=\Ext^{d-i}_R(-,R):\mod R\to\mod R$ induces the duality of the category of Cohen-Macaulay $R$-modules of dimension $i$.
As an application of AR duality on not-necessarily-isolated singularities \cite{IW4}, we have the following.

\begin{thm}\label{nonisolatedAR} 
Let $R$ be a commutative equi-codimensional Gorenstein ring with $\dim R=d$ and $\dim\Sing R\leq 1$.
Then $\uCM R$ is a $(d-1)$-CY triangulated category with $\dim_R(\uCM R)\le 1$.
\end{thm}
\begin{proof}
Let $\cC=\uCM R$. Then the assumption $\dim\Sing R\leq 1$ implies that $\dim_R\Hom_{\cC}(X,Y)\leq 1$ for all $X,Y\in\cC$.  
By \cite[3.1]{IW4}, there exist functorial isomorphisms
\begin{align*}
D_0(\fl_R\Hom_{\cC}(X,Y))&\cong \fl_R\Hom_{\cC}(Y,X[d-1])\\
D_1\left(\frac{\Hom_{\cC}(X,Y)}{\fl_R\Hom_{\cC}(X,Y)}\right)&\cong
\frac{\Hom_{\cC}(Y,X[d-2])}{\fl_R\Hom_{\cC}(Y,X[d-2])}
\end{align*}
for all $X,Y\in \cC$. Thus the assertion follows.
\end{proof}

Before stating the next theorem, we need some preliminaries.  We consider the setup of \ref{modifying 2 ok}, where in addition we assume that $R$ is normal.  We fix $M\in\refl R$ which is non-zero, and we denote by $\refl\End_R(M)$ the category of $\End_R(M)$-modules which are reflexive as $R$-modules, and by $\CM\End_R(M)$ the category of $\End_R(M)$-modules which are maximal Cohen-Macaulay as $R$-modules.
Clearly we have $\refl\End_R(M)\subset\CM\End_R(M)$.
The following is a basic observation on the category of reflexive modules \cite{RV89} (see also \cite[2.4(2)(i)]{IR}).

\begin{prop}
For any $M\in\refl R$ which is non-zero, we have an equivalence
\begin{equation}\label{reflexive equivalence}
\Hom_R(M,-):\refl R\to\refl\End_R(M).
\end{equation}
\end{prop}

Thus the category $\refl\End_R(M)$ does not depend on the choice of $M$.
On the other hand, the category $\CM\End_R(M)$ strongly depends on the
choice of $M$. Actually the equivalence (\ref{reflexive equivalence}) clearly
induces an equivalence
\begin{equation}\label{CM equivalence}
\Hom_R(M,-):\{X\in\refl R\mid \Hom_R(M,X)\in\CM R\}\simeq\CM\End_R(M)
\end{equation}
and we have the following observation.

\begin{prop}\label{CM embedding}
For any generator $M\in\refl R$, the equivalence (\ref{CM equivalence}) gives a fully faithful functor
\[\CM\End_R(M)\to\CM R.\]
In particular we have the following embeddings:
\[\begin{array}{ccccc}
\CM R&\subset&\refl R&\subset&\mod R\\
\cup&&\parallel&&\cap\\
\CM\End_R(M)&\subset&\refl\End_R(M)&\subset&\mod\End_R(M)
\end{array}\]
\end{prop}

\begin{proof}
This is clear since $\CM\End_R(M)\simeq
\{X\in\refl R\mid \Hom_R(M,X)\in\CM R\}\subset\CM R$.
\end{proof}

Now we assume that $M$ belongs to $\CM R$ and is modifying in $\uCM R$.
The second condition is equivalent to $\End_R(M)\in\CM R$ by the following observation \cite[4.3, 4.4]{IW4}.

\begin{lemma}\label{modifying 2 ok}
Let $R$ be a commutative equi-codimensional Gorenstein ring with $\dim R=d\geq 2$ and $\dim\Sing R\leq 1$ and $M\in\CM R$.  Then\\
\t{(1)} $\End_R(M)\in\CM R$ if and only if $M\in\uCM R$ is modifying in the sense of \ref{modifying def}.\\
\t{(2)} If $M$ is modifying, then
\[\cZ_M=\{ X\in\CM R\mid \Hom_R(M,X)\in\CM R\}=\{ Y\in\CM R\mid \Hom_R(Y,M)\in\CM R\}.\] 
\end{lemma}

Moreover in this case $\CM\End_R(M)$ has a structure of a Frobenius category since
\[\CM\End_R(M)=\{ X\in\mod\End_R(M)\mid \Ext^{i}_{\End_R(M)}(X,\End_R(M))=0\mbox{ for all }i\geq 1\}\]
holds (see the proof of \cite[3.4(5)(i)]{IR}).
We denote by $\u{\CM}\End_R(M)$ the stable category, where we factor out by those morphisms which factor through projective $\End_R(M)$-modules.  

On the other hand we denote by $\Dsg(\End_R(M)):=\Db(\mod\End_R(M))/\Kb(\proj\End_R(M))$ the singular derived category.
Since $\End_R(M)$ has injective dimension $d$ on both sides (see the proof of \cite[3.1(6)(2)]{IR}), we have a triangle equivalence
\[
\Dsg(\End_R(M))\simeq \uCM\End_R(M).
\] 
by a standard theorem of Buchweitz \cite[4.4.1(2)]{Buch}.

The following gives an interpretation of $\uCM\End_R(M)$ as a CY reduction of $\uCM R$.

\begin{thm}\label{CY red is uCM of End}
Let $R$ be an equi-codimensional Gorenstein normal domain with $\dim R=d\geq 2$ and $\dim\Sing R\leq 1$,
and let $M\in\CM R$ be a modifying generator of $R$. Then\\
\t{(1)} the CY reduction $(\uCM R)_M$ of $\uCM R$ is $(d-1)$-CY with $\dim_R(\uCM R)_M\leq 1$.\\
\t{(2)} $(\uCM R)_M\simeq \uCM\End_R(M)$ as triangulated categories.
\end{thm}
\begin{proof}
(1) is an immediate consequence of \ref{CYreduction theorem} and \ref{nonisolatedAR}.\\
(2) Follows from $(\uCM R)_{M}=\cZ_{M}/[M]
\simeq\CM\End_R(M)/[\End_R(M)]=\uCM\End_R(M)$.
\end{proof}

This gives the following corollary, which allows us to check maximality in terms of the corresponding CY reduction.

\begin{cor}
Let $R$ be an equi-codimensional Gorenstein normal domain with $\dim R=d\geq 2$ and $\dim\Sing R\leq 1$, and $M\in\CM R$ be modifying.
Then $M$ is an MM generator of $R$ if and only if the corresponding CY reduction $(\uCM R)_M$ has no non-zero modifying objects.
\end{cor}
\begin{proof}
This is immediate from \ref{CT and CY reduction}.
\end{proof}

We end this subsection with the following iterated version of \ref{CY red is uCM of End}.

\begin{cor}\label{chainofCMs}
Let $R$ be an equi-codimensional Gorenstein normal domain with $\dim R=d\geq 2$ and $\dim\Sing R\leq 1$, and let $M=R\oplus M_1\oplus\hdots\oplus M_n$ be a modifying $R$-module.  
Let $M_0:=R$, $N_i:=\bigoplus_{j=0}^iM_j$ and $\Lambda_i:=\End_R(N_i)$.  Then\\
\t{(1)} There is a chain of fully faithful functors
\[
\CM\Lambda_n\to\CM\Lambda_{n-1}\to\cdots\to\CM\Lambda_1\to\CM R.
\]
\t{(2)} The CY reduction $(\uCM\Lambda_i)_{\Hom_R(N_i,N_{i+1})}$ of $\uCM\Lambda_i$ is triangle equivalent to $\uCM\Lambda_{i+1}$.
\end{cor}
\begin{proof}
(1) Clearly 
\[\cZ_{N_n}\subset\cZ_{N_{n-1}}\subset\cdots\subset\cZ_{N_1}\subset\CM R.\]
Applying the equivalence $\Hom_R(M,-):\cZ_{N_i}\simeq\CM\Lambda_i$ from (\ref{CM equivalence}) shows the assertion.\\
(2) The embedding $\CM\Lambda_i\simeq\cZ_{N_i}\subset\cZ_{N_{i+1}}\simeq
\CM\Lambda_{i+1}$ in (1) induces an equivalence
\[\cZ^{\Lambda_i}_{\Hom_R(N_i,N_{i+1})}:=
\{X\in\CM\Lambda_i\mid\Hom_{\Lambda_i}(\Hom_R(N_i,N_{i+1}),X)\in\CM R\}
\simeq\cZ_{N_{i+1}}\simeq\CM\Lambda_{i+1}\]
which sends $\Hom_R(N_i,N_{i+1})$ to $\Lambda_{i+1}$. Thus we have
\[(\uCM\Lambda_i)_{\Hom_R(N_i,N_{i+1})}=
\cZ^{\Lambda_i}_{\Hom_R(N_i,N_{i+1})}/[\Hom_R(N_i,N_{i+1})]
\simeq\CM\Lambda_{i+1}/[\Lambda_{i+1}]=\uCM\Lambda_{i+1}.\]
\end{proof}

\subsection{CY Reduction for One-Dimensional Hypersurfaces}\label{one dime CY red}

Let $S=k[[x,y]]$ be a formal power series ring of two variables over an arbitrary field $k$. For $f,g\in S$, let
\[R:=S/(fg)\]
be a one-dimensional hypersurface. 
Then $M\in\mod R$ is a CM $R$-module if and only if $\fl_RM=0$.
Our main result in this subsection is the following.

\begin{thm}\label{3.7}
With notation as above, \\
\t{(1)} $\uCM R$ is a $2$-CY triangulated category with $\dim_R(\uCM R)\le 1$.\\
\t{(2)} $S/(f)$ is a modifying object in $\uCM R$, and the CY reduction $(\uCM R)_{S/(f)}$ of $\uCM R$ is triangle equivalent to $\uCM(S/(f))\times\uCM(S/(g))$.
\end{thm}

We give the proof in the remainder of this subsection.
First we note that the natural surjections $R\to S/(f)$ and $R\to S/(g)$ induce fully faithful functors $\CM(S/(f))\to\CM R$ and $\CM(S/(g))\to\CM R$.

\begin{lemma}\label{criterion}
$X\in\CM R$ satisfies $\fl_R\Ext^1_R(S/(f),X)=0$ if and only if $X/fX\in\CM R$.
\end{lemma}

\begin{proof}
Applying $\Hom_R(-,X)$ to the exact sequence
\[
0\to S/(g)\xrightarrow{f}R\to S/(f)\to 0
\]
gives an exact sequence
\[
0\to \Hom_R(S/(f),X)\to X\xrightarrow{f\cdot} \Hom_R(S/(g),X)\to \Ext^1_R(S/(f),X)\to 0.
\]
In particular $\fl_R\Ext^1_R(S/(f),X)=0$ if and only if
$\frac{\Hom_R(S/(g),X)}{fX}\in \CM R$.
On the other hand, exchanging $f$ and $g$ in the above exact sequence,
we have $\frac{X}{\Hom_R(S/(g),X)}\in\CM R$ since $\Hom_R(S/(f),X)\in\CM R$ and $\CM R$ is closed under submodules.
Since we have an exact sequence
\[0\to\frac{\Hom_R(S/(g),X)}{fX}\to\frac{X}{fX}\to\frac{X}{\Hom_R(S/(g),X)}\to0\]
and $\CM R$ is closed under submodules and extensions,
we have that $\frac{\Hom_R(S/(g),X)}{fX}\in \CM R$ if and only if $X/fX\in \CM R$.
Thus the assertion follows.
\end{proof}

We also need the following easy observation, which is valid for any dimension.

\begin{lemma}\label{two MFs}
Let $A$ and $B$ be $n\times n$ matrices over $S$ such that $AB=fgI_n=BA$ and $X:={\rm Cok}(S^n\xrightarrow{A}S^n)$. Then the following conditions are equivalent.\\
\t{(1)} $X\in\CM(S/(f))$.\\
\t{(2)} There exists an $n\times n$ matrix $B'$ over $S$ such that $AB'=fI_n=B'A$.\\
\t{(3)} All entries in $B$ belongs to $(g)$.\\
If these conditions are satisfied, then $\fl_R\Ext^1_R(S/(f),X)=0=\fl_R\Ext^1_R(S/(g),\Omega_R(X))$.
\end{lemma}

\begin{proof}
(1)$\Rightarrow$(2) is clear since $A$ gives a matrix factorization of $f$.\\
(2)$\Rightarrow$(3) Since $A$ is invertible as a matrix over $k((x,y))$, we have $B=gB'$.\\
(3)$\Rightarrow$(1) is clear since we have matrix factorization $A(g^{-1}B)=fI_n=(g^{-1}B)A$.

Since $\fl_R(X/fX)=\fl_RX=0$ by (1), we have $\fl_R\Ext^1_R(S/(f),X)=0$ by \ref{criterion}.
Since $\Omega_R(X)=S^n/B(S^n)$, we have $\Omega_R(X)/g\Omega_R(X)=S^n/(B(S^n)+gS^n)=S^n/gS^n$ by (3). 
Thus $\fl_R(\Omega_R(X)/g\Omega_R(X))=0$ and we have $\fl_R\Ext^1_R(S/(g),\Omega_R(X))=0$ by \ref{criterion}.
\end{proof}

Let $\cZ_{S/(f)}:=\{X\in \CM R\ |\ \fl_R\Ext^1_R(S/(f),X)=0\}$.
Then the CY reduction $(\uCM R)_{S/(f)}$ is given by $\cZ_{S/(f)}/[R\oplus S/(f)]$. The following is a crucial step.

\begin{lemma}\label{describe Z}
We have $\cZ_{S/(f)}=\add\{R,Y,\Omega_R(Z)\mid Y\in \CM(S/(f)),\ Z\in \CM(S/(g))\}$.
\end{lemma}

\begin{proof}
The inclusion ``$\supseteq$'' follows from \ref{two MFs}. 
We shall show ``$\subseteq$''. 
Assume that $X\in\CM R$ satisfies $\fl_R\Ext^1_R(S/(f),X)=0$.
Take a minimal free resolution
\[
0\to S^n\xrightarrow{A}S^n\to X\to 0
\]
of the $S$-module $X$, where $A$ is an $n\times n$ matrix over $S$.
Then we have a free resolution
\begin{equation}\label{resolution of X/fX 1}
S^{2n}\xrightarrow{(A\,\,\, fI_n)}S^n\to X/fX\to 0
\end{equation}
of the $S$-module $X/fX$, where $I_n$ is the identity matrix of size $n$.
On the other hand $X/fX$ belongs to $\CM R$ by our assumption and \ref{criterion}.
Since $fX$ is contained in the radical of $X$, the minimal numbers of generators of $X$ and $X/fX$ are the same. Thus we have a minimal free resolution
\begin{equation}\label{resolution of X/fX 2}
0\to S^n\xrightarrow{B}S^n\to X/fX\to 0
\end{equation}
of the $S$-module $X/fX$, where $B$ is an $n\times n$ matrix over $S$.
Let $BC=fI_n=CB$ be the corresponding matrix factorization.
We write more explicitly
\[
B=\left(\begin{array}{cc}fI_m&O\\ O&B'\end{array}\right),\ 
C=\left(\begin{array}{cc}I_m&O\\ O&C'\end{array}\right)
\]
for some $m$ with $0\le m\le n$ where  all entries of $C'$ belong to $(x,y)$ and $B'C'=fI_{n-m}=C'B'$.
Since \eqref{resolution of X/fX 2} is minimal, we can obtain \eqref{resolution of X/fX 1} by adding a trivial summand and thus obtain a commutative diagram
\[
{\SelectTips{cm}{10}
\xy0;/r.37pc/:
(45,0)*+{S^{2n}}="2",
(60,0)*+{S^n}="3",
(70,0)*+{X/fX}="4",
(78,0)*+{0}="5",
(45,-9)*+{S^{2n}}="b2",
(60,-9)*+{S^n}="b3",
(70,-9)*+{X/fX}="b4",
(78,-9)*+{0}="b5",
\ar"2";"3"^{(A\,\,\, fI_n)}
\ar"3";"4"
\ar"4";"5"
\ar"b2";"b3"^{(B\,\,\, O)}
\ar"b3";"b4"
\ar"b4";"b5"
\ar"2";"b2"_{E=\scriptsize \begin{pmatrix} E_1&E_2\\ E_3&E_4 \end{pmatrix} }^\cong
\ar"3";"b3"^{\cong}_F
\ar@{=}"4";"b4"
\endxy}
\]
where the vertical maps are isomorphisms and $E_i$ ($1\le i\le 4$) is an $n\times n$ matrix over $S$.
Hence by replacing $B$ and $C$ by $BF^{-1}$ and $FC$ respectively, we can assume $F=I_n$.
Then we have $BE_1=A$ and $BE_2=fI_n$.  Since $BC=fI_n$ and $B$ is invertible as a matrix over $k((x,y))$, we have
\[
E_2=C=\left(\begin{array}{cc}I_m&O\\ O&C'\end{array}\right).
\]
Now we write $E_1$ as $E_1=\left(\begin{array}{cc}G_1&G_2\\ G_3&G_4\end{array}\right)$, where $G_1$ is an $m\times m$ matrix. Then the map
\[
\left( G_3\ G_4\ O\ C' \right): S^{2n}\to S^{n-m}
\]
given by the $n-m$ rows of the invertible matrix $E$ is a split epimorphism.
Since all entries of the right part $(O\ C')$ are in the unique maximal ideal $(x,y)$ of $S$, the left part $(G_3\ G_4): S^n\to S^{n-m}$ must be a split epimorphism.
Hence there exists an $n\times n$ invertible matrix $U$ such that $(G_3\ G_4)U=(O\ I_{n-m})$.
Then
\[
AU=BE_1U=\left(\begin{array}{cc}fI_m&O\\ O&B'\end{array}\right)\left(\begin{array}{cc}G_1&G_2\\ G_3&G_4\end{array}\right)U=\left(\begin{array}{cc}G'_1&G'_2\\ O&B'\end{array}\right)
\]
where all entries of $G'_1$ and $G'_2$ are in $(f)$.
Since $C'B'=fI_{n-m}$, the $n\times n$ invertible matrix $V:=\left(\begin{array}{cc}I_m&-f^{-1}G'_2C'\\ O&I_{n-m}\end{array}\right)$ over $S$ satisfies
\[
VAU=\left(\begin{array}{cc}G'_1&O\\ O&B'\end{array}\right).
\]
Since both $U$ and $V$ are invertible, we have that $X={\rm Cok}(S^n\stackrel{A}{\longrightarrow}S^n)$ is a direct sum of
${\rm Cok}(S^m\stackrel{G'_1}{\longrightarrow}S^m)$ and
${\rm Cok}(S^{n-m}\stackrel{B'}{\longrightarrow}S^{n-m})$.
Since all entries of $G_1^\prime$ are in $(f)$, the former belongs to $\Omega_R(\CM(S/(g)))$ by \ref{two MFs}(3)$\Rightarrow$(1).
Since $B'C'=fI_{n-m}$, the latter belongs to $\CM(S/(f))$ by \ref{two MFs}(2)$\Rightarrow$(1).
\end{proof}

\begin{lemma}\label{Hom in CY quotient}
\t{(1)} $\Hom_{(\uCM R)/[S/(f)]}(Y,Y')=\Hom_{\uCM(S/(f))}(Y,Y')$ for all $Y,Y'\in\CM(S/(f))$.\\
\t{(2)} $\Hom_{(\uCM R)/[S/(g)]}(Z,Z')=\Hom_{\uCM(S/(g))}(Z,Z')$ for all $Z,Z'\in\CM(S/(g))$.
\end{lemma}
\begin{proof}
(1) Let $Y,Y'\in\CM(S/(f))$. 
Since $\CM(S/(f))\to\CM R$ is fully faithful, it suffices to show that if a map $Y\to Y'$ factors through $\add R$, it also factors through $\add(S/(f))$.  
Consider an exact sequence
\[ 
0\to (f)\xrightarrow{b} R\xrightarrow{a}S/(f)\to 0
\]
We only have to show that any map $R\to Y'$ factors through $a$ (i.e. $a$ is a left $\CM(S/(f))$-approximation).
Applying $\Hom_R(-,Y')$, we have an exact sequence
\[
0\to \Hom_R(S/(f), Y')\xrightarrow{a\cdot}\Hom_R(R,Y')\xrightarrow{b\cdot}\Hom_R((f), Y')
\]
where we have $(b\cdot)=0$ since $Y'\in\CM(S/(f))$. Hence $(a\cdot)$ is an isomorphism and we are done.\\
(2) follows from (1) by swapping $f$ and $g$.
\end{proof}

Now we are ready to prove \ref{3.7}.
\begin{proof}
(1) By \ref{nonisolatedAR}, we have that $\uCM R$ is a $0$-CY triangulated category with $\dim_R(\uCM R)\le1$.
Since $R$ is a hypersurface, $[2]$ is isomorphic to the identity functor \cite{E80,Y}. Thus the assertion follows.\\
(2)  \emph{Step 1:}
For $Y\in\CM(S/(f))$ and $Z\in\CM(S/(g))$, we show $\Hom_{(\uCM R)/[S/(f)]}(Y,\Omega_R(Z))=0=\Hom_{(\uCM R)/[S/(f)]}(\Omega_R(Z),Y)$.
In particular, $(\uCM R)_{S/(f)}=\cZ_{S/(f)}/[R\oplus S/(f)]$ decomposes a product of a full subcategory consisting of objects in $\CM(S/(f))$ and that consisting of objects in $\Omega_R(\CM(S/(g))$ by \ref{describe Z}.

Let $0\to S^n\xrightarrow{A}S^n\to Y\to0$ and $0\to S^m\xrightarrow{A'}S^m\to Z\to0$ be free resolutions, and $AB=fI_n=BA$ and $A'B'=gI_m=B'A'$ be matrix factorizations.
For any $a\in\Hom_R(Y,\Omega_R(Z))$, there exist matrices $C$ and $D$ over $S$ which makes the following diagram commutative:
\[
{\SelectTips{cm}{10}
\xy0;/r.37pc/:
(23,0)*+{0}="1",
(30,0)*+{S^n}="2",
(40,0)*+{S^n}="3",
(50,0)*+{Y}="4",
(57,0)*+{0}="5",
(23,-7)*+{0}="b1",
(30,-7)*+{S^m}="b2",
(40,-7)*+{S^m}="b3",
(50,-7)*+{\Omega_R(Z)}="b4",
(57,-7)*+{0}="b5",
\ar"1";"2"
\ar"2";"3"^A
\ar"3";"4"
\ar"4";"5"
\ar"b1";"b2"
\ar"b2";"b3"^{fB'}
\ar"b3";"b4"
\ar"b4";"b5"
\ar"2";"b2"^D
\ar"3";"b3"^{C}
\ar"4";"b4"^{a}
\endxy}
\]
Multiplying $B$ to the equality $AC=fDB'$ from the left, we have $C=BDB'$. Thus we have a commutative diagram
\[
{\SelectTips{cm}{10}
\xy0;/r.37pc/:
(23,0)*+{0}="1",
(30,0)*+{S^n}="2",
(40,0)*+{S^n}="3",
(50,0)*+{Y}="4",
(60,0)*+{0}="5",
(23,-7)*+{0}="b1",
(30,-7)*+{S^n}="b2",
(40,-7)*+{S^n}="b3",
(50,-7)*+{(S/(f))^n}="b4",
(60,-7)*+{0}="b5",
(23,-14)*+{0}="c1",
(30,-14)*+{S^m}="c2",
(40,-14)*+{S^m}="c3",
(50,-14)*+{\Omega_R(Z)}="c4",
(60,-14)*+{0}="c5",
\ar"1";"2"
\ar"2";"3"^A
\ar"3";"4"
\ar"4";"5"
\ar"b1";"b2"
\ar"b2";"b3"^{fI_n}
\ar"b3";"b4"
\ar"b4";"b5"
\ar"c1";"c2"
\ar"c2";"c3"^{fB'}
\ar"c3";"c4"
\ar"c4";"c5"
\ar@{=}"2";"b2"
\ar"3";"b3"^{B}
\ar"4";"b4"
\ar"b2";"c2"^{D}
\ar"b3";"c3"^{DB'}
\ar"b4";"c4"
\endxy}
\]
which shows that $a$ factors through $\add(S/(f))$.

For any $b\in\Hom_R(\Omega_R(Z),Y)$, there exist matrices $C$ and $D$ over $S$ which makes the following diagram commutative:
\[
{\SelectTips{cm}{10}
\xy0;/r.37pc/:
(23,0)*+{0}="1",
(30,0)*+{S^m}="2",
(40,0)*+{S^m}="3",
(50,0)*+{\Omega_R(Z)}="4",
(57,0)*+{0}="5",
(23,-7)*+{0}="b1",
(30,-7)*+{S^n}="b2",
(40,-7)*+{S^n}="b3",
(50,-7)*+{Y}="b4",
(57,-7)*+{0}="b5",
\ar"1";"2"
\ar"2";"3"^{fB'}
\ar"3";"4"
\ar"4";"5"
\ar"b1";"b2"
\ar"b2";"b3"^{A}
\ar"b3";"b4"
\ar"b4";"b5"
\ar"2";"b2"^D
\ar"3";"b3"^{C}
\ar"4";"b4"^{\psi}
\endxy}
\]
Multiplying $B$ to the equality $fB'C=DA$ from the right, we have $B'CB=D$.
Thus we have a commutative diagram
\[
{\SelectTips{cm}{10}
\xy0;/r.37pc/:
(23,0)*+{0}="1",
(30,0)*+{S^m}="2",
(40,0)*+{S^m}="3",
(50,0)*+{\Omega_R(Z)}="4",
(60,0)*+{0}="5",
(23,-7)*+{0}="b1",
(30,-7)*+{S^n}="b2",
(40,-7)*+{S^n}="b3",
(50,-7)*+{(S/(f))^n}="b4",
(60,-7)*+{0}="b5",
(23,-14)*+{0}="c1",
(30,-14)*+{S^n}="c2",
(40,-14)*+{S^n}="c3",
(50,-14)*+{Y}="c4",
(60,-14)*+{0}="c5",
\ar"1";"2"
\ar"2";"3"^{fB'}
\ar"3";"4"
\ar"4";"5"
\ar"b1";"b2"
\ar"b2";"b3"^{fI_n}
\ar"b3";"b4"
\ar"b4";"b5"
\ar"c1";"c2"
\ar"c2";"c3"^{A}
\ar"c3";"c4"
\ar"c4";"c5"
\ar"2";"b2"_{B'C}
\ar"3";"b3"^C
\ar"4";"b4"
\ar"b2";"c2"_B
\ar@{=}"b3";"c3"
\ar"b4";"c4"
\endxy}
\]
which shows that $b$ factors through $\add(S/(f))$.

\noindent\emph{Step 2:} 
Clearly $\Hom_{(\uCM R)_{S/(f)}}(Y,Y')=\Hom_{\uCM(S/(f))}(Y,Y')$ holds for all $Y,Y'\in\CM(S/(f))$ by \ref{Hom in CY quotient}(1).
It remains to show $\Hom_{(\uCM R)_{S/(f)}}(\Omega_R(Z),\Omega_R(Z'))=\Hom_{\uCM(S/(g))}(Z,Z')$ for all $Z,Z'\in\CM(S/(g))$.
Since $\Omega_R$ gives an equivalence $[-1]:\uCM R\to\uCM R$ and $\Omega_R(S/(f))=S/(g)$, we have
\[\Hom_{(\uCM R)/[S/(f)]}(\Omega_R(Z),\Omega_R(Z'))\cong\Hom_{(\uCM R)/[S/(g)]}(Z,Z').\]
This equals $\Hom_{\uCM(S/(g))}(Z,Z')$ by \ref{Hom in CY quotient}(2). Thus the assertion follows.
\end{proof}

\subsection{General Remarks and Conjectures}\label{conj section}

The concept of CY reduction has been invented as an algebraic tool for proving statements regarding modifying and maximal modifying modules on the base singularity $\Spec R$.  
There is now a conjectural geometric picture underlying this theory, and the following is a slightly weaker version of \cite[Conj 1.8]{IW5}.
\begin{conj}\label{main conj}
Let $R$ be a $3$-dimensional Gorenstein normal normal domain over $\mathbb{C}$ with rational singularities, so $\uCM R$ is a $2$-CY triangulated category with $\dim_R(\uCM R)\leq 1$.
Then there exists a CY reduction $(\uCM R)_M$ of $\uCM R$ with $\dim_R(\uCM R)_M=0$, and further $(\uCM R)_M$  has no non-zero rigid objects.
\end{conj}
This is somewhat remarkable, since in this level of generality $\uCM R$ is a not Krull--Schmidt, and has many modifying objects.
Yet it still will admit an extremely well-behaved CY reduction.
The best case scenario is when $(\uCM R)_M=0$, which is equivalent to there existing an NCCR of $R$.

We remark that the conjecture is true in quite a broad setting:
\begin{thm}\label{tilting complex gives}
Let $R$ be a $3$-dimensional Gorenstein normal normal domain over $\mathbb{C}$ with rational singularities. 
If some $\mathds{Q}$-factorial terminalization $Y$ of $\Spec R$ is derived equivalent to some ring $\Lambda$,
then there exists an MM generator $M\in\CM R$ of $R$ such that $(\uCM R)_M$ is triangle equivalent to $\Dsg(Y)$.
In particular, Conjecture \ref{main conj} is true.
\end{thm}
\begin{proof}
Let $f\colon Y\to\Spec R$ denote the $\mathds{Q}$-factorial terminalization which is derived equivalent to $\Lambda$.  
By \cite[4.13]{IW5} $\Lambda\cong\End_R(N)$ for some reflexive $R$-module $N$ which is an MM $R$-module.
By \cite[4.18(2)]{IW4} there exists an MM generator $M\in\CM R$ of $R$.
Since all MMAs are derived equivalent in dimension three \cite[4.16]{IW4}, we have
\[
\Db(\mod\End_R(M))\simeq\Db(\mod\End_R(N))\simeq\Db(\coh Y)
\]
which after factoring by perfect complexes gives
\[
(\uCM R)_M\stackrel{\scriptsize\mbox{\ref{CY red is uCM of End}}}{\simeq}\uCM\End_R(M)\simeq\Dsg(\End_R(M))\simeq\Dsg(Y)\hookrightarrow\bigoplus_{i=1}^{n}\uCM\c{O}_{X,x_i}
\]
where $\{ x_i\mid 1\leq i\leq n\}$ are the (necessarily isolated) singular points of $Y$ \cite[3.7]{IW5}.  
Thus $\dim_R(\uCM R)_M\le\dim_R(\bigoplus_{i=1}^{n}\uCM\c{O}_{X,x_i})=0$ holds.
By \cite[3.11]{IW5}, each $\uCM\c{O}_{X,x_i}$ has no non-zero rigid objects, hence the same is true for $(\uCM R)_M$.
\end{proof}
\begin{cor}\label{1d fibres gives}
Let $R$ be a $3$-dimensional Gorenstein normal domain over $\mathbb{C}$ with rational singularities. If the $\mathds{Q}$-factorial terminalizations of $\Spec R$ have one-dimensional fibres, then Conjecture \ref{main conj} is true.
\end{cor}
\begin{proof}
By \cite{VdB1d} the $\mathds{Q}$-factorial terminalizations carry a tilting bundle, so the result follows from \ref{tilting complex gives}.
\end{proof}

\section{Mutation}\label{Mutation Section}

\subsection{Result on Transitivity}

In this section we recall the notion of mutation of modifying modules and their basic properties given in \cite[Section 6.2]{IW4}, then give a method to prove when a given set of MM generators are all.  In \S\ref{Section cAn} we will apply this result together with the techniques of CY reduction developed in the previous sections to classify all MM generators over certain explicit singularities.

Throughout the section we assume that $R$ is a complete local normal Gorenstein domain with $\dim R=3$.  
\emph{Mutation} is an operation for modifying $R$-modules which gives a new modifying $R$-module for a given basic modifying $R$-module by replacing a direct summand of $M$. 
We recall how this is defined \cite[\S6]{IW4}.

We let $M:=\bigoplus_{i\in I}M_{i}$ be a modifying $R$-module, where we can (and will) assume that $M$ is basic, i.e.\ all summands are pairwise non-isomorphic.  
We denote $\Hom_{R}(-,R):=(-)^{*}:\refl R\to \refl R$ to be the duality functor.
For a subset $J$ of $I$, set $M_{J}:=\bigoplus_{j\in J}M_{j}$ and $J^c:=I\backslash J$. Thus we have $M_J\oplus M_{J^c}=M$. 
Now we take a \emph{minimal right $(\add M_{J^c})$-approximation}
\[N\xrightarrow{f}M_J\]
of $M_J$, which means that
\begin{itemize}
\item $N\in\add M_{J^c}$ and $(\cdot f)\colon\Hom_R(M_{J^c},N)\to\Hom_R(M_{J^c},M_J)$ is surjective,
\item if $g\in\End_R(N)$ satisfies $f=gf$, then $g$ is an automorphism.
\end{itemize}
Since $R$ is complete, such an $f$ exists and is unique up to isomorphism.
A \emph{right mutation} of $M$ is defined as
\[\mu_J^+(M):=M_{J^c}\oplus\Ker f.\]
Dually we define a \emph{left mutation} of $M$ as
\[\mu_J^-(M):=(\mu_J^+(M^*))^*.\]
Below we collect basic properties.
\begin{prop}\cite[6.10,\ 6.25]{IW4} 
\t{(1)} $\mu_J^+(M)$ and $\mu_J^-(M)$ are modifying $R$-modules and satisfy $\mu^+_J(\mu_J^-(M))\cong M\cong \mu^-_J(\mu_J^+(M))$.\\
\t{(2)} If $M$ is an MM (respectively, CT) $R$-module, then so are $\mu_J^+(M)$ and $\mu_J^-(M)$.\\
\t{(3)} If $M$ is an MM (respectively, CT) $R$-module, and $J=\{i\}$, then we have $\mu_J^+(M)\cong\mu_J^-(M)$, which we denote by $\mu_i(M)$.
\end{prop}
It is immediate from (1) and (3) above that if $M$ is an MM $R$-module, then $\mu_i(\mu_i(M))\cong M$ holds.  The following is the main result of \cite[\S6]{IW4}.

\begin{thm}\label{welldefined}\cite[6.8]{IW4}
Let $R$ be a complete local $d$-dimensional Gorenstein normal normal domain, let $M:=\bigoplus_{i\in I}M_{i}$ be a basic modifying $R$-module and choose any $\emptyset\neq J\subseteq I$.
Then $\End_R(M)$ and $\End_R(\mu_J^{\pm}(M))$ are derived equivalent.
\end{thm}

For the case $\mu_J^-(M)$, the derived equivalence is given by a tilting $\End_R(M)$-module $V_J$ constructed as follows.
Let $g\colon M_J\to N'$ be a minimal left $(\add M_{J^c})$-approximation of $M_J$.
For the induced map $(\cdot g)\colon\Hom_R(M,M_J)\to\Hom_R(M,N')$, our tilting $\End_R(M)$-module is given by
\[
V_J:=\Hom_R(M,M_{J^c})\oplus\Cok(\cdot g).
\]
This gives rise to an equivalence
\[
\RHom(V_J,-)\colon\Db(\mod\End_{R}(M))\to \Db(\mod\End_{R}(\mu_J^-(M)))
\]
but note that this functor is never the identity. On the other hand $\mu_J^-(M)=M$ can happen (see \S\ref{Section cAn}).

\subsection{MM Mutation and Tilting Mutation}

In the rest of this section, we specialize the previous setting to the case when modifying modules are \emph{MM generators}.  
Moreover we mutate them at \emph{non-free indecomposable} summands.
We denote by $\MMG R$ the set of isomorphism classes of basic MM generators of $R$.
Thus for a given basic MM generator $M=R\oplus(\bigoplus_{i\in I}M_i)$, we have a new MM generator $\mu_i(M)$ by replacing an indecomposable non-free direct summand $M_i$ of $M$.
We denote by $\Gr(\MMG R)$ the \emph{exchange graph}, thus the set of vertices is $\MMG R$, and we draw an edge between $M$ and $\mu_i(M)$ for each $M\in\MMG R$ and $i\in I$.

One of the difficulties in mutation for MM generators is that $\mu_i(M)$ can be isomorphic to $M$, which never happens in mutation in 2-CY triangulated categories $\cC$ with $\dim\cC=0$.  It is shown in \cite[1.25(1)(2)]{IW4} that $\mu_i(M)$ is isomorphic to $M$ if and only if the algebra $\End_R(M)/(1-e_i)$ is not artinian.  In this case we have a loop at $M$ in $\Gr(\MMG R)$.

The aim of this section is to prove the following result, which is an analogue of \cite[4.9]{AIR} for 2-CY triangulated categories.  However, due to the existence of loops in $\Gr(\MMG R)$, we need a more careful argument.

\begin{thm}\label{Auslander type}
If $\Gr(\MMG R)$ has a finite connected component $C$, then $\Gr(\MMG R)=C$.
\end{thm}

To prove \ref{Auslander type} requires some preparation.  
Fix an MM generator $M_0\in\CM R$ and $\Lambda:=\End_R(M_0)$.
Then the functor
\[F:=\Hom_R(M_0,-)\colon\mod R\to\mod\Lambda\]
is fully faithful and (since $R$ is normal) induces an equivalence $F\colon\refl R\to\refl\Lambda$, where we denote by $\refl\Lambda$ the full subcategory of $\mod\Lambda$ consisting of modules that are reflexive as $R$-modules.
Recall we say that $T\in\mod\Lambda$ is a \emph{tilting $\Lambda$-module} if
\begin{itemize}
\item $\pd_\Lambda T\le 1$, 
\item $\Ext^1_\Lambda(T,T)=0$,
\item there exists an exact sequence $0\to\Lambda\to T^0\to T^1\to0$ with $T^0,T^1\in\add T$.
\end{itemize}
We denote $\Fac T$ to be the full subcategory of $\mod\Lambda$ consisting of factor modules of finite direct sum of copies of $T$.  One important property of tilting modules is
\begin{eqnarray}
\Fac T=\{  X\in\mod\Lambda\mid \Ext^1_\Lambda(T,X)=0\}.\label{Fac T}
\end{eqnarray}
In particular for any $X\in\Fac T$ there is an exact sequence
\begin{eqnarray}
0\to Y\to T^\prime\to X\to 0\label{Fac T 2}
\end{eqnarray}
with $Y\in\Fac T$ and $T^\prime\in\add T$.  From this we immediately get
\begin{eqnarray}
\add T=\{  X\in\Fac T\mid \Ext^1_\Lambda(X,\Fac T)=0\}.\label{Fac T 3}
\end{eqnarray}

It is shown in \cite[1.19]{IW4} that $F$ gives an injective map
\[F\colon\MMG R\to\tilt\Lambda,\]
where we denote by $\tilt\Lambda$ the set of isomorphism classes of basic tilting $\Lambda$-modules.

The main ingredient of the proof of \ref{Auslander type} is tilting mutation theory initiated by Riedtmann--Schofield and Happel--Unger \cite{RS,HU}.
We refer to \cite{AI} for a general treatment of tilting mutation.  
Recall that $\tilt\Lambda$ has a natural structure of partially ordered set: We write $T\ge U$ if $\Ext^1_\Lambda(T,U)=0$, or equivalently by \eqref{Fac T} $\Fac T\subseteq \Fac U$.  It is immediate from \eqref{Fac T 3} that $T\geq U\geq T$ implies that $T\cong U$.

On the other hand, for a basic tilting $\Lambda$-module $T$ and an indecomposable direct summand $T_i$ of $T$,
there exists at most one basic tilting $\Lambda$-module $\mu_i(T)=(T/T_i)\oplus T^*_i$ such that $T_i\ {\not\cong}\ T^*_i$ (cf. \cite{RS}).
We call $\mu_i(T)$ a \emph{tilting mutation} of $T$. In this case, we have either an exact sequence 
\[
0\to T_i\xrightarrow{f}T'\to T^*_i\to 0
\]
with a minimal left $(\add T/T_i)$-approximation $f$, or an exact sequence
\[
0\to T^*_i\to T'\xrightarrow{g}T_i\to 0
\]
with a minimal right $(\add T/T_i)$-approximation $g$.
We have $T>\mu_i(T)$ in the former case, and $T<\mu_i(T)$ in the latter case. Conversely $T_i^*$ obtained from one of the above sequences gives $\mu_i(T)$ if $T_i^*$ has projective dimension at most one (see e.g.\ \cite[5.2]{IR}).

We denote by $\Gr(\tilt\Lambda)$ the \emph{exchange graph} of tilting $\Lambda$-modules,
i.e. the set of vertices is $\tilt\Lambda$ and we draw an edge between $T$ and $\mu_i(T)$ for all $T\in\tilt\Lambda$ and $i$ such that $\mu_i(T)$ exists.

We prepare the following results.

\begin{prop}\label{make closer}
Assume that $T,U\in\tilt\Lambda$ satisfies $T>U$.\\
\t{(1)} There exists an exact sequence $0\to T^{\prime\prime}\to T^{\prime}\to U\to0$ with $T^{\prime},T^{\prime\prime}\in\add T$ and $\add T^\prime\cap \add T^{\prime\prime}=0$.\\
\t{(2)} There exists an exact sequence $0\to T\to U^\prime\to U^{\prime\prime}\to0$ with $U^\prime,U^{\prime\prime}\in\add U$ and $\add U^\prime\cap \add U^{\prime\prime}=0$.\\
\t{(3)} There exists a tilting mutation $T'$ of $T$ such that $T>T'\ge U$.\\
\t{(4)} There exists a tilting mutation $U'$ of $U$ such that $T\ge U'>U$.
\end{prop}
\begin{proof}
Although these results follows easily from \cite{AI}, we give a direct proof for the convenience of the reader.\\
(1) Applying \eqref{Fac T 2} to $X:=U$, we have an exact sequence $0\to Y\to T^\prime\to U\to0$ with $T^\prime\in\add T$ and $Y\in\Fac T$.
Applying $\Ext^1_\Lambda(-,\Fac T)$, we have an exact sequence
$0=\Ext^1_\Lambda(T',\Fac T)\to \Ext^1_\Lambda(Y,\Fac T)
\to\Ext^2_\Lambda(U,\Fac T)=0$
since $\pd_\Lambda U\leq 1$. Thus $\Ext^1_\Lambda(Y,\Fac T)=0$ holds and we have $Y\in\add T$ by \eqref{Fac T 3}. Replacing the map $T^\prime\to U$ by a minimal right $(\add T)$-approximation gives the last statement (see e.g.\ \cite[2.25]{AI}).\\
(2) We regard $\Ext^1_\Lambda(U,T)$ as an $\End_\Lambda(U)$-module. Take a surjective map $f:\Hom_\Lambda(U,U^{\prime\prime})\to\Ext^1_\Lambda(U,T)$  of $\End_\Lambda(U)$-modules with $U^{\prime\prime}\in\add U$, then this gives an exact sequence
\begin{equation}\label{universal extension}
0\to T\to X\to U^{\prime\prime}\to 0.
\end{equation}
Applying $\Hom_\Lambda(U,-)$ to (\ref{universal extension}),
we have an exact sequence
\[
\Hom_\Lambda(U,U^{\prime\prime})\xrightarrow{f}\Ext^1_\Lambda(U,T)\to\Ext^1_\Lambda(U,X)
\to\Ext^1_\Lambda(U,U^{\prime\prime})=0
\]
which shows $\Ext^1_\Lambda(U,X)=0$ and hence $X\in\Fac U$ by \eqref{Fac T}. Applying $\Ext^1_\Lambda(-,\Fac U)$ to (\ref{universal extension}), we have
$\Ext^1_\Lambda(X,\Fac U)=0$. Thus $X\in\add U$ by \eqref{Fac T 3}.  By a similar argument as in (1), the last statement also follows. \\
(3) Take an exact sequence $0\to T''\xrightarrow{b} T'\xrightarrow{a} U\to0$ from (1). Since $T>U$, we have $T^{\prime\prime}\neq0$.
Take an indecomposable direct summand $T_i$ of $T''$, and
let $\iota:T_i\to T''$ be the inclusion.
Let $f:T_i\to V$ be a minimal left $(\add T/T_i)$-approximation of $T_i$. Since $T'$ and $T''$ have no non-zero common direct
summands, $T'\in\add T/T_i$ holds, and $\iota b$ factors through $f$. Hence $f$ has to be injective, and so it only remains to prove
$\Cok f$ has projective dimension at most one.
Since we have a commutative diagram
\[
{\SelectTips{cm}{10}
\xy0;/r.37pc/:
(23,0)*+{0}="1",
(30,0)*+{T_i}="2",
(40,0)*+{V}="3",
(50,0)*+{\Cok f}="4",
(57,0)*+{0}="5",
(23,-7)*+{0}="b1",
(30,-7)*+{T^{\prime\prime}}="b2",
(40,-7)*+{T^\prime}="b3",
(50,-7)*+{U}="b4",
(57,-7)*+{0}="b5",
\ar"1";"2"
\ar"2";"3"^f
\ar"3";"4"
\ar"4";"5"
\ar"b1";"b2"
\ar"b2";"b3"^{b}
\ar"b3";"b4"^a
\ar"b4";"b5"
\ar"2";"b2"^\iota
\ar"3";"b3"
\ar"4";"b4"
\endxy}
\]
of exact sequences, we have an exact sequence
\[
0\to V\oplus(T''/T_i)\to \Cok f\oplus T'\to U\to0,
\]
which immediately implies $\pd_\Lambda\Cok f\le 1$.\\
The proof of (4) is simpler.
\end{proof}

The following comparison between MM mutation and tilting mutation is important.

\begin{lemma}\label{compatibility of mutation}
Let $M=R\oplus(\bigoplus_{i\in I}M_i)\in\MMG R$.\\
\t{(1)} If $\mu_i(M)\ {\not\cong}\ M$, then $F(\mu_i(M))=\mu_i(F(M))$.\\
\t{(2)} If $\mu_i(F(M))$ exists and belongs to $\refl\Lambda$, then $\mu_i(M)\ {\not\cong}\ M$ and $F(\mu_i(M))=\mu_i(F(M))$.
\end{lemma}

\begin{proof}
(1) Since $M$ and $\mu_i(M)$ differ at only the $i$-th indecomposable summand, so do $F(M)$ and $F(\mu_i(M))$. Thus $\mu_i(F(M))=F(\mu_i(M))$ holds by definition of tilting mutation.\\
(2) Assume that $\mu_i(F(M)):=F(M/M_i)\oplus X$ belongs to $\refl\Lambda$.
Then we have either an exact sequence 
\begin{equation}\label{left mutation}
0\to F(M_i)\xrightarrow{f}F(M')\xrightarrow{g} X\to 0
\end{equation}
with a minimal left $\add F(M/M_i)$-approximation $f$, or an exact sequence
\begin{equation}\label{right mutation}
0\to X\xrightarrow{f}F(M')\xrightarrow{g}F(M_i)\to 0
\end{equation}
with a minimal right $\add F(M/M_i)$-approximation $g$.

We consider the case when the sequence \eqref{left mutation} exists.
Since $F\colon\refl R\to\refl\Lambda$ is an equivalence and $X$ is reflexive by our assumption,
there exists $Y\in\refl R$ and a complex
\begin{equation}\label{left mutation 2}
0\to M_i\xrightarrow{a}M'\xrightarrow{b}Y\to 0
\end{equation}
of $R$-modules for which $F$(\ref{left mutation})=(\ref{left mutation 2}). Since the image of \eqref{left mutation} by the functor $F=\Hom_R(M_0,-)$ is the exact sequence \eqref{left mutation 2} and $M_0$ is a generator,  the sequence \eqref{left mutation} must also be exact. 
Since $F\colon\refl R\to\refl\Lambda$ is an equivalence, $a$ has to be a minimal left $(\add M/M_i)$-approximation of $M_i$.  Thus $\mu_i(M)=(M/M_i)\oplus Y$ and we have $F(\mu_i(M))=\mu_i(F(M))$. 
Since the tilting mutation $\mu_i(F(M))$ is never isomorphic to $F(M)$ by the partial order, we have $\mu_i(M)\ {\not\cong}\ M$.

The same argument works when the sequence \eqref{right mutation} exists.
\end{proof}

Now we denote by $S$ the subset of $\tilt\Lambda$ consisting of tilting $\Lambda$-modules which belongs to $\refl\Lambda$ and have $F(R)$ as a direct summand.  By \cite[1.19]{IW4}
the following observation is clear.

\begin{prop}\label{from MMG to S}
$F$ gives an injective map $F\colon\MMG R\to S$.
\end{prop}

We need the following property of $S$ with respect to the partial order on $\tilt\Lambda$.

\begin{lemma}\label{S is a poset ideal}
If $T\in\tilt\Lambda$ and $U\in S$ satisfies $T\ge U$, then $T\in S$.
\end{lemma}

\begin{proof}
Since $T\ge U$, by \ref{make closer}(1) and (2), there exists exact sequences
\[
0\to T\to U^0\to U^1\to 0\ \ \ \mbox{ and }\ \ \ 
0\to T_1\to T_0\to U\to 0
\]
with $U^i\in\add U$ and $T_i\in\add T$ for $i=0,1$.
Since $U^i$ is a reflexive $R$-module, so is $T$ by the first sequence.
Since $F(R)\in\add U$ and $F(R)$ is a projective $\Lambda$-module, we have $F(R)\in\add T$ from the second sequence.
Thus we have $T\in S$.
\end{proof}

We denote by $\Gr(S)$ the full subgraph of $\Gr(\tilt\Lambda)$ with the set $S$ of vertices. 
The following is a main step in the proof.

\begin{prop}\label{F(C) is a component}
Let $C$ be a connected component of $\Gr(\MMG R)$. Then $F(C)$ is a connected component of $\Gr(S)$.
\end{prop}

\begin{proof}
By \ref{compatibility of mutation}(1), $F(C)$ is contained in some connected component of $\Gr(S)$.
Thus we only have to show that if two vertices $T$ and $U$ ($T\neq U$) in $\Gr(S)$ are connected by an edge and $T\in F(C)$, then $U\in F(C)$.
Now $T$ has a direct summand $F(R)$, so we can write $T=F(M)$ with $M=R\oplus(\bigoplus_{i\in I}M_i)\in C$.
Since $U$ has a direct summand $F(R)$, we have $U=\mu_i(T)$ for some $i\in I$.
But $U$ belongs to $\refl\Lambda$, so $U=F(\mu_i(M))$ by \ref{compatibility of mutation}(2). Since $\mu_i(M)\in C$, we have $U\in F(C)$, and so the assertion follows.
\end{proof}

The following simple criterion for connectedness of $\Gr(S)$  generalizes \cite[2.2]{HU}.

\begin{prop}\label{Auslander type 2}
Suppose that $S$ is a subset of $\tilt\Lambda$ that satisfies the property: if $T\in\tilt\Lambda$ and $U\in S$ satisfies $T\ge U$, then $T\in S$.  Then whenever $\Gr(S)$ has a finite connected component $C$, necessarily $\Gr(S)=C$.
\end{prop}

\begin{proof}
We can assume that $C$ is non-empty. Fix $T\in C$.
Since $\Lambda\ge T$, we have $\Lambda\in S$.
Applying \ref{make closer}(3) repeatedly, we have a sequence
\[
T=T_0<T_1<T_2<\hdots
\]
such that $T_{i+1}$ is a tilting mutation of $T_i$ for all $i$.
This sequence has to be finite since each $T_i$ belongs to $S$ by \ref{S is a poset ideal} and hence belongs to the finite connected component $C$.
Thus $T_\ell=\Lambda$ holds for some $\ell$, and in particular $\Lambda$ belongs to $C$.

Now fix any $U\in S$.
Applying \ref{make closer}(4) repeatedly, we have a sequence
\[
\Lambda=V_0>V_1>V_2>\hdots
\]
such that $V_{i+1}$ is a tilting mutation of $V_i$ and $V_i\ge U$ for all $i$.
This sequence has to be finite since each $V_i$ belongs to $S$ by \ref{S is a poset ideal} and hence belongs to the finite connected component $C$.
Thus $V_m=U$ holds for some $m$, and in particular $U$ belongs to $C$.
Hence we have $\Gr(S)=C$.
\end{proof}

Now we are ready to prove \ref{Auslander type}.
\begin{proof}
We know that $F(C)$ is a finite connected component of $\Gr(S)$ by \ref{F(C) is a component}.
Applying \ref{S is a poset ideal} and \ref{Auslander type 2} gives $F(C)=\Gr(S)$.
Since $F\colon\MMG R\to S$ is injective by \ref{from MMG to S}, it follows that $C=\Gr(\MMG R)$.
\end{proof}

\section{Complete Local $cA_n$ Singularities}\label{Section cAn}

In this subsection we fix notation for complete local $cA_n$ singularities, as they will be used throughout. We work over an arbitrary field $k$, let $S=k[[x,y]]$ be a formal power series ring and fix $f\in(x,y)$. Let
\[
R:=S[[u,v]]/(f(x,y)-uv)
\]
and $f=f_{1}\hdots f_{n}$ be a factorization into prime elements of $S$.  For any subset $I\subseteq\{ 1,\hdots, n\}$ set $I^c=\{ 1,\hdots, n\}\backslash I$ and denote
\[
f_I:=\prod_{i\in I}f_i\ \mbox{ and }\ T_I:=(u,f_I)
\]
where $T_I$ is an ideal of $R$ of generated by $u$ and $f_I$. Since we have the equality
\begin{equation}\label{replace u and v}
(u,f_I)=(v,f_{I^c})uf_{I^c}^{-1},
\end{equation}
we have $T_I\cong(v,f_{I^c})$. For a collection of subsets $\emptyset\subsetneq I_1\subsetneq I_2\subsetneq\hdots\subsetneq
I_m\subsetneq\{1,2,\hdots,n\}$, we say that $\F=(I_1,\hdots,I_m)$ is a {\em flag in the set $\{ 1,2,\hdots, n\}$}.  We say that the flag $\c{F}$ is {\em maximal} if $n=m+1$.  Given a flag $\c{F}=(I_1,\hdots,I_m)$, we define
\[
T^\c{F}:=R\oplus\left(\bigoplus_{j=1}^{m} T_{I_j}\right) .
\]
So as to match our notation with \cite{BIKR} and \cite{DH}, we can (and do) identify maximal flags with elements of the symmetric group $\mathfrak{S}_n$. Hence we regard each $\omega\in\mathfrak{S}_n$ as the maximal flag 
\[
\{\omega(1) \}\subset\{\omega(1),\omega(2) \}\subset\hdots\subset\{\omega(1),\hdots,\omega(n-1)\}.
\]
We denote
\[
T^\omega:=R\oplus\left(\bigoplus_{j=1}^{n-1} T_{\{\omega(1),\hdots,\omega(j)\}}\right)
\]

\subsection{MM Generators and CT Modules}\label{MM and CT}

The aim of this subsection is to use CY reduction to help to  prove the following result.
\begin{thm}\label{classification}
With the setup as above,\\
\t{(1)} The basic modifying generators of $R$ are precisely $T^{\F}$, where $\F$ is a flag in $\{1,2,\hdots,n\}$.\\
\t{(2)} The basic MM generators of $R$ are precisely $T^\omega$, where $\omega\in\mathfrak{S}_n$.\\
\t{(3)} $R$ has a CT module if and only if $f_i\notin\m^2$ for all $1\le i\le n$. In this case, the basic CT $R$-modules are precisely $T^\omega$, where $\omega\in\mathfrak{S}_n$.
\end{thm}

The following is an immediate application.

\begin{cor}
Grouping together terms in the decomposition of $f$, write $f=f_1^{a_1}\hdots f_t^{a_t}$ for some {\em distinct} prime elements $f_i$, and some $a_i\in\mathbb{N}$. 
Then $R$ has precisely $\frac{(a_1+\hdots+a_t)!}{a_1!\hdots a_t!}$ basic MM generators.
\end{cor}
\begin{proof}
All basic MM generators have the form $T^\omega$ for some $\omega\in\mathfrak{S}_{a_1+\hdots+a_t}$ by \ref{classification}.  
Accounting for the repetitions, there are precisely $\frac{(a_1+\hdots+a_t)!}{a_1!\hdots a_t!}$ of these.
\end{proof}

The strategy of proof of \ref{classification} is to use Kn\"orrer periodicity, the CY reduction prepared in \S\ref{reduction section}, together with the MM mutation from \S\ref{Mutation Section}.  
Let us start with the following observation.

\begin{lemma}\label{(u,f_I) is CM}
$(u,f_I)\in\CM R$ for any subset $I$ of $\{1,2,\hdots,n\}$.
\end{lemma}
\begin{proof}
This is clear since $(u,f_I)$ arises from the matrix factorization 
\[
R^2\xrightarrow{\left(\begin{smallmatrix}
-f_{I^c}&u\\ v& -f_I
\end{smallmatrix}\right)}
R^2 \xrightarrow{
\left(\begin{smallmatrix}
f_I&u\\ v& f_{I^c}
\end{smallmatrix}\right)}
R^2 \to
(u,f_I)\to 0.
\]
\end{proof}

We also need the following simple calculations.

\begin{lemma}\label{describe Hom-sets}
For any decomposition $f=abc$ with $a,b,c\in S$, we have isomorphisms\\
\t{(1)} $(u,b)\cong\Hom_R((u,a),(u,ab))$, $r\mapsto(\cdot r)$.\\
\t{(2)} $(u,ac)u^{-1}\cong\Hom_R((u,ab),(u,a))$, $r\mapsto(\cdot r)$, where $(u,ac)u^{-1}$ is a fractional ideal of $R$ generated by $1$ and $acu^{-1}$.
\end{lemma}

\begin{proof}
Let $I$ and $I'$ be non-zero ideals of $R$. Since $R$ is a domain, we have an isomorphism $\{q\in Q\mid Iq\subseteq I'\}\cong\Hom_R(I,I')$, $q\mapsto(\cdot q)$ for the quotient field $Q$ of $R$.\\
(1) Clearly we have $(u,b)\subseteq\{q\in Q\mid (u,a)q\subseteq (u,ab)\}$.
Conversely, any element $q$ belonging to the right hand side is certainly contained in $\{q\in Q\mid (u,a)q\subseteq (u,a)\}$ which equals $R$ since $R$ is normal. 
Thus it remains to show that if $r\in R$ satisfies $(u,a)r\subseteq (u,ab)$, then $r\in(u,b)$.

View such an $r\in R$ as an element of $S[[u,v]]$, then since $ar\in (u,ab)$ we can write $ar=up+abq$ in $R$ for some $p,q\in S[[u,v]]$.  
Since $f-uv$ is contained in the $S[[u,v]]$-ideal $(u,ab)$, we still have $ar=up+abq$ in $S[[u,v]]$ for some $p,q\in S[[u,v]]$.  
Then $a(r-bq)=up$, so since $a$ and $u$ have no common factors, we have $r-bq\in(u)$. Thus $r\in(u,b)$.\\
(2) By (\ref{replace u and v}), we have
\[\Hom_R((u,ab),(u,a))=\Hom_R((v,c)uc^{-1},(v,bc)u(bc)^{-1})=\Hom_R((v,c),(v,bc))b^{-1},\]
which equals $(v,b)b^{-1}=(u,ac)u^{-1}$ by (1).
\end{proof}

Immediately we have the following consequence.

\begin{prop}
$T^\c{F}$ is a modifying $R$-module for any flag $\c{F}$ in the set $\{ 1,2,\hdots,n \}$.
\end{prop}

Next we show that $T^\c{F}$ is an MM $R$-module if $\c{F}$ is maximal (\ref{key of classification 2}). For this, we need Kn\"orrer periodicity. Let
\[
R^{\flat}:=S/(f),
\]
then there is a triangle equivalence
\[
K:\uCM R^\flat\xrightarrow{\sim}\uCM R
\]
called \emph{Kn\"oerrer periodicity} (valid for any field $k$ by \cite[3.1]{Solberg}) which is defined as follows \cite{Y}:  For any $X\in \CM R^\flat$,
we have a free resolution
\[
R^\flat{}^{\oplus n}\xrightarrow{B}R^\flat{}^{\oplus n} \xrightarrow{A}R^\flat{}^{\oplus n}\to X \to0
\]
where $A$ and $B$ are $n\times n$ matrices over $S$
satisfying $AB=fI_n=BA$.
Then $K(X)$ is defined by the following exact sequence:
\[
R^{\oplus 2n}\xrightarrow{\left(\begin{smallmatrix}
B&-u\\ -v& A
\end{smallmatrix}\right)}
R^{\oplus 2n}\xrightarrow{\left(\begin{smallmatrix}
A&u\\ v& B
\end{smallmatrix}\right)}R^{\oplus 2n}\to K(X)\to 0.
\]
For any subset $I\subseteq\{ 1,2,\hdots, n\}$, and any flag $\c{F}=(I_1,\hdots,I_m)$ in the set $\{ 1,2,\hdots, n\}$, we define
\[
S_I:=S/(f_I)\cong R^\flat/(f_I)\ \mbox{ and }\ S^\c{F}:=R^\flat\oplus\left(\bigoplus_{j=1}^{m} S_{I_j}\right).
\]
Again we identify maximal flags with elements of the symmetric group, and so for $\omega\in\mathfrak{S}_n$ we set $S^\omega:=\bigoplus_{j=1}^{n}S_{\{\omega(1),\hdots,\omega(j)\}}$. The following is immediate:

\begin{lemma}\label{S and T}
\t{(1)} $K(S^\c{F})\cong T^\c{F}$ for any flag $\c{F}$ in the set $\{ 1,\hdots,n \}$.\\
\t{(2)} $K(S^\omega)\cong T^\omega$ for any $\omega\in\mathfrak{S}_n$.
\end{lemma}

The following is the key step.
\begin{prop}\label{key of classification}
Given a flag $\cF=(I_1,\hdots, I_m)$, we have a triangle equivalence
\[
(\uCM R)_{T^\cF}\simeq \bigoplus_{i=1}^{m+1}\uCM\left(\frac{k[[x,y,u,v]]}{(f_{I^i\backslash I^{i-1}}-uv)}\right),
\] 
where by convention  $I_0:=\emptyset$ and $I_{m+1}:=\{1,2,\hdots,n\}$.
\end{prop}
\begin{proof}
Since the equivalence $K:\uCM R^\flat\simeq\uCM R$ sends $T^\c{F}$ to $S^\c{F}$ by \ref{S and T}, we only have to calculate the CY reduction of $\uCM R^\flat$ with respect to $S^\c{F}$.
Applying \ref{3.7} repeatedly, it is triangle equivalent to $\bigoplus_{i=1}^{m+1}\uCM\left(S/(f_{I^i\backslash I^{i-1}})\right)$.
Applying Kn\"oerrer periodicity to each factor again, we have the result.
\end{proof}

It is known that under certain assumptions on the base field $k$, $[k]=0$ (and so $[R^\flat]=[\m]$) in the Grothendieck group $K_0(\mod R^\flat)$. However, since we do not have any assumptions on the field $k$, below we require the following technical observation.

\begin{lemma}\label{K(m) is zero}
Let $\m=(x,y)$ be the maximal ideal of $R^\flat$.  Then $[K(\m)]=2[R]$ in $K_0(\mod R)/\langle[X]\mid \dim_RX\le 1\rangle$.
\end{lemma}
\begin{proof}
Since $f\in\m$, by changing variables if necessary, we can assume that $f$ is not a multiple of $y$. Then we can find $g,h\in k[[x,y]]$ such that $f=xg+yh$ and  $g(x,0)\neq 0$.  Now a projective presentation of $\m$ is given by
\[
R^\flat{}^{\oplus 2}\xrightarrow{\left(\begin{smallmatrix}x&-h\\ y&g\end{smallmatrix}\right)}
R^\flat{}^{\oplus 2}\xrightarrow{\left(\begin{smallmatrix}g&h\\ -y&x\end{smallmatrix}\right)}
R^\flat{}^{\oplus 2}\xrightarrow{\left(\begin{smallmatrix}x\\ y\end{smallmatrix}\right)}\m\to 0
\]
and thus a projective presentation of $K(\m)$ is given by
\[
R^{\oplus 4}\xrightarrow{A=\left(\begin{smallmatrix}x&-h&-u&0\\ y&g&0&-u\\ -v&0&g&h\\ 0&-v&-y&x\end{smallmatrix}\right)}R^{\oplus 4}\xrightarrow{B=\left(\begin{smallmatrix}g&h&u&0\\ -y&x&0&u\\ v&0&x&-h\\ 0&v&y&g\end{smallmatrix}\right)}R^{\oplus 4}\to K(\m)\to 0.
\]
We claim that the sequence 
\begin{equation}\label{C and Km}
0\to K(\m)\to R^{\oplus 4}\xrightarrow{d_3=B}
R^{\oplus 4}\xrightarrow{d_2=\left(\begin{smallmatrix}x&-h&-u\\ y&g&0\\ -v&0&g\\ 0&-v&-y\end{smallmatrix}\right)}
R^{\oplus 3}\xrightarrow{d_1=\left(\begin{smallmatrix}g\\ -y\\ v\end{smallmatrix}\right)}R\to C\to0
\end{equation}
is exact for $C:=R/(g,y,v)$.  We first show that $\ker d_2=\Im d_3$. For the $i$-th column $A_i$ of $A$, we have a linear relation $A_4=-(h/v)A_1-(x/v)A_2$ over the quotient field $Q$ of $R$. Thus we have $\ker(Q^{\oplus 4}\xrightarrow{Q\otimes_Rd_2}Q^{\oplus 3})=\ker(Q^{\oplus 4}\xrightarrow{Q\otimes_RA}Q^{\oplus 4})$ and hence
\[
\ker d_2=R^{\oplus 4}\cap \ker(Q^{\oplus 4}\xrightarrow{Q\otimes_Rd_2}Q^{\oplus 3})=R^{\oplus 4}\cap \ker(Q^{\oplus 4}\xrightarrow{Q\otimes_RA}Q^{\oplus 4})=\ker A=\Im d_3.
\]	
We next show that $\ker d_1=\Im d_2$.  The regular sequence  $(g,y,v)$ on $S[[u,v]]$ gives an exact sequence 
\[
S[[u,v]]^{\oplus 3}\xrightarrow{e_2=\left(\begin{smallmatrix}y&g&0\\ -v&0&g\\ 0&-v&-y\end{smallmatrix}\right)}
S[[u,v]]^{\oplus 3}\xrightarrow{e_1=\left(\begin{smallmatrix}g\\ -y\\ v\end{smallmatrix}\right)}
S[[u,v]].
\]
Assume that the image $(a,b,c)\in R^{\oplus 3}$ of $(a,b,c)\in S[[u,v]]^{\oplus 3}$ satisfies $d_1(a,b,c)=0$. 
Then $e_1(a,b,c)=(f-uv)s$ holds for some $s\in S[[u,v]]$, which implies $e_1(a-sx,b+sh,c+su)=0$.
Thus there exists $(a',b',c')\in S[[u,v]]^{\oplus 3}$ such that $(a-sx,b+sh,c+su)=e_2(a',b',c')$.
This implies $(a,b,c)=d_2(s,a',b',c')$, which shows the assertion.

Hence \eqref{C and Km} is exact, thus $[K(\m)]=2[R]+[C]$ holds.
Since $C=k[[x,u]]/(g(x,0))$ and $g(x,0)\neq 0$, we have $\dim_RC\le 1$ and so the result follows.
\end{proof}

Now to show that $T^\c{F}$ is an MM $R$-module for any maximal flag $\c{F}$, we need to first understand the case when $f$ is irreducible.  The following extends \cite[4.3]{DH} by removing field restrictions.

\begin{prop}\label{irreducible case}
\t{(1)} ${\rm Cl}(R)$ is generated by $[(u,f_1)],\ldots,[(u,f_n)]$.\\
\t{(2)} Assume that $f$ is irreducible (i.e. $n=1$). Then
\begin{enumerate}
\item[(a)] $R$ is factorial.
\item[(b)] Any modifying $R$-module is free.
\end{enumerate}
\end{prop}
\begin{proof}
(1)  By d\'evissage, $R^\flat$, $\m$ and the factors  $S/(f_1),\hdots,S/(f_t)$ of the minimal primes generate $ K_0(\mod R^\flat)$.   There are isomorphisms
\[
K_0(\mod R^\flat)/\langle[R^\flat]\rangle\simeq K_0(\underline{\CM}R^\flat)
\stackrel{K}{\simeq} 
K_0(\underline{\CM}R)\simeq K_0(\mod R)/\langle[R]\rangle,
\]
thus $K_0(\mod R)$ is generated by $[R]$, $[K(\m)]$ and $[K(R^\flat/(f_i))]=[(u,f_i)]$ for $1\le i\le n$. Now by \cite[VII.4.7]{bourbaki}, there is an isomorphism
\[
\mathbb{Z}\oplus{\rm Cl}(R)=K_0(\mod R)/\langle[X]\mid \dim_RX\le 1\rangle,
\]
which implies
\[
{\rm Cl}(R)=K_0(\mod R)/\langle[R],\ [X]\mid \dim_RX\le 1\rangle.
\]
Thus by \ref{K(m) is zero}, it follows that ${\rm Cl}(R)$ is generated by $[(u,f_i)]$ for $1\le i\le n$.\\
(2) $R$ is factorial by (1), since $(u,f)\cong R$. The assertion (b) follows from factoriality of $R$ and Dao's result \cite[3.1(1)]{Dao} (see also \cite[2.10]{IW5}).
\end{proof}

\begin{remark}
It is possible to give another proof of \ref{irreducible case}(2)(b) by appealing instead to  a result by Huneke--Wiegand \cite[3.7]{HW}.  If $f$ is irreducible then $R^{\flat}$ is a one-dimensional domain, so necessarily it is an isolated singularity and thus $\uCM R^{\flat}$ is Hom-finite.   By Kn\"orrer periodicity, $\uCM R$ is Hom-finite, so  $R$ is an isolated singularity.  Thus by \cite[2.9(2)$\Leftrightarrow$(3)]{IW5}, to prove \ref{irreducible case}(2)(b) it is enough to show that any $M\in\CM R^{\flat}$ satisfying $\Ext^1_{R^{\flat}}(M,M)=0$ is free. Now we have
\[
\Ext^1_{R^\flat}(\Tr M,M^*)=\u{\Hom}_{R^\flat}(M^*[2], M^*[1])=\u{\Hom}_{R^\flat}( M[1], M[2])=\Ext^1_{R^\flat}(M,M)=0,
\]
so the exact sequence \cite{AB}
\[
0\to\Ext^1_{R^\flat}(\Tr M,M^*)\to M\otimes_{R^\flat} M^*\to \Hom_{R^\flat}(M^*,M^*)\to\Ext^2_{R^\flat}(\Tr M,M^*)\to 0
\]
shows that $M\otimes_{R^\flat} M^*\in\CM R^\flat$.  Since $R^\flat$ is a domain, $M\in\CM R^{\flat}$ has constant rank, hence since $R^{\flat}$ is a hypersurface, \cite[3.7]{HW} implies that $M$ is free.
\end{remark}

\begin{cor}\label{key of classification 2}
\t{(1)} $T^\omega$ is an MM generator of $R$ for all $\omega\in S_n$.\\
\t{(2)} $T^\omega$ is a CT $R$-module if and only if $f_i\notin(x,y)^2$ for all $i$.
\end{cor}
\begin{proof}
(1) By \ref{key of classification} and \ref{irreducible case}(2)(b), every modifying object in $(\uCM R)_{T^\omega}$ is zero.
Thus the assertion follows from \ref{CT and CY reduction}.\\
(2) By \ref{CT and CY reduction}, $T^\omega$ is a CT $R$-module if and only if $(\uCM R)_{T^\omega}=0$. 
By \ref{key of classification}, this is equivalent to that $f_i\notin(x,y)^2$ for all $i$.
\end{proof}

We now calculate the mutations $\mu_i(T^\omega)$ introduced in \S\ref{Mutation Section}.

\begin{lemma}\label{exactness}
For any decomposition $f=abcd$ with $a,b,c,d\in S$, we have an exact sequence
\[
0\to (u,ab)\xrightarrow{(1\ c)}
(u,b)\oplus(u,abc)\xrightarrow{{-c\choose 1}}
(u,bc)\to  0.
\]
\end{lemma}
\begin{proof}
Clearly the sequence is a complex, and the right map is surjective and the left map is injective.
We only have to show that the kernel of the right map is contained in the image of the left map, or equivalently $(uc,bc)\cap(u,abc)\subset(uc,abc)$.
Any element in the left hand side can be written as $ucp+bcq=ur+abcs$ for some $p,q,r,s$.
It is enough to show $q\in(u,a)$.
Since $R=S[[u,v]]/(f-uv)$, we have an equality $ucp+bcq=ur+abcs+(abcd-uv)t$ in $S[[u,v]]$ for some $p,q,r,s,t\in S[[u,v]]$.
Thus we have $bc(q-as-adt)=u(r-vt-cp)$. Since $S[[u,v]]$ is factorial and $bc$ and $u$ have no common factors, we have $q-as-adt\in (u)$.
Thus $q\in(u,a)$ as we required.
\end{proof}

\begin{lemma}\label{approximation}
With the assumptions in \ref{exactness}, assume that $g\in S$ is either a factor of $b$ or has $abc$ as a factor. Then\\
\t{(1)} The map $\Hom_R((u,b)\oplus(u,abc),(u,g))\xrightarrow{(1\ c)\cdot}\Hom_R((u,ab),(u,g))$ is surjective.\\
\t{(2)} The map $\Hom_R((u,g),(u,b)\oplus(u,abc))\xrightarrow{\cdot{c\choose 1}}\Hom_R((u,g),(u,bc))$ is surjective.
\end{lemma}

\begin{proof}
(1) Assume that $g$ has $abc$ as a factor.
Using the isomorphisms in \ref{describe Hom-sets}, the map is given by ${1\choose c}\colon(u,b^{-1}g)\oplus(u,(abc)^{-1}g)\to(u,(ab)^{-1}g)$, which is clearly surjective.
All other cases can be checked similarly.
\end{proof}

Recall that $T^\omega$ is given by $R\oplus(u,f_{\omega(1)})\oplus(u,f_{\omega(1)}f_{\omega(2)})\oplus\hdots\oplus(u,f_{\omega(1)}\hdots f_{\omega(n-1)})$.
Let $s_i:=(i\ i+1)$ be a permutation in $\mathfrak{S}_n$.

\begin{lemma}\label{mutation and si}
The MM mutation $\mu_i(T^\omega)$ of $T^\omega$ with respect to the summand $(u,f_{\omega(1)}\hdots f_{\omega(i)})$ is $T^{\omega s_i}$.
\end{lemma}
\begin{proof}
We only have to consider the case $\omega={\rm id}$. By \ref{exactness}, the sequence
\[
0\to (u,f_1\hdots f_i)\to (u,f_1\hdots f_{i-1})\oplus(u,f_1\hdots f_{i+1}) \to (u,f_1\hdots f_{i-1}\widehat{f_i}f_{i+1})\to 0
\]
is exact. By \ref{approximation}, the left map is a minimal left $(\add T^\omega/(u,f_1\hdots f_i))$-approximation. Thus we have
\[
\mu_i(T^\omega)=\left(\frac{T^\omega}{(u,f_1\hdots f_i)}\right)\oplus(u,f_1\hdots f_{i-1}\widehat{f_i} f_{i+1})=T^{\omega s_i}
\]
as required.
\end{proof}

We consider mutations at non-indecomposable summands later in \S\ref{mutation}. Now we are ready to prove \ref{classification}.
\begin{proof}
(2) It is shown in \ref{key of classification 2}(1) that $T^\omega$ for any $\omega\in\mathfrak{S}_n$ is an MM generator of $R$.
We need to show that there are no more.
By \ref{mutation and si}, the MM generators $T^\omega$ for all $\omega\in\mathfrak{S}_n$ forms a finite connected component in $\MMG R$.
By \ref{Auslander type}, they give all MM generators of $R$.\\
(1) Since $\dim R=3$, every modifying generator is a summand of an MM generator \cite[4.18]{IW4}. Hence (1) is immediate from (2).\\
(3) The first assertion follows from \ref{key of classification 2}(2).
Since $\dim R=3$, if there exists a CT $R$-module, then CT $R$-modules are precisely the MM generators of $R$ \cite[5.11(2)]{IW4}.
Thus the second assertion follows from (2).
\end{proof}

Recall that the \emph{length} $\ell(w)$ of an element $w\in\mathfrak{S}_n$ is
the minimal number $k$ for each expression $w=s_{i_1}\ldots s_{i_k}$.
The \emph{weak order} on $\mathfrak{S}_n$ is defined as follows:
$w\le w'$ if and only if $\ell(w')=\ell(w)+\ell(w^{-1}w')$.

\begin{thm}
With the setup as above, assume that $R$ is an isolated singularity,
or equivalently $(f_i)\neq (f_j)$ as ideals of $S$ for any $i\neq j$.\\
\t{(1)} The map $\mathfrak{S}_n\to\MMG R$, $\omega\mapsto T^\omega$
is bijective.\\
\t{(2)} The exchange graph $\Gr(\MMG R)$ is isomorphic to the
Hasse graph of the partially ordered set $\mathfrak{S}_n$ with respect
to the weak order.
\end{thm}

\begin{proof}
(1) By assumption, two subsets $I$ and $I'$ of $\{1,2,\hdots,n\}$
satisfies $T_I\cong T_{I'}$ if and only if $I=I'$.
Thus two elements $\omega$ and $\omega'$ in $\mathfrak{S}_n$ satisfies
$T^\omega\cong T^{\omega'}$ if and only if $\omega=\omega'$.\\
(2) By \ref{mutation and si}, the edges in $\Gr(\MMG R)$ connect
$\omega$ and $\omega s_i$ for any $\omega\in \mathfrak{S}_n$ and $1\le i\le n-1$.
This is nothing but the Hasse graph of $\mathfrak{S}_n$.
\end{proof}

\begin{example}
We give examples for $n=4$.\\
\t{(1)} If $(f_1)=(f_2)=(f_3)=(f_4)$, then $\Gr(\MMG R)$ is the following:
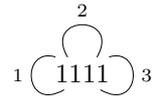
\[\xymatrix{
1111\ar@{-}@(ul,ur)^2\ar@{-}@(ur,dr)^3\ar@{-}@(dl,ul)^1
}\\[2mm]
\]
\t{(2)} If $(f_1)=(f_2)=(f_3)\neq(f_4)$, then $\Gr(\MMG R)$ is the following:
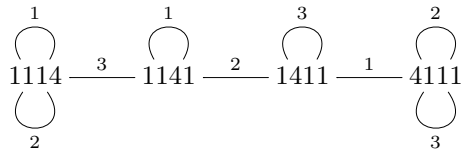
\[\xymatrix{
1114\ar@{-}@(ul,ur)^1\ar@{-}@(dr,dl)^2\ar@{-}[r]^3&
1141\ar@{-}@(ul,ur)^1\ar@{-}[r]^2&
1411\ar@{-}@(ul,ur)^3\ar@{-}[r]^1&
4111\ar@{-}@(ul,ur)^2\ar@{-}@(dr,dl)^3
}\]\\
\t{(3)} If $(f_1)=(f_2)\neq(f_3)=(f_4)$, then $\Gr(\MMG R)$ is the following:
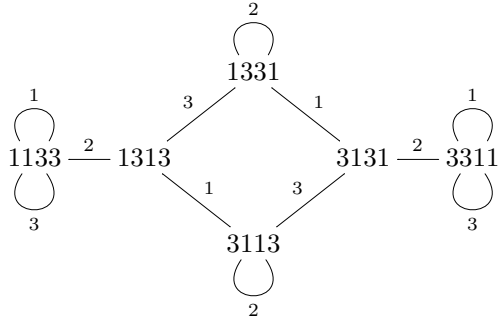
\[\xymatrix@C=15pt@R=20pt{
&&1331\ar@{-}@(ul,ur)^2\ar@{-}[dr]^1\\
1133\ar@{-}@(ul,ur)^1\ar@{-}@(dr,dl)^3\ar@{-}[r]^2&
1313\ar@{-}[ur]^3\ar@{-}[dr]^1&&3131\ar@{-}[r]^2&
3311\ar@{-}@(ul,ur)^1\ar@{-}@(dr,dl)^3\\
&&3113\ar@{-}@(dr,dl)^2\ar@{-}[ur]^3
}\]\\
\t{(4)} If $(f_1)=(f_2)$, $(f_3)$ and $(f_4)$ are different, then $\Gr(\MMG R)$ is the following:
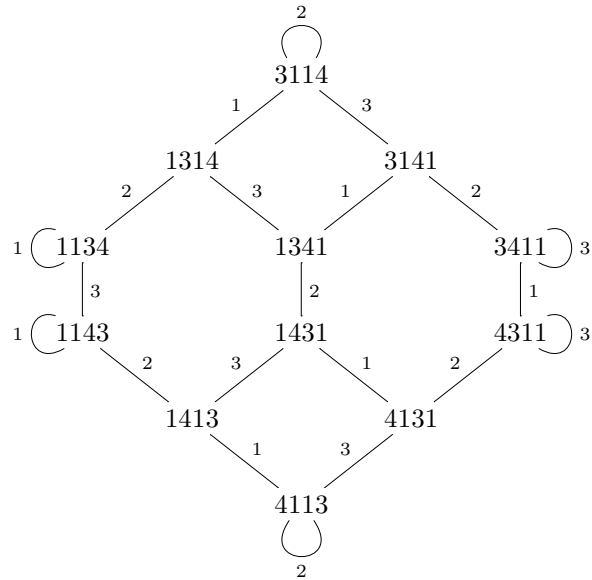
\[\xymatrix@C=15pt@R=20pt{
&&3114\ar@{-}@(ul,ur)^2\ar@{-}[rd]^3\\
&1314\ar@{-}[ur]^1\ar@{-}[rd]^3&&3141\ar@{-}[rd]^2\\
1134\ar@{-}@(dl,ul)^1\ar@{-}[ur]^2\ar@{-}[d]^3&&
1341\ar@{-}[d]^2\ar@{-}[ur]^1&&3411\ar@{-}@(ur,dr)^3\ar@{-}[d]^1\\
1143\ar@{-}@(dl,ul)^1\ar@{-}[dr]^2&&1431\ar@{-}[dr]^1&&4311\ar@{-}@(ur,dr)^3\\
&1413\ar@{-}[ur]^3\ar@{-}[dr]^1&&4131\ar@{-}[ur]^2\\
&&4113\ar@{-}[ur]^3\ar@{-}@(dr,dl)^2
}\]
\t{(5)} If $(f_i)\neq(f_j)$ for all $i\neq j$, then $\Gr(\MMG R)$ is the following:
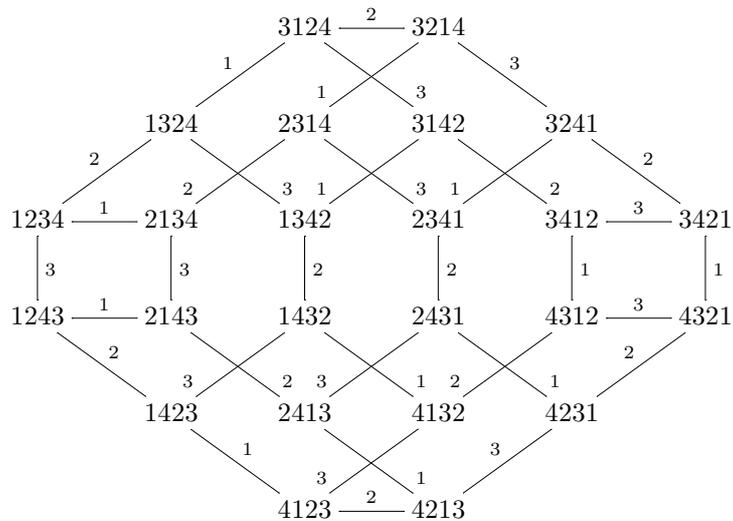
\[\xymatrix{
&&3124\ar@{-}[r]^2\ar@{-}[rd]^(.8)3&3214\ar@{-}[rd]^3\\
&1324\ar@{-}[ur]^1\ar@{-}[rd]^(.8)3&2314\ar@{-}[ur]^(.2)1\ar@{-}[rd]^(.8)3&3142\ar@{-}[rd]^(.8)2&3241\ar@{-}[rd]^2\\
1234\ar@{-}[r]^1\ar@{-}[ur]^2\ar@{-}[d]^3&2134\ar@{-}[ur]^(.2)2\ar@{-}[d]^3&
1342\ar@{-}[d]^2\ar@{-}[ur]^(.2)1&2341\ar@{-}[ur]^(.2)1\ar@{-}[d]^2&3412\ar@{-}[r]^3\ar@{-}[d]^1&3421\ar@{-}[d]^1\\
1243\ar@{-}[r]^1\ar@{-}[dr]^2&2143\ar@{-}[dr]^(.8)2&1432\ar@{-}[dr]^(.8)1&2431\ar@{-}[dr]^(.8)1&4312\ar@{-}[r]^3&4321\\
&1423\ar@{-}[ur]^(.2)3\ar@{-}[dr]^1&2413\ar@{-}[ur]^(.2)3\ar@{-}[dr]^(.8)1&4132\ar@{-}[ur]^(.2)2&4231\ar@{-}[ur]^2\\
&&4123\ar@{-}[ur]^(.2)3\ar@{-}[r]^2&4213\ar@{-}[ur]^3
}\]
\end{example}
There is a geometric interpretation of all these examples in terms  of curves; for example the following will be given a geometric interpretation in \ref{EG geom}.
\begin{example}\label{EG alg}
Let $n=5$. If $(f_1)=(f_2)\neq(f_3)=(f_4)=(f_5)$, then $\Gr(\MMG R)$ is the following:
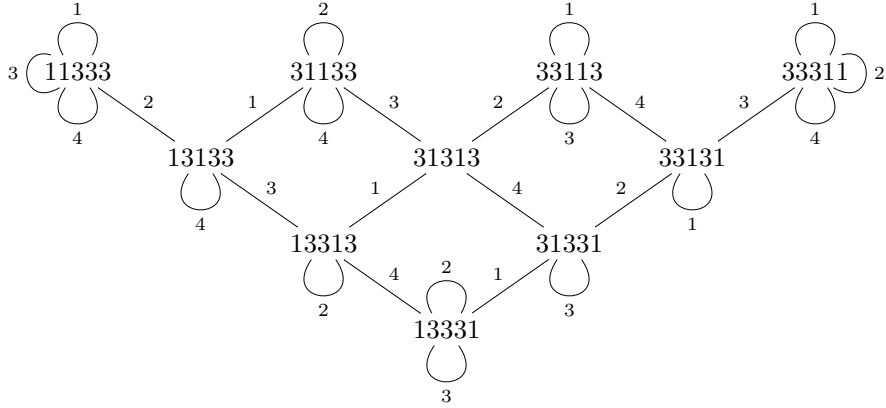
\[\xymatrix@C=15pt@R=20pt{
11333\ar@{-}@(dr,dl)^4\ar@{-}@(ld,lu)^3\ar@{-}@(ul,ur)^1\ar@{-}[dr]^2&&
31133\ar@{-}@(dr,dl)^4\ar@{-}@(ul,ur)^2\ar@{-}[dr]^3&&
33113\ar@{-}@(dr,dl)^3\ar@{-}@(ul,ur)^1\ar@{-}[dr]^4&&
33311\ar@{-}@(dr,dl)^4\ar@{-}@(ru,rd)^2\ar@{-}@(ul,ur)^1\\
&13133\ar@{-}@(dr,dl)^4\ar@{-}[ur]^1\ar@{-}[dr]^3&&
31313\ar@{-}[ur]^2\ar@{-}[dr]^4&&33131\ar@{-}[ur]^3\ar@{-}@(dr,dl)^1\\
&&13313\ar@{-}@(dr,dl)^2\ar@{-}[ur]^1\ar@{-}[dr]^4
&&31331\ar@{-}@(dr,dl)^3\ar@{-}[ur]^2\\
&&&13331\ar@{-}[ur]^1\ar@{-}@(ul,ur)^2\ar@{-}@(dr,dl)^3
}\]
\end{example}

\subsection{The class group of $R$}\label{class section}
In this section we give a complete description of the class group $\Cl(R)$, which is independent of the base field $k$.  This requires the following lemma.  We denote by ${\rm Ch}(R)$ the Chow group of $R$ (see e.g. \cite{R98}).

\begin{lemma}\label{prelim for class}
\t{(1)} There is an injective homomorphism ${\rm Cl}(R)\to{\rm Ch}(R)$ sending 
\[
I\mapsto\sum_{{\rm ht}\p=1}{\rm length}_{A_{\p}}((R/I)_{\p})[\p].
\]
\t{(2)} Let $g,h\in \k[[x,y]]$, and suppose that $gh$ divides $f$.  Then
\begin{enumerate}
\item[(a)] $[(u,g)]+[(u,h)]=[(u,gh)]$ in $K_0(\mod R)$.
\item[(b)] $((u,g)\otimes_R(u,h))^{**}\cong (u,gh)$.
\end{enumerate}
\t{(3)} If $f_1^{c_1}\hdots f_t^{c_t}$ with each $c_i\ge0$ is contained in the principal ideal $(u)$ of $R$, then $c_i\ge a_i$ for all $i$.
\end{lemma}
\begin{proof}
(1) See e.g.\ \cite[\S1.2]{R98}.\\
(2a) By \ref{(u,f_I) is CM}, $(u,gh)$ is a CM $R$-module of rank one.  It follows that $(u,gh)\in\Cl(R)$ and hence by (1) we only have to show that
\[
{\rm length}_{A_{\p}}(({R}/{(u,g)})_{\p})+{\rm length}_{A_{\p}}(({R}/{(u,h)})_{\p})={\rm length}_{A_{\p}}(({R}/{(u,gh)})_{\p})
\]
for all height one primes $\p$.  We know that ${R}/{(u,g)}=\k[[x,y,v]]/(g)$, ${R}/{(u,h)}=\k[[x,y,v]]/(h)$ and ${R}/{(u,gh)}=\k[[x,y,v]]/(gh)$. Since we have an exact sequence
\[
0\to \k[[x,y,v]]/(g)\xrightarrow{h}\k[[x,y,v]]/(gh)\xrightarrow{}\k[[x,y,v]]/(h)\to0,
\]
the desired equality holds by localizing at $\p$.\\
(2b) Follows immediately from (2a), since all terms in (2a) are CM of rank one, so can be viewed as elements of the class group.\\ 
(3) The element $f_1^{c_1}\hdots f_t^{c_t}$ is zero in $R/(u,v)=k[[x,y]]/(f)$. Thus $f_1^{c_1}\hdots f_t^{c_t}$ must be a multiple of $f$.
\end{proof}

This allows us to describe the class group $\Cl(R)$.
\begin{thm}\label{class group description}
As above, $R:=k[[x,y,u,v]]/(f-uv)$ and we write the decomposition into irreducibles $f=f_1^{a_1}\hdots f_t^{a_t}$ with $a_i\ge1$ and $(f_i)\neq(f_j)$ for all $i\neq j$. Then
\[
\Cl(R)\cong \frac{\mathbb{Z}^t}{\langle(a_1,\ldots,a_t)\rangle},
\]
where $[(u,f_i)]$ corresponds to the element $(0,\ldots,0,1,0,\ldots,0)$.
\end{thm}

\begin{proof}
Let $M_i:=(u,f_i)$, then it follows from \ref{irreducible case}(1) that $[M_i]$ for $1\le i\le n$ generate the class group. Moreover, by \ref{prelim for class}(2a), the relation $\sum_{i=1}^ta_i[M_i]=0$ is satisfied.

We need to show that $\sum_{i=1}^tb_i[M_i]=0$ implies that
$(b_1,\ldots,b_t)$ is an integer multiple of $(a_1,\ldots,a_t)$.
To prove this, we can  without loss of generality assume that $b_i\ge0$ for all $i$, by if necessary adding a multiple of $(a_1,\ldots,a_t)$ to $(b_1,\ldots,b_t)$.  Throughout, we let $I$ be the ideal of $R$ given by
\[
I:=(u,f_1)^{b_1}\hdots(u,f_t)^{b_t}.
\]
(a) We claim that, if $b_i=0$ for some $i$,  then
$(b_1,\ldots,b_t)=(0,\ldots,0)$.  Assume that $b_i\geq 0$ for all $i$, and $b_i=0$ for some $i$.  Then $I^*$ is a free $R$-module since $[I^*]=-\sum_{i=1}^tb_i[M_i]=0$.
We identify $I^*$ with the fractional ideal
\[
I^*=\{x\in Q\mid xI\subset R\}
\]
in the quotient field $Q$ of $R$. Then $I^*$ is generated by an element
in $Q$. Since $1\in I^*$, there exists $r\in R$ such that $I^*=R(1/r)$.
Then we have $I\subset Rr$, which implies
\[
(u,f_1)^{b_1}\hdots(u,f_t)^{b_t}\subset(r,f-uv)
\]
as ideals of $k[[x,y,u,v]]$. In particular, $u^b$ for $b:=b_1+\hdots+b_t$
and $F:=f_1^{b_1}\hdots f_t^{b_t}$ are contained in $(r,f-uv)$.
Factoring by $u-f_i$ and $v-f/f_i$, we have that $f_i^b$ and $F$ are
contained in the principal ideal of $k[[x,y]]$ generated by
$s:=r|_{u=f_i,\ v=f/f_i}$. Since $f_i^b$ and $F$ do not have a
common factor by the assumption $b_i=0$, we have that $s$ is a unit in $k[[x,y]]$ and so it must have a constant term. 
Hence $r$ must also have a constant term, thus $r$ is a unit in $k[[x,y,u,v]]$, and so $I^*=R$.

Now let $g:=\prod_{i=1}^tf_i^{\max\{a_i-b_i,0\}}$. Then $(g/u)I\subset R$ since all generators of $I$ except $F$ are multiples of $u$, and moreover $gF$ is a multiple of $f$, which equals $uv$. Hence $g/u\in I^*=R$ holds. Thus $g$ is contained in the ideal $(u)$ of $R$, and by \ref{prelim for class}(3) we have $(b_1,\ldots,b_t)=(0,\ldots,0)$. Thus the claim (a) holds.\\
(b) We now prove the theorem.  Without loss of generality, we can assume that $b_1/a_1\ge b_i/a_i$ for all $i$. Then
\[
(c_1,\ldots,c_t):=b_1(a_1,\ldots,a_t)-a_1(b_1,\ldots,b_t)
\]
satisfies $\sum_{i=1}^tc_i[M_i]=0$, $c_1=0$ and $c_i\ge0$ for all $i$.
Applying (a) to $(c_1,\ldots,c_t)$, we have $c_i=0$ for all $i$, and so $(b_1,\ldots,b_t)=\frac{b_1}{a_1}(a_1,\hdots,a_t)$, which implies $\frac{b_1}{a_1}=\frac{b_i}{a_i}$ for all $i$.  Now if $\frac{b_1}{a_1}$  is an integer we are done, otherwise by subtracting an integer multiple of $(a_1,\hdots,a_t)$ from
$(b_1,\ldots,b_t)$, we can assume $0\le b_1<a_1$, from which necessarily $0\le b_i<a_i$ holds for all $i$.  But then using \ref{prelim for class}(2b) repeatedly we see that
\[
I^{**}\cong (u,f_1^{b_1}\hdots f_t^{b_t}).
\]
This is non-free except the case $(b_1,\ldots,b_t)=(0,\ldots,0)$, by \ref{prelim for class}(3). 
\end{proof}

\subsection{MM Modules}\label{subsection on MM}

We keep the notations from previous sections, but we now set $k=\mathbb{C}$.  The purpose of this section is to prove the following theorem, which extends \ref{classification} to cover non-generators.
\begin{thm}\label{MMnotgen}
Let $R=\mathbb{C}[[x,y,u,v]]/(f(x,y)-uv)$ as above, then the basic MM $R$-modules are precisely $(I\otimes_R T^\omega)^{**}$ for some $\omega\in\mathfrak{S}_n$ and some $I\in\Cl(R)$.
\end{thm}

Together with the description of the class group from \ref{class group description}, this gives a full description of all MM $R$-modules.  Before we prove \ref{MMnotgen}, we give the following surprising corollary.
\begin{cor}\label{finite number cor} Let $R=\mathbb{C}[[x,y,u,v]]/(f(x,y)-uv)$ as above.  Then\\
\t{(1)} There are only finitely many algebras (up to Morita equivalence) in the derived equivalence class containing the MMAs of $R$.\\
\t{(2)} There are only finitely many algebras (up to Morita equivalence) in the derived equivalence class containing the $\mathds{Q}$-factorial terminalizations of $\Spec R$.
\end{cor}
\begin{proof}
(1) Since $R$ is a normal three-dimensional domain, MMAs are closed under derived equivalences \cite[4.8]{IW4}.
Thus we only have to show that there are only finitely many MMAs up to Morita equivalence.
By \ref{MMnotgen} every MMA of $R$ is Morita equivalent to $\End_R((I\otimes_R T^\omega)^{**})$ for some $\omega\in\mathfrak{S}_n$ and some $I\in\Cl(R)$.  
Since $\End_R((I\otimes_RT^\omega)^{**})\cong \End_R(T^\omega)$, there are at most $n!$ possible algebras. Thus the result follows.\\
(2) By \cite[1.9]{IW5} every $\mathds{Q}$-factorial terminalization is derived equivalent to an MMA. Thus the assertion follows by (1).
\end{proof}

The strategy to prove \ref{MMnotgen} is to use the following easy fact.

\begin{lemma}\label{if M has rank 1 summand}
Let $M\in\refl R$ be a modifying $R$-module. If $M$ has a direct summand $I$ whose rank is $1$, then $M\cong (I\otimes_RN)^{**}$ for some modifying generator $N$ of $R$.
\end{lemma}

\begin{proof}
Let $N:=\Hom_R(I,M)$. Then clearly $N$ is a generator of $R$ and we have $M\cong (I\otimes_RN)^{**}$. Since $\End_R(N)\cong\End_R(M)$, we have that $N$ is a modifying generator of $R$.
\end{proof}

On the other hand, we require the following lemma.

\begin{lemma}\label{generic}
Suppose that $\k=\mathbb{C}$.  Then for any modifying $R$-module $M$, we can choose
a hyperplane section $t$ satisfying the following two conditions.\\
\t{(1)} $R/(t)$ is an $A_m$ singularity, where $m$ is
the degree of the lowest term of $f$ minus one.\\
\t{(2)} $t$ acts on $E:=\Ext^1_{R}(M,M)$ as a non-zerodivisor.
\end{lemma}

\begin{proof}
Since $R$ is a $cA_m$ singularity (see e.g\ \cite[6.1(e)]{BIKR}), a generic hyperplane section $t$ satisfies the condition (1). If $t$ acts on $E$ as a zero divisor, then $t$ is contained in an associated prime ideal of $E$. 
But $\fl_RE=0$ by \ref{modifyingThesame}, so any associated prime  ideal of $E$ is necessarily non-maximal.
Since $E$ has only finitely many associated prime ideals, we can find
a hyperplane $t$ which is not contained in any associated prime
ideal of $E$, and furthermore satisfies (1).
\end{proof}

This gives the following result, which generalizes \cite[A1]{VdB1d} and \cite[4.3]{DH}.

\begin{prop}\label{modifying is rank 1}
Assume $k=\mathbb{C}$. Then any indecomposable modifying $R$-module has rank 1.
\end{prop}

\begin{proof}
Suppose that $M\in\refl R$ is indecomposable and satisfies $\End_{R}(M)\in\CM R$.
By \ref{modifyingThesame} we have $\fl_{R}\Ext^{1}_{R}(M,M)=0$.
By \ref{generic} we can pick $t\in R$ such that $R_1:=R/(t)\cong k[[x,y,z]]/(x^2+y^2+z^m)$, and $t$ acts on $\Ext^1_{R}(M,M)$ as a non-zerodivisor.  Denote $\Lambda:=\End_{R}(M)$, then applying $\Hom_{R}(M,-)$ to the exact sequence
\[
0\to M\stackrel{t}{\to} M\to M/tM\to 0
\]
yields 
\[
0\to \Lambda\xrightarrow{t}\Lambda\stackrel{}{\to} \Hom_{R}(M,M/tM)\to \Ext^1_{R}(M,M)\xrightarrow{t}\Ext^1_{R}(M,M)
\]
Since $t$ acts on $\Ext^1_{R}(M,M)$ as a non-zerodivisor, we have
\[
\Lambda/t\Lambda\cong\Hom_{R}(M,M/tM)\cong\End_{R}(M/tM)=\End_{R_1}(M/tM).
\]
In particular $\End_{R_1}(M/tM)\in\CM R_1$, so by \cite[4.1]{AG} we have
\[
\Lambda/t\Lambda\cong \End_{R_1}(M/tM)\cong \End_{R_1}((M/tM)^{**})
\]
with $(M/tM)^{**}\in\CM R_1$. 
Since $M$ is indecomposable, we have that $\Lambda=\End_{R}(M)$ is a local ring.
Thus $\Lambda/t\Lambda\cong\End_{R_1}((M/tM)^{**})$ is also a local ring, and so $(M/tM)^{**}$ is indecomposable.
Since $R_1$ is a simple surface singularity of type $A_{m-1}$, it is well-known (e.g.\ \cite{Y}) that all indecomposable CM $R_1$-modules have rank one. 
Thus we have $1=\rk_{R_1}((M/tM)^{**})=\rk_{R_1}(M/tM)=\rk_R(M)$.
\end{proof}

Now we are ready to prove \ref{MMnotgen}.
\begin{proof}
\ref{MMnotgen} now follows immediately from \ref{classification}, \ref{if M has rank 1 summand} and \ref{modifying is rank 1}.
\end{proof}

\subsection{General MM Mutation}\label{mutation}
We once again work over a general field $k$.  In order to extend \ref{mutation and si} and describe mutation for non-maximal flags, it is combinatorially useful, given some (possibly non-maximal) flag $\c{F}=(I_1,\hdots,I_m)$ in the set $\{ 1,2,\hdots, n\}$, to assign to $\c{F}$ the following picture consisting of $m$ curves:
\[
\begin{array}{ccc}
\begin{array}{c}
\c{P}(\c{F}):=
\end{array} &
\begin{array}{c}
\begin{tikzpicture}[xscale=0.6,yscale=0.6]
\draw[black] (-0.1,-0.04,0) to [bend left=25] (2.1,-0.04,0);
\draw[black] (1.9,-0.04,0) to [bend left=25] (4.1,-0.04,0);
\node at (5.5,0,0) {$\hdots$};
\draw[black] (6.9,-0.04,0) to [bend left=25] (9.1,-0.04,0);
\node at (1,0.6,0) {$\scriptstyle C_{1}$};
\node at (3,0.6,0) {$\scriptstyle C_{2}$};
\node at (8,0.6,0) {$\scriptstyle C_{m}$};
\filldraw [red] (0,0,0) circle (1pt);
\filldraw [red] (2,0,0) circle (1pt);
\filldraw [red] (4,0,0) circle (1pt);
\filldraw [red] (7,0,0) circle (1pt);
\filldraw [red] (9,0,0) circle (1pt);
\node at (0,-0.4,0) {$\scriptstyle g_1$};
\node at (2,-0.4,0) {$\scriptstyle g_2$};
\node at (4,-0.4,0) {$\scriptstyle g_3$};
\node at (7,-0.4,0) {$\scriptstyle g_{m}$};
\node at (9,-0.4,0) {$\scriptstyle g_{m+1}$};
\end{tikzpicture} 
\end{array}
\end{array}
\]
where $g_j:=f_{I_j\backslash I_{j-1}}$ for all $1\le j\le m+1$, where by convention  $I_0:=\emptyset$ and $I_{m+1}:=\{1,2,\hdots,n\}$.

If we denote $M_0:=R$ and $M_j:=(u,f_{I_j})=(u,\prod_{i=1}^{j}g_i)$, and set $T^\c{F}:=\bigoplus_{j=0}^{m}M_j$, then $T^\c{F}$ is the modifying generator of $R$ corresponding to $\c{F}$.  
The correspondence between non-free summands of $T^\c{F}$ and curves of $\c{P}(\c{F})$ is as follows:
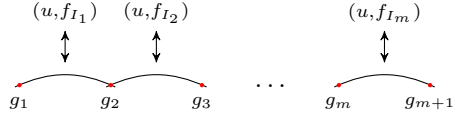
\[
\begin{tikzpicture}[xscale=0.6,yscale=0.6]
\draw[black] (-0.1,-0.04,0) to [bend left=25] (2.1,-0.04,0);
\draw[black] (1.9,-0.04,0) to [bend left=25] (4.1,-0.04,0);
\node at (5.5,0,0) {$\hdots$};
\draw[black] (6.9,-0.04,0) to [bend left=25] (9.1,-0.04,0);
\node at (1,1.6,0) {$\scriptstyle (u,f_{I_1})$};
\node at (2,1.6,0) {$\scriptstyle$};
\node at (3,1.6,0) {$\scriptstyle (u,f_{I_2})$};
\node at (4,1.6,0) {$\scriptstyle$};
\node at (7,1.6,0) {$\scriptstyle$};
\node at (8,1.6,0) {$\scriptstyle (u,f_{I_m})$};
\draw[<->] (1,1.2,0) -- (1,0.5,0);
\draw[<->] (3,1.2,0) -- (3,0.5,0);
\draw[<->] (8,1.2,0) -- (8,0.5,0);
\filldraw [red] (0,0,0) circle (1pt);
\filldraw [red] (2,0,0) circle (1pt);
\filldraw [red] (4,0,0) circle (1pt);
\filldraw [red] (7,0,0) circle (1pt);
\filldraw [red] (9,0,0) circle (1pt);
\node at (0,-0.4,0) {$\scriptstyle g_1$};
\node at (2,-0.4,0) {$\scriptstyle g_2$};
\node at (4,-0.4,0) {$\scriptstyle g_3$};
\node at (7,-0.4,0) {$\scriptstyle g_{m}$};
\node at (9,-0.4,0) {$\scriptstyle g_{m+1}$};
\end{tikzpicture} 
\]
This gives us a combinatorial model to visualize mutation.

\begin{remark}\label{remark that combinatorics has meaning}
The combinatorial model $\c{P}(\F)$ has geometric meaning when $k=\mathbb{C}$, since it is precisely the fibre above the origin of a certain partial crepant resolution, denoted $X^\F$ in \cite[\S5]{IW5}.  
In the more general case of an arbitrary field (i.e.\ in the setting above), we do not know whether the derived equivalence with a geometric space holds, but it turns out that the combinatorial model is still useful.
\end{remark}

\begin{example}\label{mutexample1}
Consider $f=f_1f_2f_3f_4f_5f_6$ with a flag $\c{F}=(\{ 2,3\}\subsetneq \{ 2,3,1 \})$. Then   $\c{F}$ corresponds to
\[
\begin{tikzpicture}[xscale=0.6,yscale=0.6]
\draw[black] (-0.1,-0.04,0) to [bend left=25] (2.1,-0.04,0);
\draw[black] (1.9,-0.04,0) to [bend left=25] (4.1,-0.04,0);
\filldraw [red] (0,0,0) circle (1pt);
\filldraw [red] (2,0,0) circle (1pt);
\filldraw [red] (4,0,0) circle (1pt);
\node at (0,-0.4,0) {$\scriptstyle f_2f_3$};
\node at (2,-0.4,0) {$\scriptstyle f_1$};
\node at (4,-0.4,0) {$\scriptstyle f_4f_5f_6$};
\end{tikzpicture}
\]
The corresponding $T^\c{F}$ is $R\oplus (u,f_2f_3)\oplus (u,f_1f_2f_3)$.
\end{example}

We are interested in mutations of non-free summands of $T^\c{F}$, so since above such summands correspond to subsets of the curves, pick an arbitrary $\emptyset\neq J\subseteq \{ 1,\hdots,m\}$.  Now write $J$ as a disjoint union of connected components:
\begin{defin}
A connected component of $J$ is a collection of consecutive numbers from $i_1$ to $i_2$ inside $\{ 1,\hdots, n\}$, each of which belongs to $J$, such that $i_1-1\notin J$ and $i_2+1\notin J$.  We write $J=\coprod_{j=1}^{t} J_{j}$ as a disjoint union of connected components.
\end{defin}
Geometrically, if say $m=6$ and $J:=\{ 2,3,5\}$ then we are simply bunching the curves corresponding to $J$ into connected components as in the following picture:
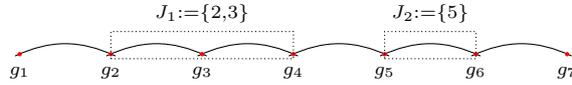
\begin{center}
\begin{tikzpicture}[xscale=0.6,yscale=0.6]
\draw[black] (-0.1,-0.04,0) to [bend left=25] (2.1,-0.04,0);
\draw[black] (1.9,-0.04,0) to [bend left=25] (4.1,-0.04,0);
\draw[black] (3.9,-0.04,0) to [bend left=25] (6.1,-0.04,0);
\draw[black] (5.9,-0.04,0) to [bend left=25] (8.1,-0.04,0);
\draw[black] (7.9,-0.04,0) to [bend left=25] (10.1,-0.04,0);
\draw[black] (9.9,-0.04,0) to [bend left=25] (12.1,-0.04,0);
\filldraw [red] (0,0,0) circle (1pt);
\filldraw [red] (2,0,0) circle (1pt);
\filldraw [red] (4,0,0) circle (1pt);
\filldraw [red] (6,0,0) circle (1pt);
\filldraw [red] (8,0,0) circle (1pt);
\filldraw [red] (10,0,0) circle (1pt);
\filldraw [red] (12,0,0) circle (1pt);
\node at (0,-0.4,0) {$\scriptstyle g_{1}$};
\node at (2,-0.4,0) {$\scriptstyle g_{2}$};
\node at (4,-0.4,0) {$\scriptstyle g_{3}$};
\node at (6,-0.4,0) {$\scriptstyle g_{4}$};
\node at (8,-0.4,0) {$\scriptstyle g_{5}$};
\node at (10,-0.4,0) {$\scriptstyle g_{6}$};
\node at (12,-0.4,0) {$\scriptstyle g_{7}$};
\draw [densely dotted] (2,-0.1,0) -- (6,-0.1,0) -- (6,0.5,0) -- (2,0.5,0) -- cycle;
\draw [densely dotted] (8,-0.1,0) -- (10,-0.1,0) -- (10,0.5,0) -- (8,0.5,0) -- cycle;
\node at (4,0.9,0) {$\scriptstyle J_{1}:=\{ {2}, {3} \}$};
\node at (9,0.9,0) {$\scriptstyle J_{2}:=\{ {5} \}$};
\end{tikzpicture}
\end{center}

The mutation operation acts on the set of modifying modules, hence on $T^\c{F}$, and hence on the set of $\c{P}(\c{F})$.  Below (see \ref{mainmutation}) we will justify the following intuitive geometric picture:
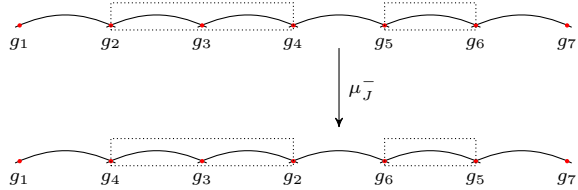
\begin{center}
\begin{tikzpicture}[xscale=0.6,yscale=0.6]
\draw[black] (-0.1,-0.04,0) to [bend left=25] (2.1,-0.04,0);
\draw[black] (1.9,-0.04,0) to [bend left=25] (4.1,-0.04,0);
\draw[black] (3.9,-0.04,0) to [bend left=25] (6.1,-0.04,0);
\draw[black] (5.9,-0.04,0) to [bend left=25] (8.1,-0.04,0);
\draw[black] (7.9,-0.04,0) to [bend left=25] (10.1,-0.04,0);
\draw[black] (9.9,-0.04,0) to [bend left=25] (12.1,-0.04,0);
\filldraw [red] (0,0,0) circle (1pt);
\filldraw [red] (2,0,0) circle (1pt);
\filldraw [red] (4,0,0) circle (1pt);
\filldraw [red] (6,0,0) circle (1pt);
\filldraw [red] (8,0,0) circle (1pt);
\filldraw [red] (10,0,0) circle (1pt);
\filldraw [red] (12,0,0) circle (1pt);
\node at (0,-0.4,0) {$\scriptstyle g_{1}$};
\node at (2,-0.4,0) {$\scriptstyle g_{2}$};
\node at (4,-0.4,0) {$\scriptstyle g_{3}$};
\node at (6,-0.4,0) {$\scriptstyle g_{4}$};
\node at (8,-0.4,0) {$\scriptstyle g_{5}$};
\node at (10,-0.4,0) {$\scriptstyle g_{6}$};
\node at (12,-0.4,0) {$\scriptstyle g_{7}$};
\draw [densely dotted] (2,-0.1,0) -- (6,-0.1,0) -- (6,0.5,0) -- (2,0.5,0) -- cycle;
\draw [densely dotted] (8,-0.1,0) -- (10,-0.1,0) -- (10,0.5,0) -- (8,0.5,0) -- cycle;
\draw[black] (-0.1,-3.04,0) to [bend left=25] (2.1,-3.04,0);
\draw[black] (1.9,-3.04,0) to [bend left=25] (4.1,-3.04,0);
\draw[black] (3.9,-3.04,0) to [bend left=25] (6.1,-3.04,0);
\draw[black] (5.9,-3.04,0) to [bend left=25] (8.1,-3.04,0);
\draw[black] (7.9,-3.04,0) to [bend left=25] (10.1,-3.04,0);
\draw[black] (9.9,-3.04,0) to [bend left=25] (12.1,-3.04,0);
\filldraw [red] (0,-3,0) circle (1pt);
\filldraw [red] (2,-3,0) circle (1pt);
\filldraw [red] (4,-3,0) circle (1pt);
\filldraw [red] (6,-3,0) circle (1pt);
\filldraw [red] (8,-3,0) circle (1pt);
\filldraw [red] (10,-3,0) circle (1pt);
\filldraw [red] (12,-3,0) circle (1pt);
\node at (0,-3.4,0) {$\scriptstyle g_{1}$};
\node at (2,-3.4,0) {$\scriptstyle g_{4}$};
\node at (4,-3.4,0) {$\scriptstyle g_{3}$};
\node at (6,-3.4,0) {$\scriptstyle g_{2}$};
\node at (8,-3.4,0) {$\scriptstyle g_{6}$};
\node at (10,-3.4,0) {$\scriptstyle g_{5}$};
\node at (12,-3.4,0) {$\scriptstyle g_{7}$};
\draw [densely dotted] (2,-3.1,0) -- (6,-3.1,0) -- (6,-2.5,0) -- (2,-2.5,0) -- cycle;
\draw [densely dotted] (8,-3.1,0) -- (10,-3.1,0) -- (10,-2.5,0) -- (8,-2.5,0) -- cycle;
\draw[->] (7,-0.5,0) -- node[right] {$\scriptstyle\mu^-_J$}(7,-2.25,0);
\end{tikzpicture} 
\end{center}
where each connected component of $J$ gets reflected. It is clear from this picture (and indeed we prove it in \ref{mainmutation}) that $\mu^-_{J}(M)=M$ if and only if $J$ is componentwise symmetric.  

\begin{defin}
For a given flag $\c{F}=(I_1,\hdots,I_m)$, associate the  combinatorial picture $\c{P}(\F)$ as above.  For $\emptyset\neq J\subseteq \{1,\hdots,m\}$, we define the \emph{$J$-reflection} of $\c{P}(\F)$ as follows:
the number of curves remains the same, but the new position of the $g$'s is obtained from the positions in $\c{P}(\F)$ by reflecting each connected component of $J$ in the vertical axis. 
\end{defin}

We now build up to \ref{mainmutation}.  To fix some convenient notation, we write $J=\coprod_{j=1}^{t}J_{j}$ and denote
\[
J_{j}=\{ {l_{j}},l_{j}+1,\hdots,u_{j}-1,{u_{j}}\}
\]
to be the connected components of $J$, where $l_{j}$ stands for lower bound and $u_{j}$ stands for upper bound.

\begin{lemma}
Fix flag $\c{F}=(I_1,\hdots,I_m)$, and choose $\emptyset\neq J\subseteq \{1,\hdots,m\}$.  Write $J=\coprod_{j=1}^{t}J_{j}$ and consider one of the summands $M_i$ of $T^\c{F}$ which lies in $J$.  Say $M_i$ lives in the component $J_j$, then the following sequence is exact
\begin{eqnarray}
0\to M_i\xrightarrow{(inc\,\, \prod\limits_{a=i+1}^{u_j+1}g_a)}M_{l_{j}-1}\oplus M_{u_{j}+1}\xrightarrow{{\prod\limits_{a=i+1}^{u_j+1}g_a}\choose -inc}(u,({\textstyle\prod\limits_{b=1}^{l_j-1}}g_b)({\textstyle\prod\limits_{a=i+1}^{u_j+1}}g_a))\to 0.\label{leftapprox}
\end{eqnarray}
\end{lemma}
\begin{proof}
This is a special case of \ref{exactness}.
\end{proof}

\begin{lemma}\label{rightapprox}
The dual short exact sequence of (\ref{leftapprox}), namely
\[
0\to {(u,({\textstyle\prod_{b=1}^{l_j-1}}g_b)({\textstyle\prod_{a=i+1}^{u_j+1}}g_a))}^{*}\to M_{l_{j}-1}^{*}\oplus M_{u_{j}+1}^{*}\to M_{i}^{*}\to 0,
\]
is a minimal right $\add\frac{M^{*}}{M_{J}^{*}}$-approximation of $M_{i}^{*}$.
\end{lemma}
\begin{proof}
This is a special case of \ref{approximation}.
\end{proof}

\begin{thm}\label{mainmutation}
Fix a flag $\c{F}=(I_1,\hdots,I_m)$, and associate to $\c{F}$ the module $T^\F$ and the combinatorial picture $\c{P}(\F)$, as before.  Choose $\emptyset\neq J\subseteq \{1,\hdots,m\}$,  then $\mu^-_J(T^\F)$ is the module corresponding to the $J$-reflection of $\c{P}(\F)$. 
\end{thm}
\begin{proof}
Consider the first connected component $J_1$. By \ref{rightapprox} we know that $M_{u_1}$ mutates to $(u,({\textstyle\prod_{b=1}^{l_1-1}}g_b)g_{u_1+1})^{**}$, which is isomorphic to 
$(u,({\textstyle\prod_{b=1}^{l_1-1}}g_b)g_{u_1+1})$.  Similarly, $M_{u_1-1}$ mutates to $(u,({\textstyle\prod_{b=1}^{l_1-1}}g_b)g_{u_1+1}g_{u_1})$. Continuing, we see that $M_{u_1-i}$ mutates to $(u,({\textstyle\prod_{b=1}^{l_1-1}}g_b)g_{u_1+1}\hdots g_{u_1-i+1})$ for all $1\leq i\leq u_1-l_1$. Since the combinatorial picture is built by ordering the summands in increasing lengths of products, we see that in the combinatorial picture, the component $J_1$ has been reflected.  The proof that the remaining components are reflected is identical.
\end{proof}

\begin{example}\label{mutexample2}
As in \ref{mutexample1}, consider $f=f_1f_2f_3f_4f_5f_6$ with flag $\c{F}=(\{ 2,3\}\subsetneq \{ 2,3,1 \})$. Then $T^\c{F}=R\oplus(u,f_2f_3)\oplus(u,f_1f_2f_3)$, which pictorially is
\[
\begin{tikzpicture}[xscale=0.6,yscale=0.6]
\draw[black] (-0.1,-0.04,0) to [bend left=25] (2.1,-0.04,0);
\draw[black] (1.9,-0.04,0) to [bend left=25] (4.1,-0.04,0);
\filldraw [red] (0,0,0) circle (1pt);
\filldraw [red] (2,0,0) circle (1pt);
\filldraw [red] (4,0,0) circle (1pt);
\node at (0,-0.4,0) {$\scriptstyle f_2f_3$};
\node at (2,-0.4,0) {$\scriptstyle f_1$};
\node at (4,-0.4,0) {$\scriptstyle f_4f_5f_6$};
\end{tikzpicture}
\]
Pick summand $(u,f_1f_2f_3)$, then the mutation is given by 
\[
\begin{tikzpicture}
\node at (0,0) {\begin{tikzpicture}[xscale=0.6,yscale=0.6]
\draw[black] (-0.1,-0.04,0) to [bend left=25] (2.1,-0.04,0);
\draw[black] (1.9,-0.04,0) to [bend left=25] (4.1,-0.04,0);
\filldraw [red] (0,0,0) circle (1pt);
\filldraw [red] (2,0,0) circle (1pt);
\filldraw [red] (4,0,0) circle (1pt);
\node at (0,-0.4,0) {$\scriptstyle f_2f_3$};
\node at (2,-0.4,0) {$\scriptstyle f_1$};
\node at (4,-0.4,0) {$\scriptstyle f_4f_5f_6$};
\draw [densely dotted] (2,-0.1,0) -- (4,-0.1,0) -- (4,0.5,0) -- (2,0.5,0) -- cycle;
\end{tikzpicture}};
\node at (5,0) {\begin{tikzpicture}[xscale=0.6,yscale=0.6]
\draw[black] (-0.1,-0.04,0) to [bend left=25] (2.1,-0.04,0);
\draw[black] (1.9,-0.04,0) to [bend left=25] (4.1,-0.04,0);
\filldraw [red] (0,0,0) circle (1pt);
\filldraw [red] (2,0,0) circle (1pt);
\filldraw [red] (4,0,0) circle (1pt);
\node at (0,-0.4,0) {$\scriptstyle f_2f_3$};
\node at (2,-0.4,0) {$\scriptstyle f_4f_5f_6$};
\node at (4,-0.4,0) {$\scriptstyle f_1$};
\draw [densely dotted] (2,-0.1,0) -- (4,-0.1,0) -- (4,0.5,0) -- (2,0.5,0) -- cycle;
\end{tikzpicture}};
\draw[->] (2,0) -- (3,0);
\end{tikzpicture}
\]
and so $\mu^-(T^\c{F})=R\oplus (u,f_2f_3)\oplus(u,f_2f_3f_4f_5f_6)$.
\end{example}

\medskip

We now calculate the quiver of $\End_R(T^\cF)$.  To ease notation, for a given flag $\F=(I_1,\hdots,I_m)$ we denote $g_1:=f_{I_1}$, set $g_j:=\tfrac{f_{I_j}}{f_{I_{j-1}}}$ for all $2\leq j\leq m$, and $g_{m+1}:=\tfrac{f}{f_{I_m}}$.
\begin{cor}\label{quiverOK}
Given a flag $\c{F}=(I_1,\hdots,I_m)$, with notation as above the quiver of $\End_R(T^\c{F})$ is as follows:
\[
\begin{tikzpicture}
\node at (0,-1.3) {$\scriptstyle m\geq 2$};
\node at (0,0)
{\begin{tikzpicture}[xscale=1.8,yscale=1.4,bend angle=13, looseness=1]
\node (1) at (1,0) {$\scriptstyle {T_{I_1}}$}; 
\node (2) at (2,0) {$\scriptstyle {T_{I_2}}$};
\node (4) at (3,0) {$\scriptstyle \cdots$};
\node (5) at (4,0) {$\scriptstyle {T_{I_m}}$};
\node (R) at (2.5,-1) {$\scriptstyle R$};
\draw [bend right,<-,pos=0.5] (1) to node[inner sep=0.5pt,fill=white,below=-3pt] {$\scriptstyle inc$} (2);
\draw [bend right,<-,pos=0.5] (2) to node[inner sep=0.5pt,fill=white,below=-3pt] {$\scriptstyle g_2$}(1);
\draw [bend right,<-,pos=0.5] (2) to node[inner sep=0.5pt,fill=white,below=-3pt] {$\scriptstyle inc$} (4);
\draw [bend right,<-,pos=0.5] (4) to node[inner sep=0.5pt,fill=white,below=-3pt] {$\scriptstyle g_3$}(2);
\draw [bend right,<-,pos=0.5] (4) to node[inner sep=0.5pt,fill=white,below=-3pt] {$\scriptstyle inc$} (5);
\draw [bend right,<-,pos=0.5] (5) to node[inner sep=0.5pt,fill=white,below=-3pt] {$\scriptstyle g_{m}$} (4);
\draw [bend right=7,<-,pos=0.5] ($(R)+(135:4pt)$) to node[inner sep=0.5pt,fill=white,below=-5pt] {$\scriptstyle inc$} ($(1)+(-45:6pt)$);
\draw [bend right=7,<-,pos=0.5] ($(1)+(-75:6pt)$) to node[inner sep=0.5pt,fill=white,below=-3pt] {$\scriptstyle g_1$}  ($(R)+(165:4pt)$);
\draw [bend right=7,<-,pos=0.5] ($(R)+(15:4pt)$) to node[inner sep=0.5pt,fill=white,below=-1pt] {$\scriptstyle \frac{g_{m{\scriptstyle +}1}}{u}$} ($(5)+(-100:5pt)$);
\draw [bend right=7,<-,pos=0.5] ($(5)+(-125:5pt)$) to node[inner sep=0.5pt,fill=white,below=-3pt] {$\scriptstyle u$} ($(R)+(45:4pt)$);
\end{tikzpicture}};
\node at (6,0) {\begin{tikzpicture} 
\node (C1) at (0,0)  {$\scriptstyle R$};
\node (C1a) at (-0.1,0.05)  {};
\node (C1b) at (-0.1,-0.05)  {};
\node (C2) at (1.75,0)  {$\scriptstyle T_{I_1}$};
\node (C2a) at (1.85,0.05) {};
\node (C2b) at (1.85,-0.05) {};
\draw [->,bend left=45,looseness=1,pos=0.5] (C1) to node[inner sep=0.5pt,fill=white]  {$\scriptstyle g_1$} (C2);
\draw [->,bend left=20,looseness=1,pos=0.5] (C1) to node[inner sep=0.5pt,fill=white]  {$\scriptstyle u$} (C2);
\draw [->,bend left=45,looseness=1,pos=0.5] (C2) to node[inner sep=0.5pt,fill=white]  {$\scriptstyle \frac{g_2}{u}$} (C1);
\draw [->,bend left=20,looseness=1,pos=0.5] (C2) to node[inner sep=0.5pt,fill=white,below=-5pt] {$\scriptstyle inc$} (C1);
\end{tikzpicture}};
\node at (6,-1.3) {$\scriptstyle  m=1$};
\end{tikzpicture}
\] 
together with the possible addition of some loops, given by the following rules:
\begin{itemize}
\item Consider vertex $R$.  If $(g_1,g_{m+1})=(x,y)$ in the ring $\k[[x,y]]$, add no loops at vertex $R$. Hence suppose $(g_1,g_{m+1})\subsetneq (x,y)$. If there exists $t\in (x,y)$ such that $(g_1,g_{m+1},t)=(x,y)$, add a loop labelled $t$ at vertex $R$.  If there exists no such $t$, add two loops labelled $x$ and $y$ at vertex $R$.
\item Consider vertex $T_{I_i}$.  If $(g_i,g_{i+1})=(x,y)$ in the ring $\k[[x,y]]$, add no loops at vertex $T_{I_i}$.  Hence suppose $(g_i,g_{i+1})\subsetneq (x,y)$. If there exists $t\in (x,y)$ such that $(g_i,g_{i+1},t)=(x,y)$, add a loop labelled $t$ at vertex $T_{I_i}$.  If there exists no such $t$, add two loops labelled $x$ and $y$ at vertex $T_{I_i}$.
\end{itemize}
\end{cor}
\begin{proof}
(1) Since $\Hom_R(R,R)\cong \Hom_R(T_{I_i},T_{I_i})\cong R$, we first must verify that at each vertex we can see the elements $u,v,x,y$ as cycles at that vertex.  Certainly $u$ is there (as it is the path followed anticlockwise around the circle), and certainly $v$ is there (being the path followed clockwise around the circle).  It is possible to see cycles $x$ and $y$ at every vertex by the rules for loops. \\
(2) Since $\Hom_R(R,T_{I_i})\cong T_{I_i}$, we must verify that we can see the generators of $T_{I_i}$ as paths from vertex $R$ to vertex $T_{I_i}$. But this is clear, as $\prod_{j=1}^ig_j$ is the clockwise path, and $u$ is the anticlockwise path.\\
(3) As a module over the centre, the paths from vertex $T_{I_i}$ to vertex $R$ are generated by two paths from vertex $T_{I_i}$ to $R$, namely the one clockwise, which is $\frac{f}{uf_I}$, and the one anticlockwise, which is inclusion.  Since $\Hom_R(T_{I_i},R)\cong (u,\frac{f}{f_I})$ by \ref{describe Hom-sets}, clearly this is isomorphic to paths from  $T_{I_i}$ to $R$.\\
(4) The argument for paths from $T_{I_i}$ to $T_{I_j}$ is identical to the argument in (3), using the isomorphism in \ref{describe Hom-sets}.
\end{proof}

\begin{remark} When $k$ is algebraically closed of characteristic zero, we remark that there are at least two other methods for computing the quiver of $\End_R(T^\F)$.  
One way would be to use reconstruction on 1-dimensional fibres \`a la $\GL(2,\C{})$ McKay Correspondence \cite{Wem}.  Another is to compute the quiver of $\End_R(T^\F)$ in the one-dimensional setting, as in \cite[4.10]{BIKR}, and then use Kn\"orrer periodicity.  
This last method only gives the quiver of the stable endomorphism algebra, so more work would be needed.
\end{remark}

\subsection{Geometric Corollaries}\label{Geo cor} To apply \S\ref{mutation} to geometry, in this section we revert to the assumption that $k=\mathbb{C}$.  As remarked before in \ref{remark that combinatorics has meaning}, given a flag $\F=(I_1,\hdots,I_m)$ we can associate a scheme $X^\F$ that gives a partial crepant resolution of $\Spec R$.   The procedure is described in \cite[\S5.1]{IW5}: first blowup the ideal $(u,f_{I_1})$ on $\Spec R$ to obtain a space denoted $X^{\cF_1}$.  Then on $X^{\cF_1}$ blowup the ideal $(u,f_{I_2\backslash I_1})$ to obtain a space $X^{\cF_2}$. On $X^{\cF_2}$ blowup the ideal $(u,f_{I_3\backslash I_2})$ to obtain a space $X^{\cF_3}$, and so on. Continuing in this fashion we obtain a chain of projective birational morphisms
\[
X^{\cF_m}\to X^{\cF_{m-1}}\to\hdots\to X^{\cF_1}\to\Spec R
\]
and we define $X^{\cF}:=X^{\cF_m}$.  The following was shown in \cite[5.2]{IW5}.
\begin{thm}\label{db from previous}
$X^{\cF}$ is derived equivalent to $\End_R(T^\cF)$, and the fibre above the origin of the composition $X^{\cF}\to\Spec R$ can be represented by the picture
\begin{eqnarray}
\begin{array}{c}
\begin{tikzpicture}[xscale=0.6,yscale=0.6]
\draw[black] (-0.1,-0.04,0) to [bend left=25] (2.1,-0.04,0);
\draw[black] (1.9,-0.04,0) to [bend left=25] (4.1,-0.04,0);
\draw[black] (3.9,-0.04,0) to [bend left=25] (6.1,-0.04,0);
\draw[black] (7.9,-0.04,0) to [bend left=25] (10.1,-0.04,0);
\draw[black] (9.9,-0.04,0) to [bend left=25] (12.1,-0.04,0);
\filldraw [red] (0,0,0) circle (1pt);
\filldraw [red] (2,0,0) circle (1pt);
\filldraw [red] (4,0,0) circle (1pt);
\filldraw [red] (6,0,0) circle (1pt);
\filldraw [red] (8,0,0) circle (1pt);
\filldraw [red] (10,0,0) circle (1pt);
\filldraw [red] (12,0,0) circle (1pt);
\node at (0,-0.4,0) {$\scriptstyle g_{1}$};
\node at (2,-0.4,0) {$\scriptstyle g_{2}$};
\node at (4,-0.4,0) {$\scriptstyle g_{3}$};
\node at (7,0,0) {$\scriptstyle \hdots$};
\node at (10,-0.4,0) {$\scriptstyle g_{m}$};
\node at (12,-0.4,0) {$\scriptstyle g_{m+1}$};
\end{tikzpicture}
\end{array}\label{curve picture}
\end{eqnarray}
where $g_j:=f_{I_j\backslash I_{j-1}}$ for all $1\le j\le m+1$, where by convention  $I_0:=\emptyset$ and $I_{m+1}:=\{1,2,\hdots,n\}$.  The red dots in the above picture represent the possible points where the scheme $X^{\cF}$ is singular, where the red dot marked $g_i$ is a point which complete locally is given by $\mathbb{C}[[x,y,u,v]]/(g_i-uv)$.
\end{thm}

Combining this information together with \ref{welldefined} and \ref{mainmutation}, we can now produce derived equivalences between partial crepant resolutions of $\Spec R$, and also produce derived autoequivalences.

\begin{thm}
Every collection of curves above the origin in the partial resolution $X^{\cF}\to\Spec R$ determines a derived autoequivalence of $X^{\cF}$.
\end{thm}
\begin{proof}
Pick a collection of curves $\{ C_j\mid j\in J\}$.  For simplicity, we give the proof for the case $|J|=1$, but the general situation is the same.  Pick a curve $C$ in $X^{\cF}$, so locally the fibre above the origin \eqref{curve picture} is
\begin{center}
\begin{tikzpicture}[xscale=0.6,yscale=0.6]
\draw[black] (1.9,-0.04,0) to [bend left=25] (4.1,-0.04,0);
\draw[black] (3.9,-0.04,0) to [bend left=25] (6.1,-0.04,0);
\draw[black] (5.9,-0.04,0) to [bend left=25] (8.1,-0.04,0);
\filldraw [red] (2,0,0) circle (1pt);
\filldraw [red] (4,0,0) circle (1pt);
\filldraw [red] (6,0,0) circle (1pt);
\filldraw [red] (8,0,0) circle (1pt);
\node at (2,-0.4,0) {$\scriptstyle g_{i-1}$};
\node at (4,-0.4,0) {$\scriptstyle g_{i}$};
\node at (6,-0.4,0) {$\scriptstyle g_{i+1}$};
\node at (8,-0.4,0) {$\scriptstyle g_{i+2}$};
\draw [densely dotted] (4,-0.1,0) -- (6,-0.1,0) -- (6,0.5,0) -- (4,0.5,0) -- cycle;
\node at (1.5,0,0) {$\scriptstyle \hdots$};
\node at (8.5,0,0) {$\scriptstyle \hdots$};
\end{tikzpicture}
\end{center}
Now $X^{\cF}$ is derived equivalent to $\End_R(T^{\cF})$ by \ref{db from previous}, and by \ref{mainmutation}, mutating the summand $(u,f_{I_i})$ 
\[
\begin{array}{c}
\begin{tikzpicture}[xscale=0.6,yscale=0.6]
\draw[black] (1.9,-0.04,0) to [bend left=25] (4.1,-0.04,0);
\draw[black] (3.9,-0.04,0) to [bend left=25] (6.1,-0.04,0);
\draw[black] (5.9,-0.04,0) to [bend left=25] (8.1,-0.04,0);
\filldraw [red] (2,0,0) circle (1pt);
\filldraw [red] (4,0,0) circle (1pt);
\filldraw [red] (6,0,0) circle (1pt);
\filldraw [red] (8,0,0) circle (1pt);
\node at (2,-0.4,0) {$\scriptstyle g_{i-1}$};
\node at (4,-0.4,0) {$\scriptstyle g_{i}$};
\node at (6,-0.4,0) {$\scriptstyle g_{i+1}$};
\node at (8,-0.4,0) {$\scriptstyle g_{i+2}$};
\draw [densely dotted] (4,-0.1,0) -- (6,-0.1,0) -- (6,0.5,0) -- (4,0.5,0) -- cycle;
\node at (1.5,0,0) {$\scriptstyle \hdots$};
\node at (8.5,0,0) {$\scriptstyle \hdots$};
\draw[->] (9.5,0,0) -- (10.5,0,0);
\draw[black] (11.9,-0.04,0) to [bend left=25] (14.1,-0.04,0);
\draw[black] (13.9,-0.04,0) to [bend left=25] (16.1,-0.04,0);
\draw[black] (15.9,-0.04,0) to [bend left=25] (18.1,-0.04,0);
\filldraw [red] (12,0,0) circle (1pt);
\filldraw [red] (14,0,0) circle (1pt);
\filldraw [red] (16,0,0) circle (1pt);
\filldraw [red] (18,0,0) circle (1pt);
\node at (12,-0.4,0) {$\scriptstyle g_{i-1}$};
\node at (14,-0.4,0) {$\scriptstyle g_{i+1}$};
\node at (16,-0.4,0) {$\scriptstyle g_{i}$};
\node at (18,-0.4,0) {$\scriptstyle g_{i+2}$};
\draw [densely dotted] (14,-0.1,0) -- (16,-0.1,0) -- (16,0.5,0) -- (14,0.5,0) -- cycle;
\node at (11.5,0,0) {$\scriptstyle \hdots$};
\node at (18.5,0,0) {$\scriptstyle \hdots$};
\end{tikzpicture}
\end{array}
\] 
gives us a derived equivalence between $\End_R(T^\cF)$ and $\End_R(T^\cG)$, where $\cG$ is the flag associated with the reflection. Since $X^{\cG}$ is derived equivalent to $\End_R(T^\cG)$  by \ref{db from previous}, composing it follows that $X^{\cF}$ is derived equivalent to $X^{\cG}$.  If $g_i=g_{i+1}$ then $X^\cF=X^\cG$ and so this is our derived autoequivalence of $X^\cF$.  Otherwise $g_i\neq g_{i+1}$, and so in this case reflecting again
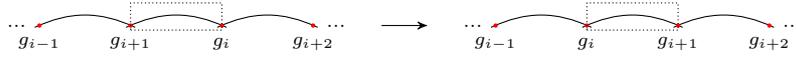
\[
\begin{array}{c}
\begin{tikzpicture}[xscale=0.6,yscale=0.6]
\draw[black] (1.9,-0.04,0) to [bend left=25] (4.1,-0.04,0);
\draw[black] (3.9,-0.04,0) to [bend left=25] (6.1,-0.04,0);
\draw[black] (5.9,-0.04,0) to [bend left=25] (8.1,-0.04,0);
\filldraw [red] (2,0,0) circle (1pt);
\filldraw [red] (4,0,0) circle (1pt);
\filldraw [red] (6,0,0) circle (1pt);
\filldraw [red] (8,0,0) circle (1pt);
\node at (2,-0.4,0) {$\scriptstyle g_{i-1}$};
\node at (4,-0.4,0) {$\scriptstyle g_{i+1}$};
\node at (6,-0.4,0) {$\scriptstyle g_{i}$};
\node at (8,-0.4,0) {$\scriptstyle g_{i+2}$};
\draw [densely dotted] (4,-0.1,0) -- (6,-0.1,0) -- (6,0.5,0) -- (4,0.5,0) -- cycle;
\node at (1.5,0,0) {$\scriptstyle \hdots$};
\node at (8.5,0,0) {$\scriptstyle \hdots$};
\draw[->] (9.5,0,0) -- (10.5,0,0);
\draw[black] (11.9,-0.04,0) to [bend left=25] (14.1,-0.04,0);
\draw[black] (13.9,-0.04,0) to [bend left=25] (16.1,-0.04,0);
\draw[black] (15.9,-0.04,0) to [bend left=25] (18.1,-0.04,0);
\filldraw [red] (12,0,0) circle (1pt);
\filldraw [red] (14,0,0) circle (1pt);
\filldraw [red] (16,0,0) circle (1pt);
\filldraw [red] (18,0,0) circle (1pt);
\node at (12,-0.4,0) {$\scriptstyle g_{i-1}$};
\node at (14,-0.4,0) {$\scriptstyle g_{i}$};
\node at (16,-0.4,0) {$\scriptstyle g_{i+1}$};
\node at (18,-0.4,0) {$\scriptstyle g_{i+2}$};
\draw [densely dotted] (14,-0.1,0) -- (16,-0.1,0) -- (16,0.5,0) -- (14,0.5,0) -- cycle;
\node at (11.5,0,0) {$\scriptstyle \hdots$};
\node at (18.5,0,0) {$\scriptstyle \hdots$};
\end{tikzpicture}
\end{array}
\] 
gives a derived equivalence between $X^{\cG}$ and $X^{\cF}$.  Composing the chain of equivalences $\Db(\coh X^{\cF})\to\Db(\coh X^{\cG})\to\Db(\coh X^{\cF})$ is then our desired autoequivalence.
\end{proof}

Now we note that even if two partial crepant resolutions of $\Spec R$ both contain the same number of curves above the origin, they need not be derived equivalent:

\begin{example}\label{notallpartialDb}
In the case $f=xxy$, the crepant partial resolutions with one curve are
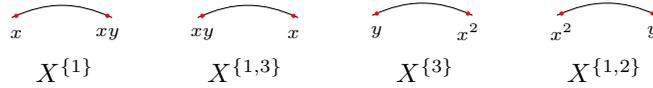
\[
\begin{array}{ccccccc}
\begin{tikzpicture}[xscale=0.6,yscale=0.6]
\draw[black] (-0.1,-0.04,0) to [bend left=25] (2.1,-0.04,0);
\filldraw [red] (0,0,0) circle (1pt);
\filldraw [red] (2,0,0) circle (1pt);
\node at (0,-0.4,0) {$\scriptstyle x$};
\node at (2,-0.4,0) {$\scriptstyle xy$};
\end{tikzpicture} &&
\begin{tikzpicture}[xscale=0.6,yscale=0.6]
\draw[black] (-0.1,-0.04,0) to [bend left=25] (2.1,-0.04,0);
\filldraw [red] (0,0,0) circle (1pt);
\filldraw [red] (2,0,0) circle (1pt);
\node at (0,-0.4,0) {$\scriptstyle xy$};
\node at (2,-0.4,0) {$\scriptstyle x$};
\end{tikzpicture}&&
\begin{tikzpicture}[xscale=0.6,yscale=0.6]
\draw[black] (-0.1,-0.04,0) to [bend left=25] (2.1,-0.04,0);
\filldraw [red] (0,0,0) circle (1pt);
\filldraw [red] (2,0,0) circle (1pt);
\node at (0,-0.4,0) {$\scriptstyle y$};
\node at (2,-0.4,0) {$\scriptstyle x^2$};
\end{tikzpicture} &&
\begin{tikzpicture}[xscale=0.6,yscale=0.6]
\draw[black] (-0.1,-0.04,0) to [bend left=25] (2.1,-0.04,0);
\filldraw [red] (0,0,0) circle (1pt);
\filldraw [red] (2,0,0) circle (1pt);
\node at (0,-0.4,0) {$\scriptstyle x^2$};
\node at (2,-0.4,0) {$\scriptstyle y$};
\end{tikzpicture}\\
X^{\{1\}}&&X^{\{1,3\}} && X^{\{3\}} &&X^{\{1,2\}}
\end{array}
\]
The two spaces $X^{\{1\}}$ and $X^{\{1,3\}}$ are derived equivalent via mutation, and the two spaces $X^{\{3\}}$ and $X^{\{1,2\}}$ are also derived equivalent via mutation.  However they are not all derived equivalent, since if  $\Db(\coh X^{\{1\}})\approx\Db(\coh X^{\{3\}})$ then $\Dsg(X^{\{1\}})\approx\Dsg(X^{\{3\}})$, which by \cite[3.7(2)]{IW5} gives $\uCM \mathbb{C}[[u,v,x,y]]/(uv-x^2)\approx \uCM \mathbb{C}[[u,v,x,y]]/(uv-xy)$. But this is impossible, since (for example) the left hand side has infinite dimensional Hom-spaces, whereas in the right hand side all Hom spaces are finite dimensional.
\end{example}

The following is immediate from the theory of mutation:
\begin{cor}\label{permuteDbresult}
In the case $k=\mathbb{C}$, suppose that $\F=(I_1,\hdots,I_m)$ and $\G=(J_1,\hdots,J_n)$ are flags. Then $X^\F$ and $X^\G$ are derived equivalent if $n=m$ and the singularities of $X^\F$ can be permuted to the singularities of $X^\G$.  
\end{cor}
\begin{proof}
Since $\Db(X^\F)\simeq\Db(\End_R(T^\F))$ and $\Db(X^\G)\simeq\Db(\End_R(T^\G))$, the result follows if we establish that $\End_R(T^\F)$ and $\End_R(T^\G)$ are derived equivalent.  Now mutation always gives derived equivalences (\ref{welldefined}), and mutation corresponds to permuting the order of the singularities (\ref{mainmutation}).  Hence if the singularities of $X^\F$ can be permuted to the singularities of $X^\G$, certainly there is a finite number of mutations which transforms $\c{P}(\F)$ into $\c{P}(\G)$, hence $\End_R(T^\F)$ and $\End_R(T^\G)$ are derived equivalent.
\end{proof}

\begin{example}\label{EG geom}
Let $f=f_1^2f_3^3=f_1f_1f_3f_3f_3$.
Then the exchange graph (removing loops) for maximal modifying generators is the following, which has already been observed in \ref{EG alg}
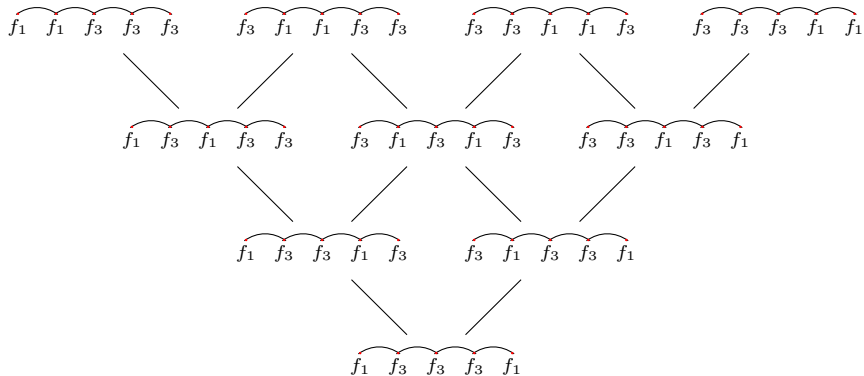
\begin{center}
\begin{tikzpicture}
\node (a1) at (0,0) 
{\begin{tikzpicture}[xscale=0.25,yscale=0.25]
\draw[black] (-0.1,-0.04,0) to [bend left=35] (2.1,-0.04,0);
\draw[black] (1.9,-0.04,0) to [bend left=35] (4.1,-0.04,0);
\draw[black] (3.9,-0.04,0) to [bend left=35] (6.1,-0.04,0);
\draw[black] (5.9,-0.04,0) to [bend left=35] (8.1,-0.04,0);
\filldraw [red] (0,0,0) circle (1pt);
\filldraw [red] (2,0,0) circle (1pt);
\filldraw [red] (4,0,0) circle (1pt);
\filldraw [red] (6,0,0) circle (1pt);
\filldraw [red] (8,0,0) circle (1pt);
\node at (0,-0.7,0) {$\scriptstyle f_1$};
\node at (2,-0.7,0) {$\scriptstyle f_1$};
\node at (4,-0.7,0) {$\scriptstyle f_3$};
\node at (6,-0.7,0) {$\scriptstyle f_3$};
\node at (8,-0.7,0) {$\scriptstyle f_3$};
\end{tikzpicture}}; 

\node (a2) at (3,0) 
{\begin{tikzpicture}[xscale=0.25,yscale=0.25]
\draw[black] (-0.1,-0.04,0) to [bend left=35] (2.1,-0.04,0);
\draw[black] (1.9,-0.04,0) to [bend left=35] (4.1,-0.04,0);
\draw[black] (3.9,-0.04,0) to [bend left=35] (6.1,-0.04,0);
\draw[black] (5.9,-0.04,0) to [bend left=35] (8.1,-0.04,0);
\filldraw [red] (0,0,0) circle (1pt);
\filldraw [red] (2,0,0) circle (1pt);
\filldraw [red] (4,0,0) circle (1pt);
\filldraw [red] (6,0,0) circle (1pt);
\filldraw [red] (8,0,0) circle (1pt);
\node at (0,-0.7,0) {$\scriptstyle f_3$};
\node at (2,-0.7,0) {$\scriptstyle f_1$};
\node at (4,-0.7,0) {$\scriptstyle f_1$};
\node at (6,-0.7,0) {$\scriptstyle f_3$};
\node at (8,-0.7,0) {$\scriptstyle f_3$};
\end{tikzpicture}}; 

\node (a3) at (6,0) 
{\begin{tikzpicture}[xscale=0.25,yscale=0.25]
\draw[black] (-0.1,-0.04,0) to [bend left=35] (2.1,-0.04,0);
\draw[black] (1.9,-0.04,0) to [bend left=35] (4.1,-0.04,0);
\draw[black] (3.9,-0.04,0) to [bend left=35] (6.1,-0.04,0);
\draw[black] (5.9,-0.04,0) to [bend left=35] (8.1,-0.04,0);
\filldraw [red] (0,0,0) circle (1pt);
\filldraw [red] (2,0,0) circle (1pt);
\filldraw [red] (4,0,0) circle (1pt);
\filldraw [red] (6,0,0) circle (1pt);
\filldraw [red] (8,0,0) circle (1pt);
\node at (0,-0.7,0) {$\scriptstyle f_3$};
\node at (2,-0.7,0) {$\scriptstyle f_3$};
\node at (4,-0.7,0) {$\scriptstyle f_1$};
\node at (6,-0.7,0) {$\scriptstyle f_1$};
\node at (8,-0.7,0) {$\scriptstyle f_3$};
\end{tikzpicture}}; 

\node (a4) at (9,0) 
{\begin{tikzpicture}[xscale=0.25,yscale=0.25]
\draw[black] (-0.1,-0.04,0) to [bend left=35] (2.1,-0.04,0);
\draw[black] (1.9,-0.04,0) to [bend left=35] (4.1,-0.04,0);
\draw[black] (3.9,-0.04,0) to [bend left=35] (6.1,-0.04,0);
\draw[black] (5.9,-0.04,0) to [bend left=35] (8.1,-0.04,0);
\filldraw [red] (0,0,0) circle (1pt);
\filldraw [red] (2,0,0) circle (1pt);
\filldraw [red] (4,0,0) circle (1pt);
\filldraw [red] (6,0,0) circle (1pt);
\filldraw [red] (8,0,0) circle (1pt);
\node at (0,-0.7,0) {$\scriptstyle f_3$};
\node at (2,-0.7,0) {$\scriptstyle f_3$};
\node at (4,-0.7,0) {$\scriptstyle f_3$};
\node at (6,-0.7,0) {$\scriptstyle f_1$};
\node at (8,-0.7,0) {$\scriptstyle f_1$};
\end{tikzpicture}}; 

\node (b1) at (1.5,-1.5) 
{\begin{tikzpicture}[xscale=0.25,yscale=0.25]
\draw[black] (-0.1,-0.04,0) to [bend left=35] (2.1,-0.04,0);
\draw[black] (1.9,-0.04,0) to [bend left=35] (4.1,-0.04,0);
\draw[black] (3.9,-0.04,0) to [bend left=35] (6.1,-0.04,0);
\draw[black] (5.9,-0.04,0) to [bend left=35] (8.1,-0.04,0);
\filldraw [red] (0,0,0) circle (1pt);
\filldraw [red] (2,0,0) circle (1pt);
\filldraw [red] (4,0,0) circle (1pt);
\filldraw [red] (6,0,0) circle (1pt);
\filldraw [red] (8,0,0) circle (1pt);
\node at (0,-0.7,0) {$\scriptstyle f_1$};
\node at (2,-0.7,0) {$\scriptstyle f_3$};
\node at (4,-0.7,0) {$\scriptstyle f_1$};
\node at (6,-0.7,0) {$\scriptstyle f_3$};
\node at (8,-0.7,0) {$\scriptstyle f_3$};
\end{tikzpicture}}; 

\node (b2) at (4.5,-1.5) 
{\begin{tikzpicture}[xscale=0.25,yscale=0.25]
\draw[black] (-0.1,-0.04,0) to [bend left=35] (2.1,-0.04,0);
\draw[black] (1.9,-0.04,0) to [bend left=35] (4.1,-0.04,0);
\draw[black] (3.9,-0.04,0) to [bend left=35] (6.1,-0.04,0);
\draw[black] (5.9,-0.04,0) to [bend left=35] (8.1,-0.04,0);
\filldraw [red] (0,0,0) circle (1pt);
\filldraw [red] (2,0,0) circle (1pt);
\filldraw [red] (4,0,0) circle (1pt);
\filldraw [red] (6,0,0) circle (1pt);
\filldraw [red] (8,0,0) circle (1pt);
\node at (0,-0.7,0) {$\scriptstyle f_3$};
\node at (2,-0.7,0) {$\scriptstyle f_1$};
\node at (4,-0.7,0) {$\scriptstyle f_3$};
\node at (6,-0.7,0) {$\scriptstyle f_1$};
\node at (8,-0.7,0) {$\scriptstyle f_3$};
\end{tikzpicture}}; 

\node (b3) at (7.5,-1.5) 
{\begin{tikzpicture}[xscale=0.25,yscale=0.25]
\draw[black] (-0.1,-0.04,0) to [bend left=35] (2.1,-0.04,0);
\draw[black] (1.9,-0.04,0) to [bend left=35] (4.1,-0.04,0);
\draw[black] (3.9,-0.04,0) to [bend left=35] (6.1,-0.04,0);
\draw[black] (5.9,-0.04,0) to [bend left=35] (8.1,-0.04,0);
\filldraw [red] (0,0,0) circle (1pt);
\filldraw [red] (2,0,0) circle (1pt);
\filldraw [red] (4,0,0) circle (1pt);
\filldraw [red] (6,0,0) circle (1pt);
\filldraw [red] (8,0,0) circle (1pt);
\node at (0,-0.7,0) {$\scriptstyle f_3$};
\node at (2,-0.7,0) {$\scriptstyle f_3$};
\node at (4,-0.7,0) {$\scriptstyle f_1$};
\node at (6,-0.7,0) {$\scriptstyle f_3$};
\node at (8,-0.7,0) {$\scriptstyle f_1$};
\end{tikzpicture}}; 

\node (c1) at (3,-3) 
{\begin{tikzpicture}[xscale=0.25,yscale=0.25]
\draw[black] (-0.1,-0.04,0) to [bend left=35] (2.1,-0.04,0);
\draw[black] (1.9,-0.04,0) to [bend left=35] (4.1,-0.04,0);
\draw[black] (3.9,-0.04,0) to [bend left=35] (6.1,-0.04,0);
\draw[black] (5.9,-0.04,0) to [bend left=35] (8.1,-0.04,0);
\filldraw [red] (0,0,0) circle (1pt);
\filldraw [red] (2,0,0) circle (1pt);
\filldraw [red] (4,0,0) circle (1pt);
\filldraw [red] (6,0,0) circle (1pt);
\filldraw [red] (8,0,0) circle (1pt);
\node at (0,-0.7,0) {$\scriptstyle f_1$};
\node at (2,-0.7,0) {$\scriptstyle f_3$};
\node at (4,-0.7,0) {$\scriptstyle f_3$};
\node at (6,-0.7,0) {$\scriptstyle f_1$};
\node at (8,-0.7,0) {$\scriptstyle f_3$};
\end{tikzpicture}}; 

\node (c2) at (6,-3) 
{\begin{tikzpicture}[xscale=0.25,yscale=0.25]
\draw[black] (-0.1,-0.04,0) to [bend left=35] (2.1,-0.04,0);
\draw[black] (1.9,-0.04,0) to [bend left=35] (4.1,-0.04,0);
\draw[black] (3.9,-0.04,0) to [bend left=35] (6.1,-0.04,0);
\draw[black] (5.9,-0.04,0) to [bend left=35] (8.1,-0.04,0);
\filldraw [red] (0,0,0) circle (1pt);
\filldraw [red] (2,0,0) circle (1pt);
\filldraw [red] (4,0,0) circle (1pt);
\filldraw [red] (6,0,0) circle (1pt);
\filldraw [red] (8,0,0) circle (1pt);
\node at (0,-0.7,0) {$\scriptstyle f_3$};
\node at (2,-0.7,0) {$\scriptstyle f_1$};
\node at (4,-0.7,0) {$\scriptstyle f_3$};
\node at (6,-0.7,0) {$\scriptstyle f_3$};
\node at (8,-0.7,0) {$\scriptstyle f_1$};
\end{tikzpicture}}; 

\node (d1) at (4.5,-4.5) 
{\begin{tikzpicture}[xscale=0.25,yscale=0.25]
\draw[black] (-0.1,-0.04,0) to [bend left=35] (2.1,-0.04,0);
\draw[black] (1.9,-0.04,0) to [bend left=35] (4.1,-0.04,0);
\draw[black] (3.9,-0.04,0) to [bend left=35] (6.1,-0.04,0);
\draw[black] (5.9,-0.04,0) to [bend left=35] (8.1,-0.04,0);
\filldraw [red] (0,0,0) circle (1pt);
\filldraw [red] (2,0,0) circle (1pt);
\filldraw [red] (4,0,0) circle (1pt);
\filldraw [red] (6,0,0) circle (1pt);
\filldraw [red] (8,0,0) circle (1pt);
\node at (0,-0.7,0) {$\scriptstyle f_1$};
\node at (2,-0.7,0) {$\scriptstyle f_3$};
\node at (4,-0.7,0) {$\scriptstyle f_3$};
\node at (6,-0.7,0) {$\scriptstyle f_3$};
\node at (8,-0.7,0) {$\scriptstyle f_1$};
\end{tikzpicture}}; 

\draw (a1) -- (b1);
\draw (a2) -- (b1);
\draw (a2) -- (b2);
\draw (a3) -- (b2);
\draw (a3) -- (b3);
\draw (a4) -- (b3);
\draw (b1) -- (c1);
\draw (b2) -- (c1);
\draw (b2) -- (c2);
\draw (b3) -- (c2);
\draw (c1) -- (d1);
\draw (c2) -- (d1);
\end{tikzpicture}
\end{center}
where for clarity we have illustrated only those that are connected via a mutation by an indecomposable summand.  If we include mutations by more that one summand, there are many more connecting lines.
\end{example}
\begin{remark}
Geometers will recognize the above picture, since when $k=\mathbb{C}$ it corresponds exactly to the flops of (single) curves on the $\mathds{Q}$-factorial terminalizations of $\Spec R$.  This is related to \S\ref{Geo cor}, but we will address this problem in more detail in future work.
\end{remark}

By \ref{notallpartialDb} and \ref{permuteDbresult}, for general partial resolutions it is clear that very rarely will $T^{\F_1}$ and $T^{\F_2}$ be linked by mutations. However, if both $T^{\F_1}$ and $T^{\F_2}$ have the maximal number of summands (i.e.\ $\c{P}(\F_1)$ and $\c{P}(\F_2)$ have the maximal number of curves), the homological algebra is much better behaved.  We already know that the exchange graph for MM generators is connected (\S\ref{MM and CT}).  Combining this with \ref{permuteDbresult} in the case $k=\mathbb{C}$ proves \ref{AllQfactDb_intro} in the introduction, namely:

\begin{cor}
All $\mathds{Q}$-factorial terminalizations of $\Spec R$ are derived equivalent.
\end{cor}

\end{document}